\documentclass[11pt,reqno,oneside]{amsart}
\usepackage[left=3cm,right=3cm,top=4cm,bottom=3cm]{geometry}
\usepackage{amsmath,amsfonts,amsthm,amssymb,color}
\usepackage[T1]{fontenc}
\usepackage[dvipsnames]{xcolor}
\usepackage{pdfsync}
\usepackage[citecolor=blue,colorlinks=true]{hyperref}
\usepackage{mathrsfs}
\usepackage{upref}
\usepackage{enumitem}
\usepackage{csquotes}

\numberwithin{equation}{section}

\newcommand{\id}{\mbox{Id}}
\newcommand{\1}{{\bf 1}}

\newcommand{\rmk}[1]{#1}

\newcommand{\dif}{\mathrm{d}}
\newcommand{\dd}{\mathrm{d}}
\newcommand{\mathd}{\mathrm{d}}
\newcommand{\supp}{\mathrm{supp}}
\newcommand{\ind}{\mathbf{1}}

\newcommand{\varep}{\varepsilon}
\newcommand{\ep}{\varepsilon}
\newcommand{\la}{\lambda}
\newcommand{\al}{\alpha}

\newcommand{\ka}{\kappa}
\newcommand{\vp}{\varphi}
\newcommand{\om}{\omega}

\renewcommand{\leq}{\leqslant}
\renewcommand{\geq}{\geqslant}

\newcommand{\bfA}{\mathbf{A}}

\newcommand{\mr}{\mathbb R}

\newcommand{\R}{\mathbb R}
\newcommand{\N}{\mathbb N}

\newcommand{\cac}{\mathcal C}
\newcommand{\cl}{\mathcal L}
\newcommand{\cn}{\mathcal N}

\newcommand{\lp}{\left(}
\newcommand{\rp}{\right)}

\DeclareMathOperator{\diver}{div}
\DeclareMathOperator*{\esssup}{esssup}

\newtheorem{theorem}{Theorem}[section]

\newtheorem{corollary}[theorem]{Corollary}

\newtheorem{lemma}[theorem]{Lemma}
\newtheorem{notation}[theorem]{Notation}

\newtheorem{proposition}[theorem]{Proposition}

\theoremstyle{definition}
\newtheorem{definition}[theorem]{Definition}

\newtheorem{hypothesis}[theorem]{Hypothesis}

\theoremstyle{remark}
\newtheorem{remark}[theorem]{Remark}
\theoremstyle{remark}

\newcommand{\cdummy}{\cdot}

\begin{document}

\title[A priori estimates for rough PDE\MakeLowercase{s}]{A priori estimates for rough PDE\MakeLowercase{s}\\ with application to rough conservation laws}

\author{Aur\'elien Deya}
\address[A. Deya]{Institut Elie Cartan, University of Lorraine
B.P. 239, 54506 Vandoeuvre-l\`es-Nancy, Cedex
France}
\email{Aurelien.Deya@univ-lorraine.fr}

\author{Massimiliano Gubinelli}
\address[M. Gubinelli]{Hausdorff Center for Mathematics \& Institute for Applied Mathematics, University of Bonn,
Endenicher Allee 60,
53115 Bonn, Germany}
\email{gubinelli@iam.uni-bonn.de}

\author{Martina Hofmanov\'a}
\address[M. Hofmanov\'a]{Institute of Mathematics,  Technical University Berlin, Stra\ss e des 17. Juni 136, 10623 Berlin, Germany}
\email{hofmanov@math.tu-berlin.de}

\author{Samy Tindel}
\address[S. Tindel]{Department of Mathematics, Purdue University,
150 N. University Street,  West Lafayette, IN 47907, United States}
\email{stindel@purdue.edu}

\thanks{Financial support by the DFG via Research Unit FOR 2402 and via CRC 1060 is gratefully acknowledged.}

\begin{abstract}
We introduce a general \rmk{weak formulation for PDEs driven by rough paths, as well as a new strategy to prove  well-posedness. Our procedure is based on a combination of fundamental a priori estimates with (rough) Gronwall-type arguments.} In particular this approach does not rely on any sort of transformation formula (flow transformation, Feynman--Kac representation formula etc.) and is therefore rather flexible.
As an application, we study  conservation laws driven by rough paths establishing well--posedness for the corresponding kinetic formulation.
\end{abstract}



\date{\today}

\maketitle

\tableofcontents

\section{Introduction}

Lyons~\cite{MR1654527} introduced rough paths to give a description of solutions to ordinary differential equation (ODEs) driven by external time varying signals which is robust enough to allow very irregular signals like the sample paths of a Brownian motion. His analysis singles out a rough path as the appropriate topological structure on the input signal with respect to which the solution of an ODE varies in a  continuous way. Since its invention, rough path theory (RPT) has been developed very intensively to provide robust analysis of ODEs and a novel way to define solutions of stochastic differential equations driven by non semimartingale signals. For a comprehensive review see the book of Friz and Victoir~\cite{FV} and the lecture notes of Lyons, Caruana and L\'evy~\cite{MR2314753} or the more recents ones of Friz and Hairer~\cite{friz_course_2014}. RPT can be naturally formulated also in infinite-dimension to analyse ODEs in Banach spaces. This generalisation is, however, not appropriate for the understanding of rough PDEs (RPDEs), i.e. PDEs with irregular perturbations. This is due to two basic facts. First, the notion of rough path encodes in a fundamental way only the nonlinear effects of time varying signals, without possibility to add more dimensions to the parameter space where the signal is allowed to vary in an irregular fashion. Second, in infinite dimension the action of a signal (even finite dimensional) can be described by differential (or more generally unbounded) operators. 

Due to these  basic difficulties, attempts to use RPT to study rough PDEs have been limited by two factors: the first one is the need to  look at RPDEs as evolutions in Banach spaces perturbed by one parameter rough signals (in order to keep rough paths as basic objects), the second one is the need to avoid unbounded operators by looking at mild formulations or Feynman--Kac formulas or transforming the equation in order to absorb the rough dependence into better understood objects (e.g. flow of characteristic curves).

These requirements pose strong limitations on the kind of results one is allowed to obtain for RPDEs and the proof strategies are very different from the classical PDE proofs. The most successful approaches to RPDEs do not even allow to characterise the solution directly but only via a transformation to a more standard PDE problem.
The need of a mild formulation of a given problem leads usually to very strong structural requirements like for example semilinearity.
We list here some pointers to the relevant literature:
\begin{itemize}
\item Flow transformations applied to viscosity formulation of fully non-linear RPDEs (including Backward rough differential equations) have been studied in a series of work by Friz and coauthors: Diehl and Friz~\cite{MR2978136},
Friz and Oberhauser~\cite{MR3152786},
Caruana and Friz~\cite{MR2510132},
Diehl, Friz and Oberhauser~\cite{MR3332714},
Caruana, Friz and Oberhauser~\cite{MR2765508} and finally Friz, Gassiat, Lions and Souganidis~\cite{160204746}.

\item Rough formulations of evolution heat equation with multiplicative noise (with varying degree of success) have been considered by Gubinelli and Tindel~\cite{MR2599193}, Deya, Gubinelli and Tindel~\cite{MR2925571}, Teichmann~\cite{MR2836540}, Hu and Nualart~\cite{MR2471936} and Garrido-Atienza, Lu and Schmalfuss~\cite{MR3423246}.

\item Mild formulation of rough Burgers equations with spatially irregular noise have been first introduced by Hairer and Weber~\cite{MR3010394,MR3129808} and Hairer, Maas and Weber~\cite{MR3179667} leading to the groundbreaking work of Hairer on the Kardar--Parisi--Zhang equation~\cite{MR3071506}.

\item Solutions of  conservation laws with rough fluxes have been studied via flow transformation by Friz and Gess~\cite{friz_stochastic_2014} and via the transformed test function approach by Lions, Perthame and Souganidis~\cite{MR3274890, MR3327520, MR3380984}, Gess and Souganidis~\cite{MR3351442,gess_long_time_2014}, Gess, Souganidis and Perthame ~\cite{gess_semi_discretization_2015} and Hofmanov\'{a}~\cite{H16}.

\end{itemize}

Hairer's regularity structure theory~\cite{MR3274562} is a wide generalisation of rough path which allows irregular objects parametrized by multidimensional indices. A more conservative approach, useful in many situations but not as general, is the paracontrolled calculus developed by Gubinelli, Imkeller and Perkowski~\cite{MR3406823,gubinelli_lectures_2015}. These techniques go around the first limitation. In order to apply them however the PDEs need usually to have a mild formulation where the unbounded operators are replaced by better behaved quantities and in general by bounded operators in the basic Banach spaces where the theory is set up. Existence and uniqueness of solutions to RPDEs are then consequences of standard fixed-point theorems in the Banach setting. 

PDE theory developed tools and strategies to study \emph{weak solutions} to PDEs, that is distributional relations satisfied by the unknown together with its weak derivatives. From a conceptual point of view  the wish arises to devise an approach to RPDEs which borrow as much as possible from the variety of tools and techniques of PDE theory. From this point of view various authors started to develop \emph{intrinsic} formulations of RPDEs as which involves relations between certain distributions associated to the unknown and the rough paths associated to the input signal. Let us mention the work of Gubinelli, Tindel and Torrecilla~\cite{gubinelli_controlled_2014} on viscosity solutions to fully non--linear rough PDEs, that of Catellier~\cite{catellier_rough_2015} on rough transport equations (in connection with the regularisation by  noise phenomenon), Diehl, Friz and Stannat~\cite{14126557} and finally of Bailleul and Gubinelli~\cite{BG} on rough transport equations. This last work introduces for the first time apriori estimates for RPDEs, that is estimates which holds for any weak solution of the RPDE (though we should also mention the contribution~\cite{nunu-vuillermot}, in which weak formulations are investigated for Young type equations driven by fractional Brownian motions with Hurst parameter $H>1/2$).
These estimates are crucial to derive control on various norms of the solution and obtain existence and uniqueness results, bypassing the use of the rough flow of characteristics which has been the main tool of many of the previous works on this subject. 

\medskip

\rmk{In the present paper, we continue the development of general tools for RPDEs along the ideas introduced in \cite{BG}. In particular, just as in the latter reference, we will rely on the formalism of \enquote{unbounded rough driver} in order to model the central operators governing the (rough part of the) dynamics in the equation. In fact, through the results of this paper, we propose to extend the considerations of \cite{BG} along several essential directions:}

\smallskip

\noindent
\rmk{$\bullet$ We include the possibility of an unbounded drift term in the model under consideration (see Definition \ref{def:gen}), and generalize the main apriori estimates accordingly (see Theorem \ref{theo:apriori}). This improvement considerably extends the range of possible equations covered by the approach, and we will indeed raise two fundamental examples that could not have been treated in the framework of \cite{BG}: first a heat-equation model with linear transport noise (Section \ref{subsec:first-application}), then a more compelling (and much more thorough) application to scalar conservation laws with rough fluxes, as introduced below.} 

\smallskip

\noindent
\rmk{$\bullet$ We rephrase the theory in the $p$-variation language and thus not restrict to the more specific H{\"o}lder topology used in \cite{BG}. Again, this technical extension, which requires a careful follow-up of the controls involved in the procedure, will prove to be of a paramount importance in the study of our main conservation-law model (see Remark \ref{rk:p-var-topo} for more details).}  

\smallskip

\noindent
\rmk{$\bullet$ We illustrate how to efficiently combine our general a priori estimates with Gronwall-type arguments. Skimming over any book on PDEs indeed shows how fundamental such a combination is for any nontrivial result on weak solutions. Therefore our strategy requires the clear statement of an effective \emph{rough Gronwall lemma} adapted to the $p$-variation setting (see Lemma \ref{lemma:basic-rough-gronwall} below). While Gronwall-like arguments are well known in the rough path literature, they have been essentially employed in the context of rough strong solutions. Here we show how to use them to obtain finer informations about rough weak solutions. This will require new ideas to overcome technical difficulties when working with test functions.}

\smallskip

\noindent
\rmk{$\bullet$ We solve, via the constructions of Section \ref{subsec:smoothing}, an important technical question left open in \cite{BG} about tensorization of the rough equation and the related space of test-functions (see Remark \ref{rk:smoothing-bg} for more details). For the sake of clarity,  we made the whole tensorization argument self-contained with respect to \cite{BG}.}

\medskip

\rmk{Let us now elaborate on what will be the main illustration of the above technical contributions (and what will actually occupy the largest part of the paper), namely the rough extension of the so-called \enquote{conservation laws} equation.}

\smallskip

Conservation laws and related equations have been paid an increasing attention lately and have become a very active area of research, counting nowadays quite a number of results for deterministic and stochastic setting, that is for conservation laws either of the form
\begin{equation}\label{eq:det}
\partial_t u+\diver (A(u))=0 ,
\end{equation}
(see \cite{vov1,car,vov,kruzk,lpt1,lions,perth,tadmor}) or
\begin{equation*}
\dif u+\diver (A(u))\dif t=g(x,u)\dif W ,
\end{equation*}
where the It\^o stochastic forcing is driven by a finite- or infinite-dimensional Wiener process (see \cite{bauzet,karlsen,debus2,debus,DV,feng,bgk,holden,kim,stoica,wittbold}). Degenerate parabolic PDEs were studied in~\cite{car,chen} and in the stochastic setting in~\cite{BVW,degen2,hof}.

Recently, several attempts have already been made to extend rough path techniques to conservation laws as well. First, Lions, Perthame and Souganidis (see \cite{MR3327520,MR3274890}) developed a pathwise approach for
$$\dif u+\diver(A(x,u))\circ \dif W=0 ,$$
where $W$ is a continuous real-valued signal and $\circ$ stands for the Stratonovich product in the Brownian case, then Friz and Gess (see \cite{friz_stochastic_2014}) studied
$$\dif u+\diver f(t,x,u)\dif t=F(t,x,u)\dif t+\Lambda_k(x,u,\nabla u)\dif z^k ,$$
where $\Lambda_k$ is affine linear in $u$ and $\nabla u$ and $z=(z^1,\dots,z^K)$ is a rough driving signal. Gess and Souganidis \cite{MR3351442} considered
\begin{equation}\label{eq1a}
\dif u+\diver (A(x,u))\dif z=0 ,
\end{equation}
where $z=(z^1,\dots,z^M)$ is a geometric $\alpha$-H\" older rough path and in \cite{gess_long_time_2014} they studied the long-time behavior in the case when $z$ is a Brownian motion. Hofmanov\'a \cite{H16} then generalized the method to the case of mixed rough-stochastic model
$$\dif u+\diver (A(x,u))\dif z=g(x,u)\dif W,$$
where $z$ is a geometric $\alpha$-H\"older rough path, $W$ is a Brownian motion and the stochastic integral on the right hand side is understood in the sense of It\^o.

It was observed already a long time ago that, in order to find a suitable concept of solution for problems of the form \eqref{eq:det}, on the one hand classical $C^1$ solutions do not exist in general and, on the other hand, weak or distributional solutions lack uniqueness.
The first claim is a consequence of the fact that any smooth solution has to be constant along characteristic lines, which can intersect in finite time (even in the case of smooth data) and shocks can be produced. The second claim demonstrates the inconvenience that often appears in the study of PDEs and SPDEs: the usual way of weakening the equation leads to the occurrence of nonphysical solutions and therefore additional assumptions need to be imposed in order to select the physically relevant ones and to ensure uniqueness. Hence one needs to find some balance that allows to establish existence of a unique (physically reasonable) solution.

Towards this end,  Kru\v{z}kov \cite{kruzk} introduced the notion of entropy solution to~\eqref{eq:det}, further developed in the stochastic setting in~\cite{bauzet,car,feng,kim,wittbold}.
Here we pursue the kinetic approach, a concept of solution that was first introduced by Lions, Perthame, Tadmor \cite{lions} for deterministic hyperbolic conservation laws and further studied in \cite{vov1}, \cite{chen}, \cite{vov}, \cite{lpt1}, \cite{lions}, \cite{tadmor}, \cite{perth}. This direction also appears in several works on stochastic conservation laws and degenerate parabolic SPDEs, see \cite{degen2}, \cite{debus2}, \cite{debus}, \cite{DV}, \cite{bgk}, \cite{hof} and in the (rough) pathwise works \cite{gess_long_time_2014}, \cite{MR3351442}, \cite{H16}, \cite{MR3327520}, \cite{MR3274890}. 

Kinetic solutions are more general in the sense that they are well defined even in situations when neither the original conservation law nor the corresponding entropy inequalities can be understood in the sense of distributions. Usually this happens due to lack of integrability of the flux and entropy-flux terms, e.g. $A(u)\notin L^1_{\text{loc}}$. Therefore, further assumptions on initial data or the flux function $A$ are in place in order to overcome this issue and remain in the entropy setting. It will be seen later on that no such restrictions are necessary in the kinetic approach as the equation that is actually being solved -- the so-called kinetic formulation, see~\eqref{eq:kinform} -- is in fact linear. In addition, various proofs simplify as methods for linear PDEs are available.

In the present paper, we are concerned with scalar rough conservation laws of the form~\eqref{eq1a}, where $z=(z^1,\dots,z^M)$ can be lifted to a geometric rough path of finite $p$-variation for $p\in [2,3)$. 
We will show how our general tools allow to  treat~\eqref{eq1a}
 along the lines of the standard PDEs proof strategy. Unlike the known results concerning the same problem (see e.g. \cite{MR3351442,MR3327520,MR3274890}), our method does not rely on the flow transformation method and so it overcomes the limitations inevitably connected with such a transformation. Namely, we are able to significantly weaken the assumptions on the flux coefficient $A=(A_{ij})$: we assume that $a_{ij}=\partial_\xi A_{ij}$ and $b_j=\diver_x A_{\cdot j}$ belong to $W^{3,\infty}$, whereas in \cite{MR3351442} the regularity of order $\mathrm{Lip}^{2+\gamma}$ is required for some $\gamma>\frac{1}{\alpha}$ with $\alpha\in(0,1)$ being the H\"older regularity of the driving signal. For a 2-step rough path, i.e. in the range $\alpha\in (\tfrac{1}{3},\frac{1}{2}]$, it therefore means that almost five derivatives might be necessary. Nevertheless, let us point out that even the regularity we require is not the optimal one. To be more precise, with a more refined method we expect that one could possibly only assume  $W^{\gamma,\infty}$-regularity for the coefficients $a,b$ with $\gamma > p$.

\smallskip

\rmk{For the sake of a clearer presentation and in order to convey the key points of our strategy as effectively as possible, we will limit the scope of this paper to the first \enquote{non-trivial} rough situation, that is to $p\in [2,3)$. This being said, we are very confident with the possibility to extend the general pattern of this method to rougher cases, that is to any $p\geq 2$, at the price of a heavier algebraic machinery.}

\subsection*{Outline of the paper} In Section~\ref{sec:general} we fix notations, introduce the notion of unbounded rough driver and establish the main tools used thereafter: a priori estimates for distributional solutions to rough equations and a related rough Gronwall lemma. \rmk{For pedagogical purpose, we then provide a first possible application of these results to a rough heat equation model with transport noise (Section \ref{subsec:first-application}). In Section~\ref{sec:uniq1}, we discuss the theoretical details of the tensorization method needed to prove bounds on nonlinear functions of the solution.} Section~\ref{sec:conservation-presentation} introduces the setting for the analysis of conservations laws with rough fluxes. Section~\ref{sec:uniq} uses the tensorization method to obtain estimates leading to reduction, $L^1$-contraction and finally uniqueness for kinetic solutions. In Section~\ref{sec:apriori} we prove some $L^p$-apriori bounds on solutions which are stable under rough path topology. These bounds are finally used in Section~\ref{sec:existence} to prove existence of kinetic solutions.

\subsection*{Acknowledgements}
The authors would like to thank Dr. Mario Maurelli for some discussions about tensorization and the anonymous referee for his careful reading and the extensive comments which helped them to substantially improve the presentation of the results.

\section{General a priori estimates for rough PDEs}
\label{sec:general}

\subsection{Notation}

First of all, let us recall the definition of the increment operator, denoted by $\delta$. If $g$ is a path defined on $[0,T]$ and $s,t\in[0,T]$ then $\delta g_{st}:= g_t - g_s$, if $g$ is a $2$-index map defined on $[0,T]^2$ then $\delta g_{sut}:=g_{st}-g_{su}-g_{ut}$. The norm of the element $g$, \rmk{considered as an element of a Banach space $E$, will be written indistinctly as:
\begin{equation}\label{eq:notation-norms}
\|g\|_{E},
\quad\text{or}\quad
\cn[g; E] .
\end{equation}
 } 
 For two quantities $a$ and $b$ the relation $a\lesssim_{x} b$ means $a\leq c_{x} b$, for a constant $c_{x}$ depending on a multidimensional parameter $x$.

In the sequel, given an interval $I$ we call \emph{control on $I$} (and denote it by $\omega$) any superadditive map on $\Delta_I:=\{(s,t)\in I^2: \ s\leq t\}$, that is, any map $\omega: \Delta_I \to [0,\infty[$ such that, for all $s\leq u\leq t$,
$$\omega(s,u)+\omega(u,t) \leq \omega(s,t).$$
We will say that a control is \emph{regular} if $\lim_{|t-s|\to 0} \omega(s,t) = 0$. Also, given a control $\omega$ on an interval $I=[a,b]$, we will use the notation $\omega(I):=\omega(a,b)$. 
Given a time interval $I$, a parameter $p>0$ \rmk{and a Banach space $E$, we denote by $\overline V^p_1(I;E)$ the space of functions $g : I \to E$ for which the quantity
$$\sup_{(t_i) \in \mathcal{P}(I)}  \sum_{i} |g_{t_i} - g_{t_{i+1}}|^p $$
is finite, where $\mathcal{P}(I)$ stands for the set of all partitions of the interval $I$. For any $g\in \overline V^p_1(I;E)$,
$$\omega_g(s,t) = \sup_{(t_i) \in \mathcal{P}([s,t])}  \sum_{i} |g_{t_i} - g_{t_{i+1}}|^p$$
defines a control on $I$,} and we denote by $ V^p_1(I;E)$ the set of elements $g\in \overline V^p_1(I;E)$ for which $\omega_g$ is regular on $I$. 
We denote by $\overline V^p_2(I;E)$ the set of two-index maps $g : I\times I \to E$ with left and right limits in each of the variables and for which there exists a control $\omega$ such that
$$
|g_{st}| \leq \omega(s,t)^{\frac{1}{p}}
$$
for all $s,t\in I$. We also define the space $\overline V^p_{2,\text{loc}}(I;E)$ of maps $g:I\times I\to E$ such that there exists a countable covering $\{I_k\}_k$ of $I$ satisfying $g \in \overline V^p_2(I_k;E)$ for any $k$. We write $g\in  V^p_2(I;E)$ or $g\in  V^p_{2,\text{loc}}(I;E)$ if the control can be chosen regular.

\rmk{\begin{definition}\label{def-p-rough-path}
Fix $K\geq 1$, $p\in [2,3)$, and $I$ a finite interval. We will call a continuous (weak geometric) $p$-rough path on $I$ any element $\mathbf{Z}=(Z^1,Z^2) \in V^p_2(I;\R^K) \times V^{\frac{p}{2}}_2(I;\R^{K,K})$ such that for all $1\leq i,j\leq K$ and $s<u<t\in I$,
$$Z^{1,i}_{st}=Z^{1,i}_{su}+Z^{1,i}_{ut} \quad , \quad Z^{2,ij}_{st}=Z^{2,ij}_{su}+Z^{2,ij}_{ut}+Z^{1,i}_{su}Z^{1,j}_{ut} \quad , \quad Z^{2,ij}_{st}+Z^{2,ji}_{st}=Z^{1,i}_{st}Z^{1,j}_{st} \ .$$
Then we will say that a path $z\in V^p_1(I;\R^K)$ can be lifted to a continuous (weak geometric) $p$-rough path $\mathbf{Z}=(Z^1,Z^2)$ if $\mathbf{Z}$ is a (weak geometric) $p$-rough path such that $Z^1_{st}=z_t-z_s$. 
\end{definition}}

\

\begin{lemma}[Sewing lemma]\label{lemma-lambda}
Fix an interval $I$, a Banach space $E$ and a parameter $\zeta >1$. Consider a map $h:I^3 \to E$ such that $h\in \mathrm{Im} \,\delta$ and for every $s<u<t \in I$,
\begin{eqnarray} 
|h_{sut}|\leq \omega(s,t)^{\zeta} ,
\end{eqnarray}
for some regular control $\omega$ on $I$. Then there exists a unique element $\Lambda h\in V_2^{\frac{1}{\zeta}}(I;E)$ such that $\delta(\Lambda h)=h$ and for every $s<t\in I$,
\begin{eqnarray} \label{contraction}
|(\Lambda h)_{st}|\leq C_\zeta\, \omega(s,t)^{\zeta}  ,
\end{eqnarray}
for some universal constant $C_\zeta$.
\end{lemma}

\begin{proof}
The proof follows that given in \cite[Lemmma 4.2]{friz_course_2014} for H\"older norms, we only specify the modification needed to handle variation norms.
Regarding existence, we recall that since $\delta h=0$, there exists a $2$-index map $B$ such that $\delta B=h$.
Let $s,\,t\in[0,T]$, such that $s<t$, and consider a sequence $\{\pi_n;\,n\geq0\}$ of partitions $\{s=r_0^n<\cdots< r_{k_n+1}^n=t\}$ of $[s,t]$. Assume that $\pi_n\subset\pi_{n+1}$ and \rmk{$\lim_{n\rightarrow\infty} \sup_{0\leq i\leq k_n} |r^n_{i+1}-r^n_i|=0$}. Set 
$$M^{\pi_n}_{st}=B_{st}-\sum_{i=0}^{k_n}B_{r^n_{i},r^n_{i+1}}.$$
Due to superadditivity of $\omega$ it can be seen that there exists $l\in\{1,\dots,k_n\}$ such that
$$\omega(r^n_{l-1},r^n_{l+1})\leq \frac{2\omega(s,t)}{k_n} \ .$$
Now we choose such an index $l$ and transform $\pi_n$ into $\hat \pi$, where
$\hat\pi=\{r_0^n<r_1^n<\cdots<r_{l-1}^n<r^n_{l+1}<\cdots<r_{k_n+1}^n\}$.
Then
$$M^{\hat\pi}_{st}=M^{\pi_n}_{st}-(\delta B)_{r^n_{l-1},r^n_{l},r^n_{l+1}}
=M^{\pi_n}_{st}-h_{r^n_{l-1},r^n_{l},r^n_{l+1}}$$
and hence
$$|M^{\hat\pi}_{st}-M^{\pi_n}_{st}|\leq \omega(r^n_{l-1},r^n_{l+1})^{\zeta}\leq \bigg[\frac{2\omega(s,t)}{k_n}\bigg]^\zeta.$$
Repeating this operation until we end up with the trivial partition $\hat\pi_0=\{s,t\}$, for which $M^{\hat\pi_0}_{st}=0$ this implies that $M^{\pi_n}_{st}$ converges to some $M_{st}$ satisfying
\begin{align*}
|M_{st}| = \lim_n |M^{\pi_n}_{st}|\leq  2^\zeta\omega(s,t)^\zeta\sum_{i=1}^{\infty}i^{-\zeta}\leq C_{\zeta}\,\omega(s,t)^\zeta \ .
\end{align*}
\end{proof}

\subsection{Unbounded rough drivers}\label{sec:urd} 
\rmk{Let $p\in [2,3)$ be fixed for the whole section. In what follows, we call a ($p$-)scale any 4-uplet $\big(E_n,\lVert\cdot\rVert_n\big)_{0\leq n\leq 3}$ of Banach spaces such that $E_{n+1}$ is continuously embedded into $E_n$. Besides, for $0\leq n\leq 3$, we denote by $E_{-n}$ the topological dual of $E_n$.}

\begin{definition}
\label{def:urd}
A continuous unbounded $p$-rough driver with respect to a scale \rmk{$\big(E_n,\lVert\cdot\rVert_n\big)_{0\leq n\leq 3}$}, is a pair $\mathbf{A} = \big(A^1,A^2\big)$ of $2$-index maps such that
$$A^1_{st}\in  \cl(E_{-n},E_{-(n+1)}) \ \ \text{for} \ \ n\in \{0,2\}\ , 
\qquad  A^2_{st} \in \cl(E_{-n},E_{-(n+2)}) \ \  \text{for} \ \ n\in \{0,1\}\ ,$$
and there exists a regular control $\omega_A$ on $[0,T]$ such that for every $s,t\in [0,T]$,
\begin{align}
\lVert A^1_{st}\rVert_{\cl(E_{-n},E_{-(n+1)})}^p &\leq\omega_A(s,t) \qquad \text{for}\ \ n\in \{0,2\}\ , \label{control-a-1}\\
\lVert A^2_{st}\rVert_{\cl(E_{-n},E_{-(n+2)})}^{p/2} &\leq\omega_A(s,t) \qquad \text{for}\ \ n\in \{0,1\}\ ,\label{control-a-2}
\end{align}
and, in addition, the Chen's relation holds true, that is,
\begin{equation}\label{eq:chen-relation}
\delta A^1_{sut}=0,\qquad\delta A^2_{sut}=A^1_{ut}A^1_{su},\qquad\text{for all}\quad 0\leq s<u<t\leq T.
\end{equation}
\end{definition}

\

\rmk{To see how such unbounded drivers arise in the study of rough PDEs, let us consider the following linear heat-equation model:
\begin{align}\label{eq:heat}
\begin{aligned}
\dd u &=\Delta u\, \dd t + V\cdot \nabla u \,\dd z,\qquad x\in\R^N \ ,\ t\in(0,T)\ ,\\
u(0)&=u_0,
\end{aligned}
\end{align}
where $V=(V^1,\dots,V^K)$ is a family of smooth vector fields on $\R^N$, and assuming for the moment that $z=(z^1,\dots,z^K):[0,T]\to \R^K$ is a smooth path. This (classical) equation can of course be understood in the weak sense: for any test-function $\vp\in W^{1,2}(\R^N)$, it holds that
$$\delta u(\vp)_{st}=\int_{s}^{t} u_r (\Delta \varphi)  \dd r +\int_s^t u_r( \text{div}( V^{k}\vp ))\dd z_r\ .$$
Using a basic Taylor-expansion procedure (along the time parameter) and when $\vp\in W^{3,2}(\R^N)$, the latter expression can be easily developed as
\begin{equation}\label{eq:he1}
\delta  u(\varphi )_{s t} =\int_{s}^{t} 
      u_r (\Delta \varphi)  \dd r+ u_{s} (A^{1,
     \ast}_{s t} \varphi ) +  u_{s} (A^{2, \ast}_{s t} \varphi
     ) + u^{\natural}_{s t}( \varphi ) \ ,
\end{equation}
where we have set (using Einstein summation convention)
\begin{equation}\label{eq:def-A-star}
A_{st}^{1,\ast}\vp
=- Z_{st}^{1,k} \, \text{div}( V^{k}\vp )
\quad , \quad
A_{st}^{2,\ast}\vp
= Z_{st}^{2,jk} \, \text{div}( V^{j} \text{div}( V^{k}\vp )) \ ,
\end{equation}
with $Z^{1}, Z^{2}$ defined by
\begin{equation}\label{eq:def-A-star-bis}
Z_{st}^{1,k}=\delta z^k_{st}
\quad , \quad
Z_{st}^{2,jk}=  \int_s^t \delta z^j_{sr} \dd z^k_r \ ,
\end{equation}
and where $u^{\natural}$ morally stands for some third-order remainder (in time) acting on $W^{3,2}(\R^N)$.}

\smallskip

\rmk{Expansion (\ref{eq:he1}) puts us in a position to extend the problem to a rough level and to motivate the above Definition \ref{def:urd}. Assume indeed that $z$ is no longer smooth but can still be lifted to a continuous (weak geometric) $p$-rough path $\mathbf{Z}=(Z^1,Z^2)$ (in the sense of Definition \ref{def-p-rough-path}), for some fixed $p\in [2,3)$. Then the two operator-valued paths $A^{1,\ast}$, $A^{2,\ast}$ can be extended along the very same formula (\ref{eq:def-A-star}), or equivalently along the dual forms 
\begin{equation}\label{driver}
A^1_{st} u:=Z^{1,k}_{st} \,V^k \cdot\nabla u \ ,\qquad A^2_{s t}u:= Z^{2,jk}_{st}\, V^k\cdot\nabla (V^j\cdot\nabla u) \ ,
\end{equation}
which, as one can easily check it, provides us with an example of an unbounded rough driver, for instance on the Sobolev scale $E_n:=W^{n,2}(\R^N)$ ($0\leq n\leq 3$).}

\smallskip

\rmk{Once endowed with $\bfA=(A^1,A^2)$, and along the same principles as in \cite{BG}, our interpretation of \eqref{eq:heat} (or \eqref{eq:he1}) will essentially follow Davie's approach to rough systems (\cite{davie}). Namely, we will call a solution any path $u$ satisfying the property: for every $0 \leqslant s \leqslant t \leqslant T$ and every test-function $\varphi \in E_{3}$, the decomposition (\ref{eq:he1}) holds true, for some $E_{-3}$-valued $2$-index map $u^\natural$ such that for every $\varphi\in E_3$ the map $u^{\natural}_{s t}( \varphi )$ possesses sufficient time regularity, namely, it has finite $r$-variation for some $r < 1$.}

\subsection{A priori estimates and rough Gronwall lemma}\label{subsec:apr}

\rmk{Before we turn to the main purpose of this subsection, namely the presentation of the mathematical tools at the core of our analysis, let us extend the previous considerations and introduce rough PDEs of the general form}
\begin{equation}\label{eq:gen}
\dd g_{t} = \mu(\dd t)+ {\bf A}(\dd t) g_{t} \  ,
\end{equation}
where $\bfA=(A^1,A^2)$ is an unbounded $p$-rough driver on a scale \rmk{$(E_n)_{0\leq n\leq 3}$} and the drift $\mu$, which possibly also depends on the solution, is continuous of finite variation. 

\smallskip

\rmk{Following the above ideas,} we now give a rigorous meaning to such an equation.

\begin{definition}\label{def:gen}
Let $p\in [2,3)$ and fix an interval $I\subseteq [0,T]$. Let ${\bf A} = \big(A^1,A^2\big)$ be a continuous unbounded $p$-rough driver on $I$ with respect to a scale \rmk{$(E_n)_{0\leq n\leq 3}$} and let $\mu\in \overline V_1^1(I;E_{-3})$. A path $g:I \to E_{-0}$ is called a solution (on $I$) of the equation \eqref{eq:gen} provided there exists $q<3$ and  $g^\natural \in V^{\frac{q}{3}}_{2,\text{loc}}(I,E_{-3})$ such that for every $s,t\in I$, $s<t$, \rmk{and $\varphi\in E_3$,}
\begin{equation}\label{eq:gen2}
(\delta g)_{st}(\varphi)=(\delta \mu)_{st}(\varphi)+g_s(\{A^{1,*}_{st}+A^{2,*}_{st}\}\varphi)+g^\natural_{st}(\varphi).
\end{equation}

\end{definition}

\rmk{\begin{remark}
Throughout the paper, we will set up the convention that every $2$-index map with a $\natural$ superscript denotes an element of $V^{1/\zeta}_{2,\text{loc}}([0,T],E_{-3})$ for some $\zeta > 1$.
\end{remark}
\begin{remark}
In the heat-equation model (\ref{eq:gen2}), we thus have $g=u$ and $\mu_t(\varphi)=\mu^1_t(\varphi):=\int_{s}^{t} u_r (\Delta \varphi)  \dd r$. Note however that the formulation (\ref{eq:gen2}) allows the possibility of a very general drift term $\mu$, as illustrated by our forthcoming conservation-law model (see Definition \ref{genkinsol}). In particular, the linearity of $\mu^1$ with respect to $u$ would not play any essential role in our approach of (\ref{eq:gen2}), and therefore we believe that this strategy could also be useful in situations where $\mu$ is derived from a quasilinear elliptic or monotone operator. We do not intend to pursue here this line of research.
\end{remark}
\begin{remark}\label{rk:p-var-topo}
The consideration of $p$-variation topology (and not H{\"o}lder topology) in Definition~\ref{def:gen} will be essential in the study of our rough conservation-law model (Sections \ref{sec:conservation-presentation} to \ref{sec:existence}), for two fundamental (and linked) reasons. First, it is a well-known fact that, even in the smooth case, solutions to conservation laws are likely to exhibit discontinuities, a phenomenon which could not be covered by the H{\"o}lder setting. Besides, in the course of the procedure, we will be led to consider drift terms of the form $\mu_t(\varphi):=m(\mathbf{1}_{[0,t)} \otimes \varphi)$ ($\varphi\in \cac^\infty(\R^N)$), for some measure $m$ on $[0,T]\times \R^N$ that can admit atoms: such a map $\mu$ clearly defines a $1$-variation path, but in general it may not  be continuous.
\end{remark}}

\rmk{Let us now present our first main result on an a priori estimate for the remainder $g^\natural$ involved in (\ref{eq:gen2}). An important role will be played by the $E_{-1}$-valued $2$-index map $g^\sharp$ defined as 
\begin{equation}\label{eq:def-g-sharp}
g^\sharp_{st}(\varphi):=\delta g(\vp)_{st}-g_s(A^{1,*}_{st}\vp)\ .
\end{equation}
Observe that due to \eqref{eq:gen2}, this path is also given by
$$g^\sharp_{st}(\varphi)=(\delta\mu)_{st}(\varphi)+g_s(A^{2,*}_{st}\varphi)+g^\natural_{st}(\varphi).$$
In the following result we will make use of both these expressions, depending on the necessary regularity: the former one contains terms that are less regular in time but more regular in space (i.e. they require less regular test functions) whereas the terms in the latter one are more regular in time but less regular in space.}

\smallskip

\rmk{In order to balance this competition between time and space regularities, and following the ideas of \cite{BG}, we shall assume that a suitable family of \enquote{smoothing} operators can be involved into the procedure:}

\rmk{\begin{definition}\label{def:smoothing}
We call a {\it smoothing} on a given scale $(E_n)_{0\leq n\leq 3}$ any family of operators $(J^\eta)_{\eta\in (0,1)}$ acting on $E_n$ (for $n=1,2$) in such a way that the two following conditions are satisfied:
\begin{equation}\label{eq:reg-J-eta-1}
\lVert  J^\eta-\id\rVert_{\cl(E_m,E_n)} 
\lesssim 
\eta^{m-n} \quad \text{for} \ \ (n,m)\in \{(0,1),(0,2),(1,2)\} \  , 
\end{equation}
\begin{equation}\label{eq:reg-J-eta-2}
\lVert J^\eta\rVert_{\cl(E_n,E_{m})} \lesssim \eta^{-(m-n)} \quad \text{for} \ \ (n,m)\in \{(1,1),(1,2),(2,2),(1,3),(2,3)\} \ .
\end{equation}
\end{definition}}

\

\rmk{With this framework in mind, our main technical result concerning equation \eqref{eq:gen}-\eqref{eq:gen2} can now be stated as follows:} 

\begin{theorem}\label{theo:apriori}
\rmk{Let $p\in[2,3)$ and fix an interval $I\subseteq [0,T]$. Let ${\bf A}=(A^1,A^2)$ be a continuous unbounded $p$-rough driver with respect to a scale $(E_n)_{0\leq n\leq 3}$, endowed with a smoothing $(J^\eta)_{\eta\in (0,1)}$ (in the sense of Definition~\ref{def:smoothing}), and let $\omega_A$ be a regular control satisfying \eqref{control-a-1}-\eqref{control-a-2}.
Consider a path $\mu \in \overline V_1^1(I;E_{-3})$ for which there exist two controls $\omega^1_\mu,\omega_\mu^2$ and a constant $\lambda\in[p,3]$ such that for every $\varphi \in E_3$, $s<t\in I$, $\eta\in (0,1)$ and $k\in \{1,2\}$, one has
\begin{equation}\label{cond-mu}
|(\delta \mu)_{st}(J^\eta\varphi)|\leq \omega^1_\mu(s,t)\, \|\varphi\|_{E_1}+\eta^{k-\la}\omega^2_\mu(s,t)\, \|\varphi\|_{E_k} \ .
\end{equation}
Let  $g$ be a solution on $I$  of the equation \eqref{eq:gen2} such that $g^\natural \in V^{\frac{q}{3}}_2(I;E_{-3})$, for some parameter
\begin{equation}\label{q}
q\in\left[\frac{3p\lambda}{2p+\lambda},3\right).
\end{equation}
Finally, let $G_{st}=\cn[g;L^\infty(s,t;E_{-0})]$, where we recall that the notation $\cn$ is introduced by~\eqref{eq:notation-norms},  fix
 $\kappa\in[0,\tfrac{1}{p})$ such that
\begin{equation}\label{kappa}
\frac{1}{2}\left(\frac{3}{p}-\frac{3}{q}\right)\geq \kappa\geq\frac{1}{\lambda-2}\left(\frac{3}{q}-1-\frac{3-\lambda}{p}\right),
\end{equation}
and set
\begin{align}\label{eq:def-omega-I}
\omega_\ast(s,t) 
:=  G_{st}^\frac{q}{3}\,\omega_A(s,t)^{\frac{q}{3}(\frac{3}{p}-2\kappa)}+\omega^1_\mu(s,t)^\frac{q}{3}\omega_A(s,t)^{\frac{q}{3p}}+\omega^2_\mu(s,t)^\frac{q}{3}\omega_A(s,t)^{\frac{q}{3}(\frac{3-\lambda}{p}+\kappa (\lambda-2))}.
\end{align}
Then there exists a constant $L=L(p,q,\kappa)>0$ such that if $\omega_A(I)\leq L$, one has, for all $s<t\in I$,
\begin{equation}\label{apriori-bound} 
\begin{split}
 \|g^{\natural}_{st}\|_{E_{-3}}
 \lesssim_{q} \omega_\ast(s,t)^{\frac{3}{q}}.
\end{split}
\end{equation}}
\end{theorem}

\rmk{The high level of generality of this statement (that is, the involvement of three parameters $\ka,\la,q$ and two controls $\omega^1_\mu,\omega^2_\mu$) will indeed be required in the sequel, and more precisely in the strategy displayed in Section \ref{sec:uniq} for rough conservation laws. However, let us here specialize this result for a more readable statement, which will turn out to be sufficient for our other applications (namely, in Section \ref{subsec:first-application} and in Sections \ref{sec:apriori}-\ref{sec:existence}):} 

\rmk{\begin{corollary}\label{cor:apriori}
In the setting of Theorem \ref{theo:apriori}, consider a path $\mu \in \overline V_1^1(I;E_{-3})$ for which there exists a  control $\omega_\mu$ such that for all $s<t\in I$ and $\varphi \in E_3$, 
\begin{equation}\label{cond-mu-cor}
|(\delta \mu)_{st}(\varphi)|\leq \omega_\mu(s,t)\, \|\varphi\|_{E_2} \ .
\end{equation}
Besides, let $g$ be a solution on $I$  of the equation \eqref{eq:gen} such that $g^\natural \in V^{\frac{p}{3}}_2(I;E_{-3})$.
Then there exists a constant $L=L(p)>0$ such that if $\omega_A(I)\leq L$, one has, for all $s,t\in I$, $s<t$,
\begin{equation}\label{apriori-bound-cor} 
\begin{split}
 \|g^{\natural}_{st}\|_{E_{-3}}
 \lesssim_{q} \cn[g;L^\infty(s,t;E_{-0})]\,\omega_A(s,t)^\frac{3}{p}+\omega_\mu(s,t)\omega_A(s,t)^{\frac{3-p}{p}} \ .
\end{split}
\end{equation}
\end{corollary}}

\rmk{\begin{proof}[Proof of Corollary \ref{cor:apriori}]
Thanks to (\ref{cond-mu-cor}), one has, for all $\vp\in E_3$ and $\eta\in(0,1)$,
\begin{align*}
|(\delta \mu)_{st}(J^\eta \varphi)|\leq \omega_\mu(s,t)\, \|J^\eta \varphi\|_{E_2} &\lesssim \omega_\mu(s,t) \min \big( \eta^{-1}\| \varphi\|_{E_1}, \|\varphi\|_{E_2} \big)\\
&\lesssim \omega_\mu(s,t) \min \big( \eta^{1-p}\| \varphi\|_{E_1}, \eta^{2-p}\|\varphi\|_{E_2} \big) \ ,
\end{align*}
which readily allows us to take $\la=q=p$, $\ka=0$, $\omega^1=0$ and $\omega^2=c\, \omega_\mu$ (for some universal constant $c$) in the statement of Theorem \ref{theo:apriori}.
\end{proof}}

\begin{proof}[Proof of Theorem \ref{theo:apriori}]
Let $\omega_\natural(s,t)$ be a regular control such that $\|g^{\natural}_{st}\|_{E_{-3}} \le \omega_\natural(s,t)^{\frac{3}{q}}$ for any $s,t \in I$.
Let $\varphi \in E_3$ be such that $\lVert \varphi \rVert_{E_3} \leq 1$. We first show that
\begin{equation}\label{delta-reste}
(\delta g^{\natural}(\varphi))_{sut}= (\delta g)_{su}(A^{2,*}_{ut} \varphi)+g^{\sharp}_{su}(A^{1,*}_{ut} \varphi) ,
\end{equation}
where $g^{\sharp}$ was defined in \eqref{eq:def-g-sharp}.
Indeed, owing to \eqref{eq:gen2}, we have 
\begin{equation*}
g^{\natural}_{st}(\vp) = \delta g(\vp) _{st}- g_s \big( [A_{st}^{1,*} + A_{st}^{2,*}](\vp)\big) - \delta \mu(\vp)_{st} \ .
\end{equation*}
Applying $\delta$ on both sides of this identity and recalling Chen's relations \eqref{eq:chen-relation} as well as the fact that $\delta \delta=0$ we thus get
\begin{equation*}
\delta g^{\natural}_{sut}(\vp) = (\delta g)_{su} ( [A^{1,*}_{ut} + A_{ut}^{2,*}](\vp)) - g_s (  A^{1,*}_{su}A^{1,*}_{ut}(\vp)) .
\end{equation*}
Plugging relation \eqref{eq:def-g-sharp} again into this identity, we end up with our claim \eqref{delta-reste}.

The aim now is to bound the terms on the right hand side of \eqref{delta-reste} separately by the allowed quantities $G,\,\omega_\mu,\,\omega_\natural$ and to reach a sufficient time regularity as required by the sewing lemma \ref{lemma-lambda}.
To this end, we make use of the smoothing operators $(J^\eta)$ and repeatedly apply the equation \eqref{eq:gen2} as well as the two equivalent definitions of $g^\sharp$ from \eqref{eq:def-g-sharp}. We obtain
\begin{align}\label{eq:dcp-g-nat}
\delta g^\natural(\vp)_{sut}&=(\delta g)_{su}(J^{\eta} A^{2,*}_{ut} \varphi)+(\delta g)_{su}((\id-J^{\eta}) A^{2,*}_{ut} \varphi) \notag\\
&\quad+g^{\sharp}_{su}(J^{\eta} A^{1,*}_{ut} \varphi)+g^{\sharp}_{su}((\id-J^{\eta})A^{1,*}_{ut} \varphi) \notag\\
&=g_{s}(A^{1,*}_{su}J^{\eta} A^{2,*}_{ut} \varphi)+g_{s}(A^{2,*}_{su}J^{\eta} A^{2,*}_{ut} \varphi)+(\delta\mu)_{su}(J^{\eta} A^{2,*}_{ut} \varphi)+g^\natural_{su}(J^{\eta} A^{2,*}_{ut} \varphi) \notag\\
&\quad+(\delta g)_{su}((\id-J^{\eta}) A^{2,*}_{ut} \varphi) \notag\\
&\quad+g_{s}(A^{2,*}_{su}J^{\eta} A^{1,*}_{ut} \varphi)+(\delta\mu)_{su}(J^{\eta} A^{1,*}_{ut} \varphi)+g^\natural_{su}(J^{\eta} A^{1,*}_{ut} \varphi) \notag\\
&\quad+(\delta g)_{su}((\id-J^{\eta})A^{1,*}_{ut}\vp)-g_{s}(A^{1,*}_{su}(\id-J^{\eta}) A^{1,*}_{ut} \varphi) \notag\\
&=I_1+\cdots+I_{10}
\end{align}
The use of the smoothing operators reflects the competition between space and time regularity in the various terms in the equation. To be more precise, the only available norm of $g$ is $L^\infty(0,T;E_{-0})$. So on the one hand $g$ does not possess any time regularity (at least at this point of the proof) but on the other hand it does not require any space regularity of the corresponding test functions. In general, the presence of $(\id-J^\eta)$ allows to apply the first estimate \eqref{eq:reg-J-eta-1} to make use of the additional space regularity in order to compensate for the lack of time regularity. Correspondingly, the second estimate \eqref{eq:reg-J-eta-2} allows to use the additional time regularity in order to compensate for the lack of space regularity.

\rmk{Now bound the above as follows.
\begin{align*}
|I_1|+|I_2|+|I_6|&\lesssim G_{st}\,\omega_A(s,t)^\frac3p+\eta^{-1}\,G_{st}\,\omega_A(s,t)^\frac4p+G_{st}\,\omega_A(s,t)^\frac3p\ ,\\
|I_3|&\lesssim \omega_\mu^1(s,t)\omega_A(s,t)^{\frac2p}+\eta^{1-\lambda}\,\omega_\mu^2(s,t)\omega_A(s,t)^{\frac2p}\ ,\\
|I_7|&\lesssim \omega_\mu^1(s,t)\omega_A(s,t)^{\frac1p}+\eta^{2-\lambda}\,\omega_\mu^2(s,t)\omega_A(s,t)^{\frac1p}\ ,\\
|I_4|+|I_8|&\lesssim \eta^{-2} \,\omega_\natural(s,t)^\frac{3}{q}\omega_A(s,t)^{\frac2p}+\eta^{-1}\,\omega_\natural(s,t)^\frac{3}{q}\omega_A(s,t)^{\frac1p}\ ,\\
|I_5|+|I_9|+|I_{10}|&\lesssim {\eta}\, G_{st}\,\omega_A(s,t)^\frac2p+\eta^2\,G_{st}\,\omega_A(s,t)^\frac1p+{\eta}\,G_{st}\,\omega_A(s,t)^\frac2p \ .
\end{align*}
In order to balance the various terms, we choose 
$$\eta=\omega_A(I)^{-\frac{1}{p}+\kappa} \omega_A(s,t)^{\frac{1}{p}-\kappa} \in (0,1) \ ,$$ 
where $\kappa\in[ 0,\tfrac{1}{p})$ is the parameter picked along (\ref{kappa}). Assuming that $\omega_A(I)\leq 1$ we deduce
\begin{align*}
|\delta g^\natural(\vp)_{sut}|&\lesssim G_{st}\,\omega_A(s,t)^\frac3p+\omega_\mu^1(s,t)\omega_A(s,t)^{\frac1p}+\omega_\mu^2(s,t)\omega_A(s,t)^{\frac{3-\lambda}{p}+\kappa(\lambda-2)}\\
&\quad+2\omega_A(I)^{\frac{1}{p}+\kappa}\omega_\natural(s,t)^\frac{3}{q}\omega_A(s,t)^{\kappa}+G_{st}\,\omega_A(I)^{-2(\frac{1}{p}+\kappa)}\omega_A(s,t)^{\frac{3}{p}-2\kappa}.
\end{align*}}
Note that there are only two terms where we kept track of $\omega_A(I)$, namely, the one that needs to be absorbed to the left hand side eventually, i.e. the one containing $\omega_\natural$, and the one with a negative power. The latter one can be further estimated from above by a constant depending on $\bfA$ and $I$ if we assume that $I\neq\emptyset$ and the former one will then determine the value of the constant $L$ from the statement of the Theorem.
Consequently, recalling the definition \eqref{eq:def-omega-I} of $\omega_{\ast}$, we obtain
\rmk{\begin{equation*}
\|\delta g^{\natural}_{sut} \|_{E_{-3}} \lesssim \omega_{\ast}(s,t)^\frac{3}{q}+\omega_A(I)^{\frac{1}{p}+2\kappa}\omega_\natural(s,t)^\frac{3}{q} \ .
\end{equation*}
At this point, observe that the mapping $\omega_{\ast}$ defines a regular control. Indeed, on the one hand, the regularity of $\omega_{\ast}$ easily stems from the continuity of $\omega_A$. On the other hand, superadditivity is obtained from \cite[Exercise 1.9]{FV} by recalling that both $\omega_{A}$ and $\omega_{\mu}$ are controls and using condition (\ref{kappa}), which can also be expressed as
$$\frac{q}{3}\left(\frac{3}{p}-2\kappa\right)\geq 1\ ,\qquad \frac{q}{3}+\frac{q}{3}\left(\frac{3-\lambda}{p}+\kappa (\lambda-2)\right)\geq1 \ .$$
Since $\omega_\natural$ is also a regular control, $\delta g^{\natural}$ satisfies the assumptions of Lemma~\ref{lemma-lambda} and we can thus conclude that, for all $s<t\in I$,
$\|g^{\natural}_{st}\|_{E_{-3}} \leq \omega'_\natural(s,t)^{\frac{3}{q}}
$
where 
$$
 \omega'_\natural(s,t) := C_{q} (\omega_\ast(s,t)+\omega_A(I)^{\frac{q}{3}(\frac{1}{p}+2\kappa)}\omega_\natural(s,t))
$$
is a new control. Let us define $L>0$ through the relation $C_{q} L^{\frac{q}{3}(\frac{1}{p}+2\kappa)}=\frac12$, so that if the interval $I$ satisfies $\omega_A(I)\leq \min(1, L)$, the above reasoning yields, for all $s<t\in I$,
$$\|g^{\natural}_{st}\|_{E_{-3}}^\frac{q}{3}\leq C_{q}\, \omega_\ast(s,t)+\frac12 \omega_\natural(s,t) \ .$$
Iterating the procedure (on $I$ such that $\omega_A(I)\leq \min(1,L)$), we get that for all $s<t\in I$ and $n\geq 0$,
$$\|g^{\natural}_{st}\|_{E_{-3}}^\frac{q}{3}\leq C_{q}\, \omega_\ast(s,t) \Big( \sum_{i=0}^n 2^{-i}\Big)+2^{-(n+1)} \omega_\natural(s,t)  \ .$$
By letting $n$ tend to infinity, we obtain the desired estimate \eqref{apriori-bound}.}
\end{proof}

\rmk{With Theorem \ref{theo:apriori} in hand, let us introduce the second main ingredient of our strategy toward a priori bounds for equation \eqref{eq:gen2}: the Rough Gronwall Lemma. In brief, and just as its classical counterpart, this property will allow us to turn local affine-type estimates (for the increments of a path) into a global uniform bound.}

\begin{lemma}[Rough Gronwall Lemma]\label{lemma:basic-rough-gronwall}
Fix a time horizon $T>0$ and let $G:[0,T] \to [0,\infty)$ be a path such that for some constants $C,L>0$, $\kappa\geq 1$ and some controls $\omega_1,\omega_2$ on $[0,T]$ with $\omega_1$ being regular, one has
\begin{equation}\label{eq:a-priori-g}
\delta G_{st} \leq C \Big(\sup_{0\leq r\leq t} G_r \Big) \, \omega_1(s,t)^{\frac{1}{\kappa}} + \omega_2(s,t),
\end{equation}
for every $s<t\in [0,T]$ satisfying $\omega_1(s,t) \leq L$. Then it holds
$$\sup_{0\leq t\leq T} G_t \leq 2 \exp\Big( \frac{\omega_1(0,T)}{\al L}\Big) \cdot \Big\{ G_0+\sup_{0\leq t\leq T}\Big(\omega_2(0,t) \, \exp\Big(-\frac{\omega_1(0,t)}{\al L}\Big)\Big) \Big\} ,$$
where $\al$ is defined as 
\begin{equation}\label{eq:def-alpha}
\al=\min\left(1, \, \frac{1}{L (2Ce^2)^{\kappa} }\right).
\end{equation}
\end{lemma}

\rmk{\begin{remark}
We are aware that similar Gronwall-type properties have already been used in the rough or fractional literature, especially when dealing with linear problems (see for instance \cite[Theorem 3.1 (ii)]{hu-nualart-2007} or \cite[Section 10.7]{FV}). Nevertheless, we have found it important to have a clear statement of this result at our disposal, and we will refer to it several times in the sequel.
\end{remark}}

\begin{proof}
Let us successively set, for every $t\in [0,T]$,
\begin{equation}\label{eq:def-sup-g-h}
G_{\leq t}:=\sup_{0\leq s\leq t} G_s  , \quad H_t:=G_{\leq t} \exp\Big(-\frac{\omega_1(0,t)}{\al L}\Big) \quad \text{and} \quad H_{\leq t}:=\sup_{0\leq s\leq t} H_s .
\end{equation}
Also, let us denote by $K$ the integer such that $\al L (K-1)\leq \omega_1(0,T) \leq \al L K$, and define a set of times $t_0<t_1 <\cdots <t_K$ as follows: $t_0:=0$, $t_K:=T$ and for $k\in \{1,\ldots,K-1\}$, $t_k$ is such that $\omega_1(0,t_k)=\al Lk$. In particular, $\omega_1(t_k,t_{k+1})\leq \al L\leq L$ (recall that we have chosen $\al\leq 1$ in~\eqref{eq:def-alpha}).
Fix $t\in [t_{k-1},t_k]$, for some $k\in \{1,\ldots,K\}$.
 We start from the trivial decomposition
\begin{equation*}
(\delta G)_{0t} = \sum_{i=0}^{k-2} (\delta G)_{t_it_{i+1}}+(\delta G)_{t_{k-1}t}.
\end{equation*}
Then on each interval $[t_i,t_{i+1}]$ one can apply the a priori bound \eqref{eq:a-priori-g}. Taking into account the facts that $\omega_1(t_k,t_{k+1})\leq \al L$ and that $\om_{2}$ is a super-additive function, we get
\begin{eqnarray}\label{proof-gron-1}
(\delta G)_{0t} 
&\leq& 
C (\al L)^{\frac{1}{\kappa}} \sum_{i=0}^{k-2} G_{\leq t_{i+1}}+\omega_2(0,t_{k-1})+C (\al L)^{\frac{1}{\kappa}} G_{\leq t}+\omega_2(t_{k-1},t) \notag \\ 
&\leq&  C (\al L)^{\frac{1}{\kappa}} \sum_{i=0}^{k-1} G_{\leq t_{i+1}}+\omega_2(0,t) .
\end{eqnarray}
Let us bound the term $\sum_{i=0}^{k-1} G_{\leq t_{i+1}}$ above. According to our definitions \eqref{eq:def-sup-g-h}, we have 
\begin{equation}\label{proof-gron-2}
\sum_{i=0}^{k-1} G_{\leq t_{i+1}}=\sum_{i=0}^{k-1} H_{t_{i+1}} \exp\Big( \frac{\omega_1(0,t_{i+1})}{\al L}\Big) \leq H_{\leq T} \sum_{i=0}^{k-1} \exp(i+1) \leq  \exp(k+1)\, H_{\leq T}  ,
\end{equation}
where we have used the  fact that $\om_{1}(0,t_{i+1})\leq \al L (i+1)$ for the first inequality.
Combining~(\ref{proof-gron-1}) and (\ref{proof-gron-2}), we thus get that
$$G_{\leq t} \leq G_0+\omega_2(0,t)+ C (\al L)^{\frac{1}{\kappa}} \exp(k+1)\, H_{\leq T}  .$$
Now,
\begin{eqnarray*}
H_t &=& G_{\leq t} \exp\Big(-\frac{\omega_1(0,t)}{\al L} \Big)\\
&\leq& \{G_0+\omega_2(0,t)\} \exp\Big(-\frac{\omega_1(0,t)}{\al L} \Big)+ C (\al L)^{\frac{1}{\kappa}} \exp(k+1)\, H_{\leq T}\exp\Big(-\frac{\omega_1(0,t_{k-1})}{\al L} \Big) ,
\end{eqnarray*}
and since $\om(0,t_{k-1})= \al L (k-1)$, we end up with:
\begin{equation*}
H_t  \leq
G_0+\sup_{0\leq s\leq T}\Big(\omega_2(0,s) \, \exp\Big(-\frac{\omega_1(0,s)}{\al L}\Big)\Big)+ Ce^2 (\al L)^{\frac{1}{\kappa}} H_{\leq T} .
\end{equation*}
By taking the supremum over $t\in [0,T]$, we deduce that
$$H_{\leq T} \leq Ce^2 (\al L)^{\frac{1}{\kappa}} H_{\leq T}+G_0+\sup_{0\leq s\leq T}\Big(\omega_2(0,s) \, \exp\Big(-\frac{\omega_1(0,s)}{\al L}\Big)\Big)$$
and recalling the definition \eqref{eq:def-alpha} of $\al$, it entails that 
$$H_{\leq T} \leq 2G_0+2\sup_{0\leq s\leq T}\Big(\omega_2(0,s) \, \exp\Big(-\frac{\omega_1(0,s)}{\al L}\Big)\Big)  .$$
The conclusion is now immediate, since
$$G_{\leq T}=\exp\Big( \frac{\omega_1(0,T)}{\al L}\Big) H_T \leq \exp\Big( \frac{\omega_1(0,T)}{\al L}\Big) H_{\leq T}  .$$
\end{proof}

\subsection{\rmk{A first application: a priori estimates for a (rough) heat model}}\label{subsec:first-application}

\rmk{As a conclusion to this section, we would like to give an example of the possibilities offered by the combination of the two previous results (Theorem \ref{theo:apriori} and Lemma \ref{lemma:basic-rough-gronwall}), through an application to the linear heat equation (\ref{eq:heat}).}

\rmk{Note that this rough parabolic model (when $z$ stands for a $p$-rough path, wih $p\in [2,3)$) has already been considered in the literature (see for instance \cite{MR2765508}). Our aim here is not to provide a full treatment of the equation (which would certainly overlap existing wellposedness results), but only to illustrate some of the main ideas of our apprach, before we turn to the more sophisticated conservation-law model.}

\rmk{To be more specific, let us focus on proving a uniform energy estimate for the approximation of (\ref{eq:heat}), that is the sequence of (classical) equations
\begin{align}\label{eq:heat-approx}
\begin{aligned}
\dd u^\ep &=\Delta u^\ep\, \dd t + V\cdot \nabla u^\ep \,\dd z^\ep \ ,\qquad x\in\R^N \ ,\ t\in(0,T)\ ,\\
u^\ep_0&=u_0\in L^2(\R^N) \ ,
\end{aligned}
\end{align}
where $(z^\ep)_{\ep\in (0,1)}$ is a sequence of smooth paths that converges to a continuous (weak geometric) $p$-rough path $\mathbf{Z}=(Z^1,Z^2)$ (for some fixed $p\in [2,3)$), i.e. $\mathbf{Z}^\ep:=\Big(\delta z^\ep,\int \delta z^\ep \otimes \dd z^\ep\Big)  \to \mathbf{Z}$ as $\ep\to 0$ (say for the uniform topology). We also assume that
\begin{equation}\label{uniform-assumption-z}
\sup_{\ep >0} \Big\{ \big| Z^{1,\ep}_{st}\big|^p +\big| Z^{2,\ep}_{st}\big|^{\frac{p}{2}} \Big\} \leq \omega_{\mathbf{Z}}(s,t) \ ,
\end{equation}
for some regular control $\omega_{\mathbf{Z}}$. Note that, according to \cite[Proposition 8.12]{FV}, such a sequence $(z^\ep)$ can for instance be obtained through the geodesic approximation of $\mathbf{Z}$.} 

\smallskip

\rmk{Before we turn to a suitable \enquote{rough} treatment of (\ref{eq:heat-approx}), and for pedagogical purpose, let us briefly recall how the basic energy estimate is derived in the classical smooth case. In that situation, one formally tests \eqref{eq:heat-approx} by $\varphi=u^\ep_t$ and integration by parts leads to 
\begin{align}
 \|u^\ep_t\|_{L^2}^2+2\int_0^t\|\nabla u^\ep_r\|_{L^2}^2\,\dd r &=\|u_0\|_{L^2}^2+\int_0^t (u^\ep)^2_r(\diver V) \,\dd z^\ep_r\notag\\
 &\leq \|u_0\|_{L^2}^2+\|V\|_{W^{1,\infty}}\int_0^t \|u^\ep_r\|_{L^2}^2\,\dd |z^\ep_r| \ . \label{estim-toy}
\end{align}
Hence the (classical) Gronwall lemma applies and we obtain
$$ \|u^\ep_t\|_{L^2}^2+2\int_0^t\|\nabla u^\ep_r\|_{L^2}^2\,\dd r \leq \mathrm{e}^{\|V\|_{W^{1,\infty}}\|z^\ep\|_{1\text{-var}}}\|u_0\|_{L^2}^2 \ .$$}

\rmk{In the rough setting, these two elementary steps (namely, the estimate (\ref{estim-toy}) and then the use of the Gronwall lemma) will somehow be replaced with their rough counterpart: first, the a priori estimate provided by Theorem \ref{theo:apriori}, then the Rough Gronwall Lemma \ref{lemma:basic-rough-gronwall}.}

\rmk{In order to implement this strategy, consider the path $v^\ep:=(u^\ep)^2$, and observe that, for fixed $\ep>0$, this path is (trivially) governed by the dynamics
$$\dd v^\ep =2u^\ep\Delta u^\ep\, \dd t + V\cdot \nabla v^\ep \,\dd z^\ep \ .$$
Expanding the latter equation in its weak form (just as in \ref{eq:he1}) easily entails that for any test-function $\vp\in W^{3,\infty}(\R^N)$,
\begin{equation}\label{u2}
\delta  v^\ep(\varphi )_{s t} =(\delta\mu^\ep)_{st}(\vp)+ v^\ep_{s} (A^{1,\ep,\ast}_{s t} \varphi ) +  v^\ep_{s} (A^{2,\ep, \ast}_{s t} \varphi
     ) + v^{\ep,\natural}_{s t}( \varphi ) \ ,
\end{equation}
where the finite variation term $\mu^{\ep}$ is given by:
\begin{equation*}
(\delta\mu^\ep)_{st}(\vp):=-2\int_s^t \int_{\R^N} |\nabla u^\ep_r|^2 \vp\,\dd x\,\dd r-2\int_s^t\int_{\R^N} \nabla u^\ep_r \cdot\nabla\vp \,u^\ep_r\,\dd x\,\dd r \ ,
\end{equation*}
and where $A^{1,\ep,\ast},A^{2,\ep,\ast}$ are defined along (\ref{eq:def-A-star})-(\ref{eq:def-A-star-bis}) (by replacing $z$ with $z^\ep$). Eventually, the term $v^{\ep,\natural}$ in \eqref{u2} stands for some new \enquote{third-order} remainder acting on $W^{3,\infty}(\R^N)$.}

\rmk{Equation (\ref{u2}) is actually the starting point of our analysis, that is the equation to which we intend to apply the a priori estimate of Theorem \ref{theo:apriori} (or more simply Corollary \ref{cor:apriori} in this case). To this end, and as anticipated above, we consider the scale $E_n:=W^{n,\infty}(\R^N)$ ($0\leq n\leq 3$). The construction of a smoothing on this scale (in the sense of Definition~\ref{def:smoothing}) is an easy task: consider indeed any smooth, compactly-supported and rotation-invariant function $\jmath$ on $\R^N$ such that $\int_{\R^N} \jmath(x) \dd x=1$, and define $J^\eta$ as a convolution operator, that is 
\begin{equation}\label{smoothing-convolution}
J^\eta \vp (x):=\int_{\R^N} \jmath_\eta(x-y) \vp(y)\dd y\ , \quad \text{with} \ \  \jmath_\eta(x):=\eta^{-N} \jmath(\eta^{-1}x)\ .
\end{equation}
Checking conditions (\ref{eq:reg-J-eta-1})-(\ref{eq:reg-J-eta-2}) is then a matter of elementary computations, which we leave to the reader as an exercise.}

\rmk{As far as the drift term $\mu^\ep$ is concerned, observe that it can be estimated as
\begin{align}\label{eq:estim-mu-heat}
|(\delta\mu^\ep)_{st}(\vp)|\lesssim \bigg(\int_s^t \|\nabla u^\ep_r\|_{L^2}^2 \,\dd r\bigg)\|\vp\|_{L^\infty}+\bigg(\int_s^t\|\nabla u^\ep_r\|_{L^2}^2\,\dd r \bigg)^\frac{1}{2}\bigg(\int_s^t\| u^\ep_r\|_{L^2}^2\,\dd r \bigg)^\frac{1}{2}\|\vp\|_{W^{1,\infty}}\ ,
\end{align}
and hence the assumption \eqref{cond-mu-cor} holds true for the control given by
$$
\omega_{\mu^\ep}(s,t):=\int_s^t \|\nabla u^\ep_r\|_{L^2}^2 \,\dd r+\bigg(\int_s^t\|\nabla u^\ep_r\|_{L^2}^2\,\dd r \bigg)^\frac{1}{2}\bigg(\int_s^t\| u^\ep_r\|_{L^2}^2\,\dd r \bigg)^\frac{1}{2} \ .
$$
We are thus in a position to apply Corollary \ref{cor:apriori} and deduce the existence of a constant $L>0$ (independent of $\ep$) such that on any interval  $I\subset [0,T]$ satisfying $\omega_{{\bf Z}}(I)\leq L$, one has
\begin{equation}\label{bou-ep-nat}
\| v^{\ep,\natural}_{st}\|_{E_{-3}} \lesssim \, \cn[v^\ep;L^\infty(s,t;E_{-0})]\,\omega_{{\bf Z}}(s,t)^{\frac{3}{p}}+\omega_{\mu^\ep}(s,t) \omega_{{\bf Z}}(s,t)^{\frac{3-p}{p}}\quad (s<t\in I) \ ,
\end{equation}
for some proportional constant independent of $\ep$ (due to (\ref{uniform-assumption-z})).}

\rmk{In order to exploit the (non-uniform) bound (\ref{bou-ep-nat}), let us go back to (\ref{u2}) and apply the equation to the trivial test-function $\vp=1 \in E_3$, which immediately leads to
\begin{eqnarray*}
\lefteqn{(\delta\|u^\ep\|_{L^2}^2)_{st}+2\int_s^t\|\nabla u^\ep_r\|_{L^2}^2\,\dd r}\\
& = & v^\ep_s(A^{1,*}_{st}1)+v^\ep_s(A^{2,*}_{st}1)+v^{\ep,\natural}_{0 t}( 1 )\\
&\lesssim & \|u^\ep_s\|_{L^2}^2 \,\omega_{{\bf Z}}(s,t)^\frac1p+|v^{\ep,\natural}_{st}(1)|\\
&\lesssim &\sup_{s\leq r\leq t}\|u^\ep_r\|_{L^2}^2\bigg[\omega_{{\bf Z}}(s,t)^{\frac{1}{p}}+|t-s|\,\omega_{{\bf Z}}(s,t)^{\frac{3-p}{p}}\bigg]+\bigg(\int_s^t \|\nabla u^\ep_r\|_{L^2}^2 \,\dd r\bigg) \, \omega_{{\bf Z}}(s,t)^{\frac{3-p}{p}}\\
&\lesssim &\bigg[ \sup_{s\leq r\leq t}\|u^\ep_r\|_{L^2}^2 +\int_s^t \|\nabla u^\ep_r\|_{L^2}^2 \,\dd r\bigg]\,\omega_{{\bf Z}}(s,t)^{\frac{3-p}{p}}\ ,
\end{eqnarray*}
for all $s<t$ in a sufficiently small interval.}

\rmk{At this point, we are (morally) in the same position as in (\ref{estim-toy}). By applying our second main technical tool, namely the rough Gronwall Lemma \ref{lemma:basic-rough-gronwall}, with
$$G^\ep_t:=\|u^\ep_t\|_{L^2}^2+2\int_0^t \|\nabla u^\ep_r\|_{L^2}^2\,\dd r \quad , \quad\omega_1:=\omega_{{\bf Z}} \quad , \quad \omega_2:=0 \ ,$$
we finally obtain the desired uniform estimate
\begin{align}\label{estim-toy-approx}
\sup_{0\leq t\leq T}\|u^\ep_t\|_{L^2}^2+\int_0^T\|\nabla u^\ep_t\|_{L^2}^2\,\dd t\lesssim \exp\Big( \frac{\omega_1(0,T)}{\al L}\Big) \, \|u_0\|_{L^2}^2 \ ,
\end{align}
for some proportional constant independent of $\ep$.}

\rmk{\begin{remark}
Starting from the uniform a priori estimate (\ref{estim-toy-approx}), one could certainly settle a compactness argument and deduce the existence of a solution for the rough extension of equation (\ref{eq:heat-approx}) (interpreted through Definition \ref{def:gen}). But again, our aim here is not to give a full treatment of this heat-equation example, and we refer the reader to Section \ref{sec:existence} for more details on such a compactness argument in the (more interesting) rough conservation-law case.
\end{remark}}

\section{Tensorization and uniqueness}
\label{sec:uniq1}

\rmk{We now turn to the sketch of a strategy towards uniqueness for the general rough PDE (\ref{eq:gen}), understood in the sense of Definition \ref{def:gen}. These ideas will then be carefully implemented in the next sections for the rough conservation-law model.}

\subsection{Preliminary discussion}
\label{subsec:prelim}

\rmk{Let us  go back to the model treated in Section \ref{subsec:first-application} and recall that one of the key points of our strategy regarding (\ref{eq:heat-approx}) (and ultimately leading to (\ref{estim-toy-approx})) was the derivation of the equation satisfied by the squared-path $v^\ep=(u^\ep)^2$. Observe in particular that if (\ref{u2}) were to be true at the rough level, that is above the rough path $\mathbf{Z}=(Z^1,Z^2)$ and not only above its approximation $z^\ep$ (with $u$ accordingly replacing $u^\ep$), then the very same arguments could be used to show that estimate (\ref{estim-toy-approx}) actually holds true for {\it any} solution of the underlying rough equation. The desired uniqueness property for this equation would immediately follow, due to the linearity of the problem.} 

\smallskip

\rmk{Unfortunately, when working directly at the rough level, establishing such an equation for the squared-path $u^2$ turns out to be a complicated exercise, due to the fact that $u$ cannot be considered as a test function anymore. Therefore new ideas are required for the uniqueness result. In fact, let us consider the following more general formulation of the problem (which will also encompass our strategy toward uniqueness for the conservation-law model): if $u$, resp. $v$,  is a (functional-valued) solution of the rough equation
\begin{equation}\label{eq:gen-tens}
\dd u_{t} = \mu(\dd t)+ {\bf A}(\dd t) u_{t}  \quad , \quad \text{resp.} \quad \dd v_{t} = \nu(\dd t)+ {\bf A}(\dd t) v_{t}  \ ,
\end{equation}
for a same unbounded rough driver ${\bf A}$ (but possibly different drift terms $\mu,\nu$), then what is the equation satisfied by the product $uv$?}

\rmk{In order to answer this question (at least in some particular situations), we shall follow the ideas of \cite{BG} and rely on a tensorization argument}, together with a refined analysis of the approximation error. To be more specific, starting from \eqref{eq:gen-tens}, we exhibit first the equation for the tensor product of distributions \rmk{$U(x,y):=(u\otimes v)(x,y)=u(x)v(y)$.} Namely, write
\rmk{\begin{equation}\label{eq:tensor-heat-1}
\delta U_{st} = \delta u_{st} \otimes v_{s} +  u_{s} \otimes \delta v_{st} 
+ \delta u_{st} \otimes \delta v_{st} \ ,
\end{equation}
and then expand the increments $\delta u_{st},\delta v_{st}$ along \eqref{eq:he1}, which, at a formal level, yields the decomposition
\begin{equation}\label{eq:tensor}
\delta U_{st}=\delta M_{st} + \Gamma^1_{st}U_{s}+\Gamma^2_{st}U_{s}+U^{\natural}_{st}\ ,
\end{equation}
where the finite variation term $M$ can be expressed as:
\begin{equation}
M_t :=\int_{[0,t]} \mu_{dr}\otimes v_r+\int_{[0,t]} u_r\otimes \nu_{dr} \ ,
\end{equation}
and where the tensorized rough drivers $\Gamma^{1},\Gamma^{2}$ are given by:
\begin{equation}\label{driver:tens}
 \Gamma^1_{st}:=A^1_{st}\otimes\mathbb{I}+\mathbb{I}\otimes A^1_{st}\qquad ,\qquad\Gamma^2_{st}:=A^2_{st}\otimes\mathbb{I}+\mathbb{I}\otimes A^2_{st}+A^1_{st}\otimes A^1_{st}\ .
\end{equation}
In equations \eqref{eq:tensor} and \eqref{driver:tens},} $\mathbb{I}$ denotes the identity map and $U^{\natural}$ is a remainder when tested with smooth functions of the two variables. 

\rmk{ An easy but important observation is that Chen's relation~\eqref{eq:chen-relation} is again satisfied by the components of $\mathbf{\Gamma}:=(\Gamma^1,\Gamma^2)$. Equation \eqref{eq:tensor} still fits the pattern of (\ref{eq:gen2}), and is therefore likely to be treated with the same tools as the original equations (\ref{eq:gen-tens}), that is along our a priori estimate strategy.}

\smallskip

\rmk{Our goal then is to test the tensorized equation (\ref{eq:tensor}) against functions of the form
\begin{equation}\label{testf}
\Phi_\varepsilon(x,y)=\ep^{-N}\vp\left(\frac{x+y}{2}\right)\psi(\ep^{-1}(x-y)) \ ,
\end{equation}
and try to derive, with the help of Theorem \ref{theo:apriori}, an $\ep$-uniform estimate for the resulting expression. Such an estimate should indeed allow us to pass to the limit as $\varepsilon\to0$ (or in other words, to \enquote{pass to the diagonal}) and obtain the desired equation for $uv$.}
To this end, we consider the blow-up transformation $T_\varepsilon$ defined on test functions as 
\begin{equation}\label{blow-up-trans}
T_\varepsilon \Phi(x,y) :=\varepsilon^{-N} \Phi\big(x_+ + \frac{x_-}{\varep}, x_+ - \frac{x_-}{\varep}\big) \ ,
\end{equation}
where $x_{\pm}=\frac{x\pm y}{2}$ are the coordinates parallel and transverse to the diagonal.
Note that its adjoint for the $L^2$-inner product reads
\begin{equation}\label{eq:def-T-eps-star}
T^*_\varepsilon \Phi (x, y) = \Phi (x_+ +\varepsilon x_- , x_+ -\varepsilon x_-)
\end{equation}
and its inverse is given by
\begin{equation}\label{eq:def-T-eps-inverse}
T^{-1}_\varepsilon\Phi(x,y)=\varepsilon^N\Phi(x_++\varepsilon x_-,x_+-\varepsilon x_-) \ .
\end{equation}
\rmk{Setting
$$
U^\varepsilon := T^{\ast}_\varepsilon U  \quad , \quad \mathbf{\Gamma}^{\ast}_{\varepsilon} := T_\varepsilon^{-1}\mathbf{\Gamma}^{\ast}T_\varepsilon \quad , \quad M^\varepsilon := T^{\ast}_\varepsilon M  \quad \text{and}\quad  U^{\natural,\varepsilon}:=T_\varepsilon^* U^\natural \ ,$$
the tensorized equation \eqref{eq:tensor} readily turns into 
\begin{equation}\label{eq:U-eps}
\delta U^\varepsilon_{st}( \Phi) =\delta M_{st}^\varepsilon ( \Phi) +U_s^\varepsilon (( \Gamma_{\varepsilon,st}^{1,\ast} + \Gamma_{\varepsilon,st}^{2,\ast}) \Phi)  +  U^{ \natural,\varepsilon}_{st}(\Phi) \ .
\end{equation}
Applying (\ref{eq:U-eps}) to the test-function $\Phi(x,y)=\vp(x_+) \psi(2x_-)$ then corresponds to applying~\eqref{eq:tensor} to the test-function $\Phi_{\varepsilon}$ defined by (\ref{testf}).}

\rmk{With these notations in mind, our search of an $\ep$-uniform estimate for $U^{ \natural,\varepsilon}$ (via an application of Theorem \ref{theo:apriori}) will naturally give rise to the central notion of {\it renormalizability} for the driver under consideration.}

\subsection{Renormalizable drivers}\label{sec:renormalizable-drivers}

Let us fix $p\in [2,3]$ for the rest of the section. \rmk{Motivated by the considerations of Section \ref{subsec:prelim}, we will now define and illustrate the concept of renormalizable rough driver.}

\begin{definition}[Renormalizable driver]\label{def-renorm-tensor}
Let $\bfA$ be a continuous unbounded $p$-rough driver \rmk{acting on $C^\infty_c(\R^N )$}, and $\mathbf{\Gamma}$ its tensorization defined by \eqref{driver:tens} and acting on $C^\infty_c(\R^N \times \R^N)$. We say that $\bfA$ is renormalizable in \rmk{a} scale of spaces \rmk{$(\mathcal E_n)_{0\leq n\leq 3}$}  if $\{\mathbf{\Gamma}_\varepsilon\}_{\varepsilon\in(0,1)}$ can be extended to a bounded family \rmk{(with respect to $\ep$)} of continuous unbounded $p$-rough drivers on this scale.
\end{definition}

\rmk{\begin{remark}
Although it is inspired by the equation (\ref{eq:U-eps}) governing $U^\ep$, this definition only depends on the driver $\bfA$ and not on the drift terms $\mu,\nu$ involved in (\ref{eq:gen-tens}).  
\end{remark}}

\rmk{\begin{remark}
In the context of transport-type rough drivers, the renormalization property corresponds to the fact that a commutator lemma argument in the sense of DiPerna-Lions \cite{MR1022305} can be performed.
\end{remark}}

\rmk{For a clear illustration of this property, let us slightly anticipate the next sections and consider the case of the driver that will govern our rough conservation-law model, namely $\bfA=(A^1,A^2)$ with
\begin{equation}\label{driver-tens-sec}
A^1_{st}u:=Z^{1,k}_{st} \,V^k \cdot\nabla u\ ,\qquad A^2_{s t}u:= Z^{2,jk}_{st}\, V^k\cdot\nabla (V^j\cdot\nabla u) \ ,
\end{equation} 
or equivalently
\begin{equation}\label{eq:def-A-star-tens-sec}
A_{st}^{1,\ast}\vp
=- Z_{st}^{1,k} \, \text{div}( V^{k}\vp )
\quad\text{and}\quad
A_{st}^{2,\ast}\vp
= Z_{st}^{2,jk} \, \text{div}( V^{j} \text{div}( V^{k}\vp )) \ ,
\end{equation}
for a given $p$-rough path $\mathbf{Z}=(Z^1,Z^2)$ in $\R^K$ and a family of vector fields $V=(V^1,\dots,V^K)$ on $\R^N$. Observe that this driver was already at the core of the heat-equation model evoked in Section \ref{sec:urd} (or in Section \ref{subsec:first-application}).}

\begin{proposition}\label{prop:renorm}
\rmk{Let $\mathbf{A}$ be the continuous unbounded $p$-rough driver defined by \eqref{driver-tens-sec}, and assume that $V\in W^{3,\infty}(\R^N)$. Then, for every $1\leq R\leq \infty$, $\mathbf{A}$ is a renormalizable driver in the scale $(\mathcal E_{R,n})_{0\leq n\leq 3}$ given by
\begin{equation}\label{scale-e-r-n}
\mathcal E_{R,n}:=\left\{ \Phi\in W^{n,\infty}(\R^{N}\times\R^{N});\,\Phi({x},{y})=0\;\text{if}\;\rho_R({x},{y})\ge 1 \right\} \ ,
\end{equation}
where $\rho_R( x, y)^2 = |{x}_+|^2/R^2 +  |{x}_-|^2$, and equipped with the subspace topology of $W^{n,\infty}$.
Besides, it holds that
\begin{equation}\label{prop:renorm-1}
\lVert \Gamma^{1}_{\ep,st}\rVert_{\cl(\mathcal E_{R,n},\mathcal E_{R,n-1})}^p \lesssim_{\|V\|_{W^{3,\infty}}} \omega_Z(s,t)\ ,\qquad n\in \{-0,-2\}\ ,
\end{equation}
\begin{equation}\label{prop:renorm-2}
\lVert \Gamma^{2}_{\ep,st}\rVert_{\cl(\mathcal E_{R,n},\mathcal E_{R,n-2})}^{p/2} \lesssim_{\|V\|_{W^{3,\infty}}} \omega_Z(s,t)\ ,\qquad n\in \{-0,-1\} \ ,
\end{equation}
for some proportional constants independent of both $\ep$ and $R$.}
\end{proposition}

\begin{remark}
The support condition in the definition of spaces $\mathcal E_{R,n}$ \rmk{implies in particular that the test functions are compactly supported in the $x_-$ direction. \rmk{This localization is a} key point in the proof below. As we mentionned it above, the test functions we are ultimately interested in are those of the form \eqref{testf}, i.e. the dependence on $x_-$ is only in the mollifier $\psi$, which can indeed be taken compactly supported.}
\end{remark}

\begin{remark}\label{rk:local-x-plus}
\rmk{In the subsequent conservation-law model, the possibility of a specific localization in the $x_+$ direction (as offered by the additional parameter $R\geq 1$) will turn out to be an important technical tool when looking for suitable estimates of $U^{\ep}$ and $M^{\ep}$ (along the notations of (\ref{eq:U-eps})), as detailled in Sections \ref{subsec:f-ep}-\ref{subsec:q-ep}.}
\end{remark}

\begin{proof}[Proof of Proposition \ref{prop:renorm}]
Recall that the driver $\mathbf{\Gamma}$ was defined in \eqref{driver:tens}, and that $A_{st}^{1,\ast}, A_{st}^{2,\ast}$ are defined by \eqref{eq:def-A-star-tens-sec}. Therefore, the driver $\mathbf{\Gamma}_{\varepsilon}$ which was defined by $\mathbf{\Gamma}_{\varepsilon}^*=T_\ep^{-1}\mathbf{\Gamma}^*T_\ep$ can be written as
$$\Gamma^{1,*}_{\varepsilon,st}=Z_{st}^{1,i}\Gamma^{1,i,*}_{V,\ep},\qquad\Gamma^{2,*}_{\varepsilon,st}=Z_{st}^{2,ij}\Gamma^{1,j,*}_{V,\ep}\Gamma^{1,i,*}_{V,\ep},$$
where 
\[ \Gamma_{V, \varepsilon }^{1,*} := -
   V^{+}_{\varepsilon} \cdot\nabla^{+} - \varepsilon^{-1} V^{-}_{\varepsilon}
   \cdot\nabla^{-} -D^{+}_{\varepsilon}\ . 
\]
We have here used the notation
$$\nabla^\pm:=\frac{1}{2}(\nabla_{{x}}\pm\nabla_{{y}})\ ,\qquad D({x})=\diver_{{x}} V({x})\ ,$$
and for any real-valued function $\Psi$ on $\R^{N}$ 
$$\Psi^{\pm}_{\varepsilon} ({x}, {y}) := \Psi
({x}_+ + \varepsilon {x}_-) \pm \Psi
({x}_+ - \varepsilon {x}_-) \ .$$
\rmk{From these expressions, it is clear that neither $\Gamma^{1,*}_{\varepsilon,st}$ nor $\Gamma^{2,*}_{\varepsilon,st}$ influence the support of the test-functions, and so the operators $\Gamma^{1,*}_{\varepsilon,st}: \mathcal{E}_{R,n} \to \mathcal{E}_{R,n-1}$ and $\Gamma^{2,*}_{\varepsilon,st}: \mathcal{E}_{R,n} \to \mathcal{E}_{R,n-2}$ are indeed well-defined.}
Next, by the Taylor formula we obtain
\begin{align*}
 \varepsilon^{- 1} V^-_{\varepsilon} &= 2\, x_-\int_0^1 \mathrm D V\big((x_+-\varepsilon x_-)+2\varepsilon r x_-\big)\,\dd r.
\end{align*}
Hence, since \rmk{any function $\Phi\in \mathcal{E}_{R,n}$ satisfies $\Phi(x,y)=0$ as soon as $|x_-|\geq 1$, we obtain for $n=0,1,2$
\[ \| \Gamma^{1,*}_{V, \varepsilon} \Phi \|_{\mathcal{E}_{R,n}} \lesssim(\| V
   \|_{W^{n, \infty}} + \| \mathrm{D} V \|_{W^{n, \infty}}) \| \Phi \|_{\mathcal{E}_{R,n+1}} \lesssim \| V
   \|_{W^{n+1, \infty}} \| \Phi \|_{\mathcal{E}_{R,n+1}}\ ,\]
and for $n=0,1$
\begin{align*}
\| \Gamma^{2,*}_{V, \varepsilon} \Phi \|_{\mathcal{E}_{R,n}}= \| \Gamma^{1,*}_{V, \varepsilon} \Gamma^{1,*}_{V, \varepsilon} \Phi \|_{\mathcal{E}_{R,n}}&\lesssim \| V
   \|_{W^{n+1, \infty}} \|\Gamma^{1,*}_{V, \varepsilon} \Phi \|_{\mathcal{E}_{R,n+1}}\\
	&\lesssim  \| V
   \|_{W^{n+1, \infty}} \| V
   \|_{W^{n+2, \infty}}\| \Phi \|_{\mathcal{E}_{R,n+2}}\ ,
\end{align*}
which holds true uniformly in $\varepsilon$. Consequently, uniformly in $\varepsilon$ (and of course $R$),
$$\lVert \Gamma^{1}_{\ep,st}\rVert_{\cl(\mathcal E_{R,n},\mathcal E_{R,n-1})}^p \lesssim_{\|V\|_{W^{3,\infty}}} \omega_Z(s,t),\qquad n\in \{-0,-2\},$$
$$ \lVert \Gamma^{2}_{\ep,st}\rVert_{\cl(\mathcal E_{R,n},\mathcal E_{R,n-2})}^{p/2} \lesssim_{\|V\|_{W^{3,\infty}}} \omega_Z(s,t),\qquad n\in \{-0,-1\} ,$$}
where $\omega_Z$ is a control corresponding to the rough path $\mathbf{Z}$, namely,
$$|Z_{st}^1|\leq \omega_Z(s,t)^\frac1p,\qquad |Z^2_{st}| \leq \omega_Z(s,t)^\frac2p.$$
\end{proof}

\section{Rough conservation laws I: Presentation}
\label{sec:conservation-presentation}
Throughout the remainder of the paper, we are interested in a rough path driven scalar conservation law of the form
\begin{align}\label{eq1}
\begin{aligned}
\dif u+\diver\big(A(x,u)\big)\dif z&=0,\qquad t\in(0,T),\,x\in\R^N,\\
u(0)&=u_0,
\end{aligned}
\end{align}
where $z=(z^{1},\dots,z^{K})$ and $z$ can be lifted to a geometric $p$-rough path, $A:\R^N_x\times\R_\xi\rightarrow\R^{N\times K}$. Using the Einstein summation convention, \eqref{eq1} rewrites
\begin{equation*}
\begin{split}
\dif u+\partial_{x_i}\big(A_{ij}(x,u)\big)\,\dif z^j&=0,\qquad t\in(0,T),\,x\in\R^N,\\
u(0)&=u_0.
\end{split}
\end{equation*}

As the next step, let us introduce the kinetic formulation of \eqref{eq1} as well as the basic definitions concerning the notion of kinetic solution. We refer the reader to \cite{perth} for a detailed exposition. The motivation behind this approach is given by the nonexistence of a strong solution and, on the other hand, the nonuniqueness of weak solutions, even in simple cases. The idea is to establish an additional criterion -- the kinetic formulation --
which is automatically satisfied by any strong solution to \eqref{eq1} (in case it exists) and which permits to ensure the well-posedness. The linear character of the kinetic formulation simplifies also the analysis of the remainder terms and the proof of the apriori estimates.
It is well-known that in the case of a smooth driving signal $z$, the kinetic formulation of \eqref{eq1}, which describes the time evolution of $f_{t}(x,\xi)=\1_{u_{t}(x)>\xi}$, reads as
\begin{align}\label{eq:kinform}
\begin{split}
\dif f+\nabla_x f\cdot a\,\dif z-\partial_\xi f\, b\,\dif z&=\partial_\xi m,\\
f(0)&=f_0,
\end{split}
\end{align}
where the coefficients $a,b$ are given by
$$a=(a_{ij})=(\partial_\xi A_{ij}):\R^N\times\R\to\R^{N\times K},\qquad b=(b_j)=(\diver_x A_{\cdot j}):\R^N\times\R\to \R^K$$
and $m$ is a nonnegative finite measure on $[0,T]\times\R_x^{N}\times\R_\xi$ which becomes part of the solution. The measure $m$ is called kinetic defect measure as it takes account of possible singularities of $u$. Indeed, if there was a smooth solution to \eqref{eq1} then one can derive \eqref{eq:kinform} rigorously with $m\equiv 0.$ We say then that $u$ is a kinetic solution to \eqref{eq1} provided, roughly speaking, there exists a kinetic measure $m$ such that the pair $(f=\1_{u>\xi},m)$ solves \eqref{eq:kinform} in the sense of distributions on $[0,T)\times\R^N_x\times\R_\xi$.

In the case of a rough driver $z$, we will give an intrinsic notion of kinetic solution to~\eqref{eq1}. In particular, the kinetic formulation \eqref{eq:kinform} will be understood in the framework of unbounded rough drivers presented in the previous sections.
The reader can immediately observe that~\eqref{eq:kinform} fits very naturally into this concept: the left hand side of \eqref{eq:kinform} is of the form of a rough transport equation whereas the kinetic measure on the right hand side plays the role of a drift. Nevertheless, one has to be careful since the kinetic measure is not given in advance, it is a part of the solution and has to be constructed within the proof of existence. Besides, in the proof of uniqueness, one has to compare two solutions with possibly different kinetic measures.

The kinetic formulation \eqref{eq:kinform} can be rewritten as
\begin{align*}
\dif f= \left(\begin{array}{cc}
											b\\
											-a
											\end{array}\right)\cdot\left(\begin{array}{c}
			\partial_\xi f\\
			\nabla_x f
			\end{array}\right)\dif z+\partial_\xi m
\end{align*}
or
\begin{align}\label{eq:kinform-2}
\dif f=V\cdot\nabla_{\xi,x}f\,\dif z+\partial_\xi m,
\end{align}
where the family of vector fields $V$ is given by
\begin{equation}\label{eq:def-V-conserv}
V=(V^1,\dots,V^{K})=\left(\begin{array}{cc}
											b\\
											-a
											\end{array}\right)=\left(\begin{array}{cccccc}
											b_1 &\cdots & b_K \\
											-a_{11}&\cdots& -a_{1K}\\
											\vdots&\cdots&\vdots\\
											-a_{N1}&\cdots& -a_{NK}\\
											\end{array}\right).
\end{equation}
Note that these vector fields satisfy for $i\in\{1,\dots, K\}$
\begin{equation}\label{divfree}
\diver_{\xi,x} V_i=\partial_\xi b_i-\diver_x a_{\cdot i}=\partial_\xi \diver_x A_{\cdot i}-\diver_x \partial_\xi A_{\cdot i}=0.
\end{equation}

Let us now label the assumptions we will use in order to solve equation \eqref{eq:kinform} or its equivalent form~\eqref{eq:kinform-2}, beginning with the assumptions on $V$:

\begin{hypothesis}\label{hyp:V}
Let $V$ be the family of vector fields defined by relation \eqref{eq:def-V-conserv}. We assume that:
\begin{equation}\label{H}
V\in W^{3,\infty}(\R^{N+1})\qquad\text{and}\qquad V(x,0)=0\qquad\forall x\in\R^N.
\end{equation}
\end{hypothesis}

Notice that the assumption $V(x,0)=0$ is only used for the a priori estimates on solutions of \eqref{eq:kinform}, so that it will not show up before Section \ref{sec:apriori}. In addition,
as in the toy heat model case, we shall also assume that $z$ can be understood as a rough path. 

\begin{hypothesis}\label{hyp:z}
\rmk{For some fixed $p\in [2,3)$, the function $z$ admits a lift to a continuous (weak geometric) $p$-rough path $\mathbf{Z}=(Z^1,Z^2)$ on $[0,T]$ (in the sense of Definition \ref{def-p-rough-path}), and we fix a regular control $\omega_Z$ such that for all $s<t\in [0,T]$
$$
|Z_{st}^1|\leq \omega_Z(s,t)^\frac1p\quad ,\quad |Z^2_{st}| \leq \omega_Z(s,t)^\frac2p \ .
$$}
\end{hypothesis}

\rmk{Endowed with the $p$-rough path $\mathbf{Z}=(Z^1,Z^2)$, we can turn to the presentation of the rough driver structure related to our equation (\ref{eq:kinform-2}). The scale of spaces \rmk{$(E_n)_{0\leq n\leq 3}$} where this equation will be considered is 
\begin{equation*}
E_n=W^{n,1}(\R^{N+1})\cap W^{n,\infty}(\R^{N+1})\ .
\end{equation*}
Then we define the central operator-valued paths as follows: for all $s<t\in [0,T]$ and $\vp\in E_1$ (resp. $\vp\in E_2$),}
\begin{align}\label{d:urd}
\begin{aligned}
A^1_{st}\varphi:&=Z^{1,i}_{st} \, V^i\cdot\nabla_{\xi,x}\varphi \ ,\\
\text{resp.} \quad A^2_{st}\varphi:&=Z^{2,ij}_{st}\, V^j\cdot\nabla_{\xi,x}( V^i\cdot\nabla_{\xi,x}\varphi)\ .
\end{aligned}
\end{align}
\rmk{It is readily checked that $\mathbf{A}:=(A^1,A^2)$ defines a continuous unbounded $p$-rough driver on $(E_n)_{0\leq n\leq 3}$, and that for all $s<t\in [0,T]$,}
\begin{align}\label{eq:urdest}
\begin{aligned}
\lVert A^{1}_{st}\rVert_{\cl( E_n, E_{n-1})}^p &\lesssim \omega_Z(s,t),\qquad n\in \{-0,-2\},\\
 \lVert A^{2}_{st}\rVert_{\cl( E_n, E_{n-2})}^{p/2} &\lesssim \omega_Z(s,t),\qquad n\in \{-0,-1\} .
\end{aligned}
\end{align}
As in Section \ref{sec:urd}, it will be useful to have the expression of $A^{1,\ast}$ and $A^{2,\ast}$ in mind for our computations. Here it is readily checked that:
\begin{equation}\label{eq:def-A-star-conserv}
A_{st}^{1,\ast}\vp
=- Z_{st}^{1,k} \, \text{div}_{x,\xi}\lp V^{k}\vp \rp,
\quad\text{and}\quad
A_{st}^{2,\ast}\vp
= Z_{st}^{2,jk} \, \text{div}_{x,\xi}\lp V^{j} \text{div}_{x,\xi}\lp V^{k}\vp \rp \rp .
\end{equation}

The last point to be specified in order to include \eqref{eq:kinform} in the framework of unbounded rough drivers, is how to understand the drift term given by the kinetic measure $m$. To be more precise, in view of Subsection \ref{subsec:apr}, one would like to rewrite \eqref{eq:kinform} as
\begin{equation}\label{eq:LR}
\delta f_{st} = \delta \mu _{st} + (A^1_{st}+A^2_{st})f_s +f^{\natural}_{st}
\end{equation}
where $\delta\mu$ stands for the increment of the corresponding kinetic measure term and $f^\natural$ is a suitable remainder. However, already in the smooth setting such a formulation can only be true for a.e. $s,t\in[0,T]$. Indeed, the kinetic measure contains shocks of the kinetic solution and thus it is not absolutely continuous with respect to the Lebesgue measure. The atoms of the kinetic measure correspond precisely to singularities of the solution. Therefore, it makes a difference if we \rmk{define the drift term $\mu_t(\vp)$ as $-m(\1_{[0,t]}\partial_\xi\varphi)$ or $-m(\1_{[0,t)}\partial_\xi\varphi)$.} According to the properties of functions with bounded variation, the first one is right-continuous whereas the second one is left-continuous. Furthermore, they coincide everywhere except on a set of times which is at most countable.
Note also that the rough integral $f^\flat$ defined by
$$
\delta f^\flat_{st} = (A^1_{st}+A^2_{st})f_s +f^{\natural}_{st}
$$
 is expected to be continuous in time. Thus, depending on the chosen definition of $\mu$, we obtain either right- or left-continuous representative of the class of equivalence $f$ on the left hand side of \eqref{eq:LR}. These representatives will be denoted by $f^+$ and $f^-$ respectively.

For the sake of completeness, recall that in the deterministic (as well as stochastic) setting, the kinetic formulation \eqref{eq:kinform} is understood in the sense of distributions on $[0,T)\times\R^{N+1}$. That is, the test functions depend also on time. Nevertheless, for our purposes it seems to be more convenient to consider directly the equation for the increments $\delta f^\pm_{st}$. Correspondingly, we include two versions of~\eqref{eq:kinform} in the definition of kinetic solution, even though this presentation may look slightly redundant at first. Both of these equations will actually be needed in the proof of uniqueness.

\rmk{Before doing so, we proceed with a reminder of two technical definitions introduced in \cite{debus} and extending classical concepts from PDE literature: the definition of a Young measure and a kinetic function. Just as in \cite{debus, H16}, the consideration of such specific objects will be one of the keys toward wellposedness for the problem \eqref{eq:kinform}.}

In what follows, we denote by $\mathcal{P}_1(\R)$ the set of probability measures on $\R$.

\begin{definition}[Young measure]
Let $(X,\lambda)$ be a $\sigma$-finite measure space. A mapping $\nu:X\rightarrow\mathcal{P}_1(\R)$ is called a Young measure provided it is weakly measurable, that is, for all $\phi\in C_b(\R)$ the mapping $z\mapsto \nu_z(\phi)$ from $X$ to $\R$ is measurable.
A Young measure $\nu$ is said to vanish at infinity if
\begin{equation*}
\int_X\int_\R|\xi|\,\dif\nu_z(\xi)\,\dif\lambda(z)<\infty.
\end{equation*}
\end{definition}

\begin{definition}[Kinetic function]
Let $(X,\lambda)$ be a $\sigma$-finite measure space.
A measurable function $f:X\times\R\rightarrow[0,1]$ is called a kinetic function on $X$ if there exists a Young measure $\nu$ on $X$ that vanishes at infinity and such that for a.e. $z\in X$ and for all $\xi\in\R$
$$f(z,\xi)=\nu_z((\xi,\infty)).$$
\end{definition}

\rmk{We are now ready to introduce the notion of generalized kinetic solution to \eqref{eq:kinform}, as an intermediate step in the construction of full solutions.}

\begin{definition}[Generalized kinetic solution]\label{genkinsol}
Let $f_0:\R^{N+1}\rightarrow[0,1]$ be a kinetic function. A measurable function $f:[0,T]\times\R^{N+1}\to [0,1]$ is called a generalized kinetic solution to~\eqref{eq1} with initial datum $f_0$ provided
\begin{enumerate}
\item there exist $f^\pm$, such that $f^+_t=f^-_t=f_t$ for a.e. $t\in[0,T]$, $f^\pm_t$ are kinetic functions on $\R^N$ for all $t\in[0,T]$, and the associated Young measures $\nu^\pm$  satisfy
\begin{equation}\label{integrov}
\sup_{0\leq t\leq T}\int_{\R^N}\int_{\R}|\xi|\,\dif\nu^\pm_{t,x}(\xi)\,\dd x<\infty , 
\end{equation}
\item $f^+_0=f^-_0=f_0$,
\item there exists a finite Borel measure $m$ on $[0,T]\times\R^{N+1}$,
\item there exist $f^{\pm,\natural}\in V^{\frac{q}{3}}_{2,\text{loc}}([0,T];E_{-3})$
for some $q<3$,
\end{enumerate}
such that, recalling our definition \eqref{eq:def-A-star-conserv} of $A^{1,*}$ and $A^{2,*}$, we have that
\begin{eqnarray}
\delta f^+_{st}(\varphi)&=&f^+_s(A^{1,*}_{st} \varphi +A^{2,*}_{st} \varphi)-m(\1_{(s,t]}  \partial_\xi \varphi)+f^{+,\natural}_{st}(\varphi) ,\label{eq:delta-f-plus-weak}\\
\delta f^-_{st}(\varphi)&=&f^-_s(A^{1,*}_{st} \varphi +A^{2,*}_{st} \varphi)-m(\1_{[s,t)}  \partial_\xi \varphi)+f^{-,\natural}_{st}(\varphi) , \label{eq:delta-f-minus-weak}
\end{eqnarray}
holds true for all $s<t\in[0,T]$ and all $\varphi\in E_3$. 
\end{definition}

Finally we state the precise notion of solution we will consider for eq.~\eqref{eq:kinform}:

\begin{definition}[Kinetic solution]\label{def:kinsol}
Let $u_0\in L^1(\R^N).$
Then $u\in  L^\infty(0,T;L^1(\R^N))$
is called a kinetic solution to \eqref{eq1} with initial datum $u_0$ if the function $f_{t}(x,\xi) = \1_{u_{t}(x)>\xi}$ is a generalized kinetic solution according to Definition \ref{genkinsol} with initial condition $f_0(x,\xi) = \1_{u_0(x)>\xi}$. 
\end{definition}

\subsection{The main result}
\label{subsec:main}

Our well-posedness result for the conservation law \eqref{eq1} reads as follows.

\begin{theorem}\label{thm:main}
Let $u_0\in L^1(\R^N)\cap L^2(\R^N)$, and assume our Hypothesis \ref{hyp:V} and \ref{hyp:z} are satisfied. Then the following statements hold true:
\begin{enumerate}
\item[\emph{(i)}] There exists a unique kinetic solution to \eqref{eq1} and it belongs to $L^\infty(0,T;L^2(\R^N))$.
\item[\emph{(ii)}] Any generalized kinetic solution is actually a kinetic solution, that is, if $f$ is a generalized kinetic solution to equation \eqref{eq1} with initial datum $\ind_{u_0>\xi}$ then there exists a kinetic solution $u$ to \eqref{eq1} with initial datum $u_0$ such that $f=\ind_{u>\xi}$ for a.e. $(t,x,\xi)$.
\item[\emph{(iii)}] If $u_1,\,u_2$ are kinetic solutions to \eqref{eq1} with initial data $u_{1,0}$ and $u_{2,0}$, respectively, then for a.e. $t\in[0,T]$
\begin{equation*}
\|(u_1(t)-u_2(t))^+\|_{L^1}\leq \|(u_{1,0}-u_{2,0})^+\|_{L^1}.
\end{equation*}
\end{enumerate}
\end{theorem}

\begin{remark}\label{u+}
Note that in the definition of a kinetic solution, $u$ is a class of equivalence in the functional space $L^\infty(0,T;L^1(\R^N))$. Consequently, the $L^1$-contraction property holds true only for a.e. $t\in[0,T]$. However, it can be proved that in the class of equivalence $u$ there exists a representative $u^+$, defined through $\ind_{u^+(t,x)>\xi}=f^+_{t}(x,\xi)$, which has better continuity properties and in particular it is defined for every $t\in[0,T]$. If $u^+_1$ and $u^+_2$ are these representatives associated to $u_1$ and $u_2$ respectively, then
$$\|(u^+_1(t)-u^+_2(t))^+\|_{L^1}\leq \|(u_{1,0}-u_{2,0})^+\|_{L^1}$$
is satisfied for every $t\in[0,T]$.
\end{remark}

\subsection{Conservation laws with smooth drivers}

Let us show that in the case of a smooth driver $z$, our notion of solution coincides with the classical notion of kinetic solution. Recall that using the standard  theory for conservation laws, one obtains existence of a unique kinetic solution $u\in L^\infty(0,T;L^1(\R^N))\cap L^\infty(0,T;L^2(\R^N))$ to the  problem
\begin{equation}\label{eq:classic-conservation}
\partial_t u+\diver(A(x,u))\dot z=0,\qquad
u(0)=u_0.
\end{equation}
In other words, there exists a kinetic measure $m$ such that $f=\ind_{u>\xi}$ satisfies the corresponding kinetic formulation
\begin{align*}
\partial_t f &=\partial_\xi m +V\cdot\nabla_{\xi,x} f\, \dot{z}, \\
f(0)&=f_0=\ind_{u_0>\xi},
\end{align*}
in the sense of distributions over $[0,T)\times\R^{N+1}$, that is, for every $\varphi\in C^\infty_c([0,T)\times\R^{N+1})$ it holds true
\begin{align}\label{eq:weak}
\int_0^T f_t(\partial_t\varphi_t)\,\dd t+f_0(\varphi_0)=\int_0^T f_t(V\cdot\nabla\varphi_t)\,\dd z_t+m(\partial_\xi\varphi)
\end{align}
(recall that $\mathrm{div} V=0$).

\begin{lemma}\label{lemma:classical-to-rough} 
\rmk{Let $u$ be a classical kinetic solution of equation \eqref{eq:classic-conservation}. Then $u$ is also a rough kinetic solution in the sense of Definition~\ref{def:kinsol}, and the following relation holds true:
\begin{equation*}
f_{t}(x,\xi)=\nu_{t,x}((\xi,\infty))=\1_{u_{t}(x)>\xi},
\end{equation*}
where $\nu$ is the Young measure related to $f$.
}
\end{lemma}

In order to derive from this formulation the equations for time increments needed in Definition \ref{genkinsol}, let us first recall a classical compactness result for Young measures.

\begin{lemma}[Compactness of Young measures]\label{kinetcomp}
Let $(X,\lambda)$ be a $\sigma$-finite measure space such that $L^1(X)$ is separable. Let $(\nu^n)$ be a sequence of Young measures on $X$ such that for some $p\in[1,\infty)$
\begin{equation}\label{eq:estyoungm}
\sup_{n\in\N}\int_X\int_\R|\xi|^p\,\dif\nu^n_z(\xi)\,\dif \lambda(z)<\infty.
\end{equation}
Then there exists a Young measure $\nu$ on $X$ satisfying \eqref{eq:estyoungm} and a subsequence, still denoted by $(\nu^n)$, such that for all $h\in L^1(X)$ and all $\phi\in C_b(\R)$
$$\lim_{n\rightarrow \infty}\int_X h(z)\int_\R \phi(\xi)\,\dif\nu^n_z(\xi)\,\dif\lambda(z)=\int_X h(z)\int_\R \phi(\xi)\,\dif\nu_z(\xi)\,\dif\lambda(z)$$
Moreover, if $f_n,\,n\in\N,$ are the kinetic functions corresponding to $\nu^n,\,n\in\N,$ such that \eqref{eq:estyoungm} holds true, then there exists a kinetic function $f$ (which correponds to the Young measure $\nu$ whose existence was ensured by the first part of the statement) and a subsequence still denoted by $(f^n)$ such that
$$f_n\overset{w^*}{\longrightarrow} f\quad \text{ in }\quad L^\infty(X\times\R).$$
\end{lemma}

With this result in hand, we are able to obtain the representatives $f^+$ and $f^-$ of $f$.

\begin{lemma}\label{lemma:left-right}
\rmk{Let $f$ be a classical kinetic solution defined as in Lemma \ref{lemma:classical-to-rough}.}
For fixed $t\in (0,T)$ and $\ep>0$ set:
\begin{equation*}
{f}^{+,\varepsilon}_t:=\frac{1}{\varepsilon} \int_t^{t+\varepsilon}{f}_s\,\dd s,
\qquad 
{f}^{-,\varepsilon}_t:=\frac{1}{\varepsilon} \int^t_{t-\varepsilon}{f}_s\,\dd s .
\end{equation*}
Then there exist $f^+,\,f^-$, representatives of the class of equivalence $f$, such that, for every $t\in(0,T)$, $f^+_t,\, f^-_t$ are kinetic functions on $\R^N$ and, along subsequences,
$$
{f}^{+,\varepsilon}_t \overset{*}{\rightharpoonup}f^+_t,
\quad\text{and}\quad
{f}^{-,\varepsilon}_t \overset{*}{\rightharpoonup}f^-_t\qquad\text{in}\qquad L^\infty(\R^{N+1}).$$
Moreover, 
the corresponding Young measures $\nu^\pm_t$ satisfy
\begin{equation}\label{est:23e}
\sup_{t\in(0,T)}\int_{\R^N}\int_\R \big(|\xi|+|\xi|^2\big)\,\nu^\pm_{t,x}(\dd\xi)\leq \|u\|_{L^\infty_t L^1_{x}}+\|u\|^2_{L_t^\infty L^2_x}.
\end{equation}
\end{lemma}

\begin{proof}
Both ${f}^{+,\varepsilon}$ and ${f}^{-,\varepsilon}$ are kinetic functions on $\R^N$, with associated Young measures given by:
$${\nu}^{+,\varepsilon}_{t}= \frac{1}{\varepsilon} \int_t^{t+\varepsilon} {\nu}_{s} \,\dd s,\qquad {\nu}^{-,\varepsilon}_{t}= \frac{1}{\varepsilon} \int^t_{t-\varepsilon} {\nu}_{s} \,\dd s.
$$
\rmk{Furthermore, recall that $\nu_{t,x}(d\xi)=\delta_{u_{t}(x)}(d\xi)$.
Hence, due to the fact that $u$ sits in the space $L^\infty(0,T;L^1(\R^N))\cap L^\infty(0,T;L^2(\R^N))$, the following relation holds true:}
\begin{equation*}
\begin{split}
\int_{\R^N} \int_\R \big(|\xi| +|\xi|^2\big)\, {\nu}^{\pm,\varepsilon}_{t,x}(\dd \xi)\,\dd x &\leq \esssup_{t\in [0,T]} \, \int_{\R^N} \int_\R \big(|\xi| +|\xi|^2\big) \, {\nu}_{t,x}(\dd \xi)\,\dd x\leq  \|u\|_{L^\infty_t L^1_{x}}+\|u\|^2_{L_t^\infty L^2_x}.
\end{split}
\end{equation*}
Thus, we can apply Lemma \ref{kinetcomp} to deduce the existence of $f^+_t, f^-_t$, which are kinetic functions on $\R^N$ such that, along a subsequence that possibly depends on $t$,
\begin{equation}\label{eq:33}
f^{+,\varepsilon}_t\overset{*}{\rightharpoonup} f^+_t,\qquad f^{-,\varepsilon}_t\overset{*}{\rightharpoonup}f^-_t\qquad\text{in}\qquad L^\infty(\R^{N+1}).
\end{equation}
Moreover, the associated Young measures $\nu^\pm_t$ satisfy \eqref{est:23e}.

\rmk{It remains to show that they also fulfill $f^+_t=f^-_t=f_t$ for a.e. $t\in(0,T).$
According to the classical Lebesgue differentiation theorem, there exists a set of full measure $E_\psi\subset[0,T]$ possibly depending on $\psi$, such that
\begin{equation*}
\lim_{\varepsilon\rightarrow 0}\frac{1}{\varepsilon}\int_{t}^{t+\varepsilon} f_s(\psi)\,\dif s= f_{t}(\psi)\qquad\text{for all } t\in E_\psi,
\end{equation*}
for any $\psi \in L^1(\mr^{N+1})$. Therefore, in view of \eqref{eq:33} we deduce that for every $\psi\in L^1(\R^{N+1})$ it holds
$$
f^+_t(\psi)=f_t(\psi) \qquad\text{for all } t\in E_\psi.
$$
As the space $L^1(\mr^{N+1})$ is separable (more precisely it contains a countable set $\mathcal{D}$ that separates points of $L^\infty(\R^{N+1})$), we deduce that $f^+_t=f_t$ for all $t$ from  the set of full measure $\cap_{\psi\in\mathcal{D}} E_\psi$. The same argument applied to $f^-$ then completes the claim.}
\end{proof}

\begin{proof}[Proof of Lemma \ref{lemma:classical-to-rough}]
As a consequence of Lemma \ref{lemma:left-right}, for all $\varphi\in C^\infty_c(\R^{N+1})$, 
\begin{align*}
\delta f^{+}_{st}(\varphi)&=- \int_s^t f_r(V\cdot\nabla\varphi)\,\dd z_r-m(\ind_{(s,t]}\partial_\xi\varphi),\\
\delta f^{-}_{st}(\varphi)&=- \int_s^t f_r(V\cdot\nabla\varphi)\,\dd z_r-m(\ind_{[s,t)}\partial_\xi\varphi),
\end{align*}
hold true for every $s,t\in(0,T).$ This can be obtained by testing \eqref{eq:weak} by $\psi^{+,\varepsilon}\varphi$ and $\psi^{-,\varepsilon}\varphi$ where $\psi^{+,\varepsilon}$ and $ \psi^{-,\varepsilon}$ are a suitable approximations of $\ind_{[0,t]}$ and $\1_{[0,t)}$, respectively, such as
\begin{align}\label{eq:23a}
\psi^{+,\varepsilon}_r:=\begin{cases}
1, & \text{if}\  r\in [0,t],\\
1-\frac{r-t}{\varepsilon}, & \text{if}\ r\in [t,t+\varepsilon], \\
0, & \text{if} \ r\in [t+\varepsilon,T],
\end{cases}
\qquad
\psi^{-,\varepsilon}_r:=\begin{cases}
1, & \text{if}\  r\in [0,t-\varepsilon],\\
-\frac{s-t}{\varepsilon}, & \text{if}\ r\in [t-\varepsilon,t], \\
0, & \text{if} \ r\in [t,T],
\end{cases}
\end{align}
and passing to the limit in $\varepsilon$.
Therefore, we arrive at the equivalent formulation
\begin{align*}
\begin{aligned}
\delta f^{+}_{st}(\varphi)&=f^{+}_s(A^{1,*}_{st}\varphi)+f^{+}_s(A^{2,*}_{st}\varphi)+f^{+,\natural}_{st}(\varphi)-m(\ind_{(s,t]}\partial_\xi\varphi),\\
\delta f^{-}_{st}(\varphi)&=f^{-}_s(A^{1,*}_{st}\varphi)+f^{-}_s(A^{2,*}_{st}\varphi)+f^{-,\natural}_{st}(\varphi)-m(\ind_{[s,t)}\partial_\xi\varphi),
\end{aligned}
\end{align*}
which holds true in the scale $(E_n)$ with $E_n=W^{n,1}(\R^{N+1})\cap W^{n,\infty}(\R^{N+1})$ for remainders $f^{\pm,\natural}$ given by
$$
f^{\pm,\natural}_{st}(\varphi) = - f^{\pm}_s(A^{2,*}_{st}\varphi) + \int_s^t (f^{\pm}_r(V\cdot\nabla\varphi)-f^{\pm}_s(V\cdot\nabla\varphi))\dd z_{r}.
$$
Where we have replaced $f$ by $f^\pm$  in the above Riemann-Stieltjes  since $f^{+}_t=f^{-}_t=f_t$ for a.e. $t\in(0,T)$. Plugging into the integral the equation for $f^\pm$ we get
$$
f^{\pm,\natural}_{st}(\varphi) = - f^{\pm}_s(A^{2,*}_{st}\varphi) - \int_s^t \left[ \int_s^r f_w(V\cdot\nabla (V\cdot\nabla\varphi))\,\dd z_w+m(\ind_{(s,r]}\partial_\xi (V\cdot\nabla\varphi))\right]\dd z_{r}.
$$
Inspection of this expression shows that $f^{\pm,\natural} \in V_{2,\text{loc}}^{p}(E_{-2})$ for any $p\geq 1/2$.  
Moreover, it can be proved, cf. \cite[Remark 12]{DV} or \cite[Lemma 4.3]{H16}, that the kinetic measures $m$ do not have atoms at $t=0$ and consequently $f^{+}_0=f_0$.
\end{proof}

\section{Rough conservation laws II: Uniqueness and reduction}
\label{sec:uniq}

\subsection{General strategy}\label{sec:strategy-uniq}

\rmk{Before we turn to the details, let us briefly sketch out the main steps of our method toward uniqueness for the problem (\ref{eq1}), interpreted through the above kinetic formulation. The starting observation is, in fact, the basic identity
$$(u^1-u^2)^+=\int_\R\ind_{u^1>\xi}(1-\ind_{u^2>\xi})\,\dif \xi \quad , \quad u^1,u^2\in \R \ ,$$
which, applied to two kinetic functions ${u}^1,{u}^2$, immediately yields 
$$
\big\|\big(u^1_t -u^2_t\big)^+\big\|_{L^1_x}=\| f^{1}_t(1-{f}^{2}_t)\|_{L^1_{x,\xi}} \ ,
$$
where $f^1,f^2$ stands for the generalized kinetic functions associated with $u^1,u^2$.}

\rmk{We are thus interested in a estimate for the product $f^{1}(1-{f}^{2})$, and to this end, we will naturally try to understand the dynamics of this path. At this point, observe that owing to (\ref{eq:delta-f-plus-weak})-(\ref{eq:delta-f-minus-weak}), the two paths $f^1$ and $\bar{f}^2:=1-f^2$ (or rather their representatives $f^{i,\pm}$) are solutions of rough equations driven by the same driver $\mathbf{A}$ (we will carefully justify this assertion below).}

\rmk{This brings us back to the same setting of Section \ref{subsec:prelim}, and following the ideas therein described, we intend to display a {\it tensorization} procedure based on the consideration of the path
$$F := f^{1,+} \otimes \bar{f}^{2,+} \ .$$}

\rmk{The main steps of the analysis are those outlined in Section \ref{subsec:prelim}, namely:}

\smallskip

\noindent
\rmk{$(i)$ Derive the specific rough equation satisfied by $F$ (that is, the corresponding version of (\ref{eq:tensor})), with clear identification of a drift term $Q$ and a remainder $F^\natural$. This is the purpose of Section \ref{subsec:tens} below, and, as expected, it will involve the tensorized driver $\mathbf{\Gamma}$ derived from $\mathbf{A}$ along (\ref{driver:tens}).}

\smallskip

\noindent
\rmk{$(ii)$ Apply the blow-up transformation (\ref{blow-up-trans}) to the equation and, in order to use our a priori estimate on the remainder, try to find suitable bounds for the (transformed) drift term $Q^\ep:=T_\ep^\ast Q$, as well as for the supremum of $F^\ep:=T_\ep^\ast F$. These issues will be adressed in Sections \ref{subsec:f-ep} and \ref{subsec:q-ep}.}

\smallskip

\noindent
\rmk{$(iii)$ Combine Theorem \ref{theo:apriori} with the renormalizability property of the tensorized driver (as proved in Proposition \ref{prop:renorm}) in order to estimate the (transformed) remainder $F^{\natural,\ep}:=T_\ep^\ast F^{\natural}$. Then use this control to pass to the diagonal (that is to let $\ep$ tend to $0$) and, with the limit equation at hand, try to settle a rough Gronwall argument toward the desired estimate (in a way similar way to the example treated in Section \ref{subsec:first-application}). This will be the topic of Section \ref{subsec:diago}, and it will finally lead us to the expected uniqueness property.}

\

\rmk{The principles of this three-step procedure are thus quite general, and they could certainly be used as guidelines for other rough-PDE models.} 

\smallskip

\rmk{Nevertheless, when it turns to its rigorous implementation (at least in our case), the above scheme happens to be the source of (painful) technical difficulties related to the \enquote{localizability} of the test-functions, that is the control of their support. As we already evoked it in Remark~\ref{rk:local-x-plus}, these difficulties will force us to consider the sophisticated scale $\mathcal{E}_{R,n}$ introduced in (\ref{scale-e-r-n}), and thus to handle an additional parameter $R\geq 1$ throughout the procedure (on top of the blow-up parameter $\ep$). The dependence of the resulting controls with respect to $R$ will be removed afterwards, via a (rough) Gronwall-type argument.}

\rmk{Note finally that the construction of a smoothing (in the sense of Definition \ref{def:smoothing}) for the \enquote{localized} scale $\mathcal{E}_{R,n}$ is not an as easy task as in the situation treated in Section \ref{subsec:first-application}: we will go back to this problem in Section \ref{subsec:smoothing} and therein construct a suitable family of operators.}

\

\rmk{\it{From now on and for the rest of the section, let $f^1,f^2$ be two generalized kinetic solutions to \eqref{eq1}, and fix two associated measures $m^1,m^2$ (along Definition \ref{genkinsol}(iii)). Besides, we will use the following notation: $\bar f:=1-f$ as well as $\mathbf{x}:=(x,\xi)\in\R^{N+1}$, $\mathbf{y}:=(y,\zeta)\in\R^{N+1}$.}}

\subsection{Tensorization}
\label{subsec:tens}
\rmk{We here intend to implement Step $(i)$ of the above-described procedure, that is to derive the specific rough equation governing the path} $F = f^{1,+} \otimes \bar{f}^{2,+}$ defined on $[0,T] \times \R ^{N+1}\times\R ^{N+1}$ by
\begin{equation}\label{eq:tensor-F}
F_t(\mathbf{x},\mathbf{y}): = f^{1,+}_t(\mathbf{x})\bar f^{2,+}_t(\mathbf{y}) .
\end{equation}
\rmk{For the moment, let us consider the general} scale of spaces $ \mathcal{E}_n^\otimes = W^{n,1}(\R^{N+1} \times \R^{N+1}) \cap W^{n,\infty}(\R^{N+1} \times \R^{N+1})$ \rmk{($0\leq n\leq 3$)} with norms
$$
\|\Phi\|_{\mathcal{E}_n^\otimes} = \|\Phi\|_{W^{n,1}(\R^{N+1} \times \R^{N+1})} + \|\Phi\|_{W^{n,\infty}(\R^{N+1} \times \R^{N+1})} \ ,
$$
\rmk{and recall that the tensorized driver $\mathbf{\Gamma}=(\Gamma^1,\Gamma^2)$ is defined along the formulas
\begin{equation*}
\Gamma_{st}^1 := A^1_{st}  \otimes \mathbb{I}+\mathbb{I} \otimes A^1_{st} \quad , \quad
\Gamma_{st}^2 := A^2_{st} \otimes \mathbb{I}+\mathbb{I} \otimes A^2_{st} + A^1_{st} \otimes A^1_{st} \ .
\end{equation*}}

\begin{proposition}\label{prop:tens-eq-conserv}
\rmk{In the above setting, and} for all test functions $\Phi\in \mathcal{E}_{3}^\otimes$, the following relation is satisfied:
\begin{equation}\label{eq:dcp-delta-F-weak}
\delta F_{st}(\Phi) = \delta Q_{st}(\Phi) + F_s((\Gamma_{st}^{1,*}  + \Gamma_{st}^{2,*})\Phi) +  F^{ \natural}_{st}(\Phi),
\end{equation}
where $F^{ \natural} \in V_2^{q/3}(\mathcal{E}_{-3}^\otimes)$ and where $Q$ is the path defined (in the distributional sense) as:
\begin{equation}\label{eq:def-Q-distrib}
Q_{t}: = Q^{1}_{t} - Q^{2}_{t} 
=\int_{[0,t]} \partial_\xi m^1_{dr} \otimes \bar f^{2,-}_r -  \int_{[0,t]} f^{1,+}_r \otimes \partial_\xi m^2_{dr}.
\end{equation}
\end{proposition}

\begin{proof}
Let us first work out the algebraic form of the equation governing $F$ in a formal way. Namely, according to relations \eqref{eq:delta-f-plus-weak}, the equations describing the dynamics of $f^{1,+}$ and $\bar{f}^{2,+}$  in the distributional sense are given by:
\begin{eqnarray}
\delta f_{st}^{1,+} &=&
A_{st}^{1} f_{s}^{1,+} + A_{st}^{2} f_{s}^{1,+} + \partial_{\xi}m^{1}\lp \1_{(s,t]} \rp 
+ f_{st}^{1,+,\natural} \label{eq:fst-1}\\
\delta \bar{f}_{st}^{2,+} &=&
A_{st}^{1} \bar{f}_{s}^{2,+} + A_{st}^{2} \bar{f}_{s}^{2,+} - \partial_{\xi}m^{2}\lp \1_{(s,t]} \rp 
- f_{st}^{2,+,\natural}. \label{eq:bar-fst-2}
\end{eqnarray}
In order to derive the equation for $F$, we tensorize the equation for $f^{1,+}$ with the equation for $\bar f^{2,+}$.  Similarly to \eqref{eq:tensor-heat-1} we obtain the following relation, understood in the sense of distributions over $\R^{N+1}\times\R^{N+1}$:
\begin{equation*}
\delta F_{st} = \delta f_{st}^{1,+} \otimes \bar{f}_{s}^{2,+} +  f_{s}^{1,+} \otimes \delta \bar{f}_{st}^{2,+} 
+ \delta f_{st}^{1,+} \otimes \delta \bar{f}_{st}^{2,+}.
\end{equation*}
Expanding $\delta f_{st}^{1,+}$ and $\delta \bar{f}_{st}^{2,+}$ above according to \eqref{eq:fst-1} and \eqref{eq:bar-fst-2}, we end up with:
\begin{multline}\label{eq:dcp-delta-F}
\delta F_{s t} 
= \Gamma^1_{s t} F_s + \Gamma^2_{s t} F_s \\
- f^{1, +}_s \otimes \partial_\xi m^2 (\1_{(s, t]}) - \partial_\xi m^1 (1_{(s, t]}) \otimes \partial_\xi m^2 (\1_{(s, t]}) +\partial_\xi m^1
   (\1_{(s, t]})  \otimes \bar{f}_s^{2, +} + R^1_{s t} ,
\end{multline}
where all the other terms have been included in the
remainder $R^1_{s t}$. More explicitly 
\[ R^1_{s t} 
=  A^2_{s t} f^{1, +}_s \otimes A^1_{s t}  \bar{f}^{2,
   +}_s 
   + A^1_{s t}  f^{1, +}_s \otimes  A^2_{s t} \bar{f}^{2,
   +}_s  +  A^2_{s t} f^{1, +}_s \otimes  A^2_{s t} \bar{f}^{2,
   +}_s \]
   \[ + f^{1, +, \natural}_{s t} \otimes \overline{f_{}}_s^{2, +}  - f^{1, +}_s \otimes f^{2, +, \natural}_{s t}  \]
\[     - (A^1_{s t} + A^2_{s t}) f^{1, +}_s \otimes \partial_\xi m^2 (\1_{(s, t]}) +\partial_\xi m^1 (\1_{(s, t]}) \otimes (A^1_{s t} + A^2_{s t}) \bar{f}^{2, +}_s
    -\partial_\xi  m^1 (\1_{(s,
   t]}) \otimes \partial_\xi m^2 (\1_{(s, t]})  
    \]
\[     - (A^1_{s t} + A^2_{s t}) f^{1, +}_s \otimes f^{2, +, \natural}_{s t} + f^{1, +, \natural}_{s t} \otimes (A^1_{s t} + A^2_{s t}) \bar{f}^{2, +}_s  \]
\[  -\partial_\xi m^1 (\1_{(s, t]}) \otimes f^{2, +, \natural}_{s t}
- f^{1,   +, \natural}_{s t} \otimes \partial_\xi m^2 (\1_{(s, t]}) 
- f^{1, +, \natural}_{s t} \otimes f^{2, +,   \natural}_{s t} 
 .\]
Let us further decompose the term $I:=\partial_\xi m^1 (\1_{(s, t]}) \otimes \partial_\xi m^2(\1_{(s, t]})$ in \eqref{eq:dcp-delta-F}.
The integration by parts formula for two general BV functions $A$ and $B$ reads as
\[ A_t B_t = A_s B_s + \int_{(s, t]} A_r \mathd B_r + \int_{(s, t]} B_{r -}
   \mathd A_r. 
   \]
Applying this identity to $I$, we obtain $I=I_{1}+I_{2}$ with
\begin{equation*}
I_{1}=- \int_{(s, t]} \partial_\xi m^1 (\1_{(s, r]})\otimes \mathd \partial_\xi m_r^2,
\quad\text{and}\quad
I_{2}= - \int_{(s, t]} \mathd \partial_\xi m^1_r \otimes  \partial_\xi m^2 (\1_{(s, r)}).
\end{equation*}
We now handle $I_{1}$ and $I_{2}$ separately. 
For the term $I_{1}$ we invoke again equation \eqref{eq:fst-1} describing the dynamics of $f^{1, +}$, which yields
\begin{eqnarray*}
I_1 &=& -\int_{(s, t]} (\delta f^{1, +}_{s r} - (A^1_{s r} + A^2_{s r}) f^{1,
   +}_s - f^{1, +, \natural}_{s r}) \otimes \mathd \partial_\xi m^2_r = -\int_{(s, t]} \delta f^{1,
   +}_{s r} \otimes \mathd\partial_\xi  m^2 + R^2_{s t} \\
&=&
-\int_{(s, t]} f^{1,+}_{r} \otimes \mathd\partial_\xi  m^2 
+ f^{1,+}_s  \otimes\partial_\xi m^2(\1_{(s,t]}) + R^2_{s t}.
\end{eqnarray*}
Similarly, we let the patient reader check that the equivalent of
relation \eqref{eq:bar-fst-2} for $\delta \bar{f}_{st}^{-}$, derived from~\eqref{eq:delta-f-minus-weak}, leads to
\begin{equation*}
I_{2}=
\int_{(s, t]} \bar{f}^{2, -}_r \otimes \mathd \partial_\xi m^1_r -
   \overline{f^{}}^{2, -}_s \otimes \partial_\xi m^1 (\1_{(s, t]}) +\partial_\xi m^2 (\1_{\{ s \}}) \otimes \partial_\xi m^1 (\1_{(s, t]})
   + R^3_{s t}.
\end{equation*}
In addition, observe that $f^{2, +}_s - f^{2, -}_s = \partial_\xi m^2 (\1_{\{ s \}})$. Hence $\bar{f}^{2,
-}_s- \partial_\xi m (\1_{\{ s \}}) = \bar{f}^{2, +}_s$ and we obtain
\[ 
I_2 = \int_{(s, t]} \bar{f}^{2, -}_r \otimes \mathd \partial_\xi m^1_r -\overline{f^{}}^{2,
   +}_s  \otimes \partial_\xi m^1 (\1_{(s, t]}) + R^3_{s t}. 
   \]
Plugging the  relations we have obtained for $I=I_{1}+I_{2}$  into \eqref{eq:dcp-delta-F} and looking for cancellations, we end up with the following expression for $\delta F$:
\begin{equation*}
\delta F_{s t} 
= \Gamma^1_{s t} F_s + \Gamma^2_{s t} F_s
+\int_{(s, t]} \mathd \partial_\xi m^1_r  \otimes \bar{f}^{2, -}_r
-\int_{(s, t]} f^{1,+}_{r} \otimes \mathd\partial_\xi  m^2 
+F^\natural_{s t}.
\end{equation*}
with $F^\natural_{s t} = R^1_{s t}+R^2_{s t}+R^3_{s t}$.
Having the definition \eqref{eq:def-Q-distrib} of $Q$ in mind, this proves equation~\eqref{eq:dcp-delta-F-weak} in the distributional sense, for test functions $\Phi \in C^\infty_c(\R^{N+1}\times\R^{N+1})$ since distributions can act in each set of variables separately. We \rmk{now establish the claimed regularity for $F^\natural$ through an interpolation argument. To this end, consider the smoothing $(J^\eta)_{\eta\in (0,1)}$ (with respect to $(\mathcal{E}_n^\otimes)_{0\leq n\leq 3}$) derived from the same basic convolution procedure as in (\ref{smoothing-convolution}), and for $\Phi \in C^\infty_c(\R^{N+1}\times\R^{N+1})$, write}
\[
F^\natural_{s t}(\Phi) = F^\natural_{s t}(J^\eta \Phi) + F^\natural_{s t}((\id -J^\eta) \Phi) \ .
\]
The first term will be estimated with the decomposition into the various remainder terms
$
F^\natural_{s t}(J^\eta \Phi)  = R^1_{s t}(J^\eta \Phi) +  R^2_{s t}(J^\eta \Phi) + R^3_{s t}(J^\eta \Phi).
$
Close inspection of the precise form of $R^i$ for $i=1,2,3$ shows that the terms which require more than three derivatives from $J^\eta \Phi$ (resulting in negative powers of $\eta$) are also more regular in time. On the other hand, $ F^\natural_{s t}((\id -J^\eta) \Phi)$ can be estimated directly from the equation~\eqref{eq:dcp-delta-F-weak} and while the various terms show less time regularity they also require less than three derivatives from $(\id -J^\eta) \Phi$ which in turn become positive powers of $\eta$. Reasoning as in the proof of \rmk{Theorem \ref{theo:apriori}} we can obtain a suitable choice for $\eta$ which shows that there exists a \rmk{regular} control $\omega_{\natural\natural}$, depending on the controls for $f^i, m^i, f^{i, \natural}$, $i=1,2$, such that
\begin{equation}
\label{eq:fnatural-est}
F^\natural_{s t}(\Phi) \leq \omega_{\natural \natural}(s,t)^{3/q} \|\Phi\|_{\mathcal{E}_3^\otimes}.
\end{equation}
   
To complete the argument we need to go from the distributional to the variational form of the dynamics of $F$. That is, we need to establish eq.~\eqref{eq:dcp-delta-F-weak} for all $\Phi \in \mathcal{E}^\otimes_3$ and not only for $\Phi \in C^\infty_c(\R^{N+1}\times\R^{N+1})$. In order to do so we observe that $ C^\infty_c(\R^{N+1}\times\R^{N+1})$ is weakly-$\star$ dense in $\mathcal{E}^\otimes_3$. Choosing a sequence $(\Phi_n)_n \subseteq  C^\infty_c(\R^{N+1}\times\R^{N+1})$ weakly-$\star$ converging to $\Phi \in \mathcal{E}^\otimes_3$ we see that all the terms in eq.~\eqref{eq:dcp-delta-F-weak} apart from the remainder $F^\natural$ converge. Consequently also the remainder converges and it satisfies the required estimates by~\eqref{eq:fnatural-est}.
\end{proof}

\rmk{Let us now turn to the implementation of Step $(ii)$ of the procedure described in Section~\ref{sec:strategy-uniq}. We recall that the blow-up transformation $(T_\ep)_{\ep\in (0,1)}$ has been introduced in Section \ref{subsec:prelim}, together with the explicit description of the related transforms $T_{\ep}^{*},T_{\ep}^{-1}$ (see \eqref{eq:def-T-eps-star} and \eqref{eq:def-T-eps-inverse}). Setting
\begin{equation}\label{eq:def-F-Q-eps}
\mathbf{\Gamma}^{\ast}_{\varepsilon}  := T_\varepsilon^{-1}\mathbf{\Gamma}^{\ast}T_\varepsilon \quad ,
\quad  
Q^\varepsilon := T^{\ast}_\varepsilon Q \quad \text{and}\quad
F^{\natural,\varepsilon}:=T_\varepsilon^* F^\natural \ ,
\end{equation}
it is easy to check that, for each fixed $\ep\in (0,1)$, the transformed path $F^\varepsilon := T^{\ast}_\varepsilon F$ satisfies the rough equation
\begin{equation}\label{eq:eps1}
\delta F^\varepsilon_{st}(\Phi)= \delta Q^\varepsilon_{st}(\Phi)+ F_s^\varepsilon((\Gamma^{1,*}_{\varepsilon,st}+\Gamma^{2,*}_{\varepsilon,st})\Phi)+F_{st}^{\natural,\varepsilon} (\Phi) \ ,
\end{equation}
in the same scale $(\mathcal E_{n}^\otimes)_{0\leq n\leq 3}$ as the original equation (\ref{eq:dcp-delta-F-weak}).}

\rmk{As a preliminary step toward an efficient application of Theorem \ref{theo:apriori} to equation (\ref{eq:eps1}), we need to find suitable bounds for $Q^\ep$ (keeping condition (\ref{cond-mu}) in mind) and for the supremum of $F^\ep$ (which, in view of (\ref{eq:def-omega-I}), will be involved in the resulting estimate). As we mentionned it earlier, the above scale $(\mathcal E_{n}^\otimes)$ turns out to be too general for the derivation of such bounds, and we must restrict our attention to the more specific (set of) localized scale(s) $(\mathcal E_{R,n})_{0\leq n\leq 3}$ ($R\geq 1$) defined in (\ref{scale-e-r-n}), that is
\begin{equation}\label{recall-e-r-n}
\mathcal E_{R,n}:=\left\{ \Phi\in W^{n,\infty}(\R^{N+1}\times\R^{N+1});\,\Phi(\mathbf{x},\mathbf{y})=0\;\text{if}\;\rho_R(\mathbf{x},\mathbf{y})\ge 1 \right\}\ ,
\end{equation}
with $\rho_R( \mathbf x,\mathbf y)^2 = |\mathbf{x}_+|^2/R^2 +  |\mathbf{x}_-|^2$.}

\subsection{\rmk{Construction of a smoothing}}\label{subsec:smoothing}

\rmk{The first condition involved in Theorem \ref{theo:apriori} is the existence of a suitable smoothing (in the sense of Definition \ref{def:smoothing}), and we thus need to exhibit such an object for the above scale $(\mathcal E_{R,n})_{0\leq n\leq 3}$ (for fixed $R\geq 1$).}

\rmk{Just as in Section \ref{subsec:first-application}, a first natural idea here is to turn to a convolution procedure: namely, we introduce a smooth rotation-invariant function $\jmath$ on $\R^{N+1}\times \R^{N+1}$ with support in the ball of radius $\frac12$ and such that $\int_{\R^{N+1}\times \R^{N+1}} \jmath(\mathbf{x},\mathbf{y}) \dd \mathbf{x}\dd \mathbf{y}=1$, and consider $J^\eta$ defined as  
$$
J^\eta \vp (\mathbf{x},\mathbf{y}):=\int_{\R^{N+1}\times \R^{N+1}} \jmath_\eta(\mathbf{x}-\tilde{\mathbf{x}},\mathbf{y}-\tilde{\mathbf{y}}) \vp(\tilde{\mathbf{x}},\tilde{\mathbf{y}})\dd \tilde{\mathbf{x}}\dd \tilde{\mathbf{y}}\ , \ \text{with} \ \  \jmath_\eta(\mathbf{x},\mathbf{y}):=\eta^{-2N-2} \jmath(\eta^{-1}(\mathbf{x},\mathbf{y}))\ .
$$}

\rmk{Unfortunately, the sole consideration of the so-defined family $(J^\eta)_{\eta\in (0,1)}$ is no longer sufficient in this \enquote{localized} setting, since convolution may of course increase the support of test-functions, leading to stability issues. Accordingly, an additional localization procedure must come into the picture.}

\rmk{To this end, let us first introduce a suitable cut-off function:}
\begin{notation}\label{not:theta}
Let $\eta\in (0,\tfrac{1}{3})$ and let $\theta_{\eta}\in C_c^\infty(\R)$ be such that
$$0\leq \theta_{\eta}\leq 1,\qquad\supp\, \theta_{\eta}\subset B_{1-2\eta}\subset \R,\qquad \theta_{\eta}\equiv 1\text{ on }B_{1-3\eta}\subset \R,$$
where for $\al>0$ we set $B_{\al}:=[-\al,\al]$. We also require the following condition on $\theta_{\eta}$:
$$|\nabla^k \theta_{\eta}|\lesssim \eta^{-k},\qquad\text{for}\;k=1,2.$$
Finally, \rmk{for all $R\geq 1$ and $\mathbf{x},\mathbf{y}\in \R^{N+1}$,} we define
\[ 
\Theta_\eta ( \mathbf x,\mathbf y):=\Theta_{R,\eta} ( \mathbf x,\mathbf y) =\theta_{\eta}(\rho_R(\mathbf x,\mathbf y)). 
\]
\end{notation}

\rmk{With these objects in hand}, we have the following technical result.

\begin{proposition}\label{prop:1}
\rmk{Let $\Theta_{\eta}$ be the function introduced in Notation \ref{not:theta}. Then it holds that}
\begin{align}
 \| \Theta_{\eta} \Phi \|_{\mathcal{E}_{R,k}} &\lesssim \| \Phi \|_{\mathcal{E}_{R,k}} \quad \text{for} \ \ k=0,1,2\ ,\label{1}\\
 \| ( 1- \Theta_{\eta} ) \Phi \|_{\mathcal{E}_{R,0}} &\lesssim \eta^k \| \Phi \|_{\mathcal{E}_{R,k}}\quad \text{for} \ \ k=1,2\ ,\label{2}\\
\| ( 1- \Theta_{\eta} ) \Phi \|_{\mathcal{E}_{R,1}} &\lesssim \eta \| \Phi \|_{\mathcal{E}_{R,2}} \ ,\label{4.4}
\end{align}
\rmk{with proportional constants independent of both $\eta$ and $R$. Besides, 
$$\supp(J^\eta\Theta_\eta\Phi)\subset\left\{(\mathbf{x},\mathbf{y})\in\R^{N+1}\times\R^{N+1};\, \rho_R(\mathbf{x},\mathbf{y}) \leq 1 \right\}$$
and} there exists $\Psi_R\in\mathcal{E}_{R,3}$ with  $\sup_{R\geq 1}\|\Psi_R\|_{\mathcal{E}_{R,3}}<\infty$ such that for all $\mathbf{x},\mathbf{y}\in\R^{N+1}$
\begin{align}
\eta^{3-k}|J^\eta\Theta_\eta\Phi(\mathbf{x},\mathbf{y})|&\lesssim \Psi_R(\mathbf{x},\mathbf{y})\|\Phi\|_{\mathcal{E}_{R,k}},\label{3}
\end{align}
where the proportional constant is again independent of both $\eta$ and $R$.
\end{proposition}

\rmk{Before we turn to the proof of these properties, let us observe that they immediately give rise to the expected smoothing:}
\rmk{\begin{corollary}\label{cor:hat-j}
For all $\eta\in (0,1)$ and $R\geq 1$, consider the operator $\hat{J}^\eta$ on $\mathcal{E}_{R,k}$ ($k=0,1,2$) defined as $\hat{J}^\eta(\Phi):=J^\eta(\Theta_{\eta} \Phi)$. Then $\hat{J}^\eta$ defines a smoothing on the scale $(\mathcal{E}_{R,k})$, and conditions \eqref{eq:reg-J-eta-1}-\eqref{eq:reg-J-eta-2} are both satisfied with proportional constants independent of $R$.
\end{corollary}}

\rmk{\begin{remark}\label{rk:smoothing-bg}
The existence of a smoothing for the localized scale $(\mathcal{E}_{R,n})$ is also an (unproven) assumption in the analysis carried out in \cite{BG} for a rough transport equation (see in particular Section 5.2 in the latter reference). The  statement above thus offers a way to complete these results.
\end{remark}}

\begin{proof}[Proof of Proposition \ref{prop:1}]
In order to prove \eqref{1}, we write
$$\nabla(\Theta_\eta\Phi)=(\nabla \Theta_\eta)\Phi+\Theta_\eta(\nabla\Phi),$$
where the second term does not pose any problem. On the other hand, the term $\nabla\Theta_\eta$ diverges as $\eta^{-1}$ due to the assumptions on $\theta_{\eta}$, namely, it holds
\begin{align*}
\nabla\Theta_\eta&=(\nabla \theta_{\eta})(\rho_R( \mathbf{x}, \mathbf{y}))\nabla \rho_R( \mathbf{x}, \mathbf{y}).
\end{align*}
But due to the support of $\Phi$ \rmk{and the fact that $\theta_{\eta}\equiv 1$ on $B_{1-3\eta}$,} we have that for every $({\mathbf{x}},{\mathbf{y}})$ in the region where $\nabla\Theta_{\eta}\neq0$ there exists $(\tilde{{\mathbf{x}}},\tilde{{\mathbf{y}}})$ outside of support of $\Phi$ such that $|( \mathbf{x}, \mathbf{y})-(\tilde{ \mathbf{x}},\tilde{ \mathbf{y}})|\lesssim \eta$ hence
$$|\Phi({\mathbf{x}},\mathbf{y})|=|\Phi({\mathbf{x}},\mathbf{y})-\Phi(\tilde{\mathbf{x}},\tilde{ \mathbf{y}})|\lesssim \eta\|\Phi\|_{\mathcal{E}_{R,1}}$$
and consequently \eqref{1} follows for $k=1$. If $k=2$, we have
\[ \nabla^{2} ( \Theta_{\eta} \Phi ) = ( \nabla^{2} \Theta_{\eta} ) \Phi +2
   \nabla \Theta_{\eta} \cdummy \nabla \Phi + \Theta_{\eta} \nabla^{2} \Phi,
\]
where the third term does not pose any problem and the second one can be estimated using the reasoning above.
For the first one, we observe that $\nabla_\mathbf{x}^2\Theta_\eta$
diverges like $\eta^{-2}$ but for every $(\mathbf{x}, \mathbf{y})$ such that $\nabla^2\Theta\neq 0$ there exists $(\tilde{{\mathbf{x}}},\tilde{{\mathbf{y}}})$ that lies outside of support of $\Phi$, satisfies $|(\mathbf{x},\mathbf{y})-(\tilde{\mathbf{x}},\tilde{\mathbf{y}})|\lesssim \eta$. Resorting to a second order Taylor expansion and invoking the fact that both $\Phi( \tilde{\mathbf{x}},\tilde{\mathbf{y}} )$ and $\nabla\Phi( \tilde{\mathbf{x}},\tilde{\mathbf{y}} )$ are vanishing we get:
\begin{align}\label{4.4b}
\begin{aligned}
 |\Phi (\mathbf{x},\mathbf{y}) | 
    &=  | \Phi (\mathbf{x},\mathbf{y}) - \Phi ( \tilde{\mathbf{x}},\tilde{\mathbf{y}} ) - \mathrm{D}\Phi ( \tilde{\mathbf{x}},\tilde{\mathbf{y}}) ( (\mathbf{x}, \mathbf{y})-(\tilde{\mathbf{x}},\tilde{ \mathbf{y}})) | \\
    &\lesssim 
   | \mathrm{D}^2\Phi ( z_\mathbf{x},{z}_\mathbf{y} )| |(\mathbf{x},\mathbf{y})-(\tilde{\mathbf{x}},\tilde{\mathbf{y}}) |^{2}  \lesssim\eta^2 \| \Phi \|_{\mathcal{E}_{R,2}} 
   \end{aligned}
\end{align}
and relation \eqref{1} follows.

The same approach leads to \eqref{2}. To be more precise, for $(\mathbf{x},\mathbf{y})$ from the support of $(1-\Theta_\eta)\Phi$ we have using the first and second order Taylor expansion, respectively,
\begin{align}\label{4.4a}
\begin{aligned}
 | ( 1- \Theta_{\eta} ) \Phi (\mathbf{x}, \mathbf{y} ) | 
\lesssim \eta^{k} \| ( 1- \Theta_{\eta} ) \Phi \|_{\mathcal{E}_{R,k}} \lesssim
   \eta^{k} \| \Phi \|_{\mathcal{E}_{R,k}} ,
   \end{aligned}
\end{align}
where we used \eqref{1} for the second inequality.

\rmk{To show  \eqref{4.4}, we write
$$\nabla[(1-\Theta_\eta)\Phi]=-(\nabla\Theta_\eta)\Phi+(1-\Theta_\eta)(\nabla\Phi).$$
The second term can be estimated due to \eqref{4.4a} as follows
$$|(1-\Theta_\eta)(\nabla\Phi)|\lesssim\eta\|\Phi\|_{\mathcal{E}_{R,2}}.$$
For the first term, we recall that even though $\nabla\Theta_\eta$ is of order $\eta^{-1}$, $\Phi$ can be estimated on the support of $\nabla\Theta_\eta$ by $\eta^2$ due to \eqref{4.4b}. This yields
$$|(\nabla\Theta_\eta)\Phi|\lesssim\eta\|\Phi\|_{\mathcal{E}_{R,2}}$$
and completes the proof.}

Let us now prove \eqref{3}. First of all, we observe the trivial estimate, for $k=1,2,$ and $(\mathbf{x}, \mathbf{y})\in \supp(J^\eta\Theta_\eta\Phi)$, 
\begin{equation}\label{4}
\eta^{3-k}|J^\eta\Theta_\eta\Phi(\mathbf{x}, \mathbf{y})|\leq \|J^\eta\Theta_\eta\Phi\|_{L^\infty}\lesssim \|\Phi\|_{\mathcal{E}_{0}}\lesssim \|\Phi\|_{\mathcal{E}_{k}}.
\end{equation}
Next, we note that 
\begin{equation}\label{4-1}
\supp(J^\eta\Theta_\eta\Phi)\subset\left\{(\mathbf{x},\mathbf{y})\in\R^{N+1}\times\R^{N+1};\, \rho_R(\mathbf{x},\mathbf{y}) \leq 1-\eta \right\}
\end{equation}
since $R\ge 1$, and denote
\begin{equation}\label{eq:def-D-R}
D_R:=\left\{(\mathbf{x}, \mathbf{y})\in\R^{N+1}\times\R^{N+1};\, \rho_R(\mathbf{x}, \mathbf{y}) \leq 1 \right\}.
\end{equation}
Let $d(\cdot,\partial D_R)$ denote the distance to its boundary $\partial D_R$. Owing to \eqref{4-1}, it satisfies
$$d((\mathbf{x}, \mathbf{y}),\partial D_R)\geq \eta\qquad\text{for all}\quad ( \mathbf{x}, \mathbf{y})\in \supp(J^\eta\Theta_\eta\Phi).$$
Therefore, performing a Taylor expansion we obtain for $k=1,2,$ and $(\mathbf{x}, \mathbf{y})\in \supp(J^\eta\Theta_\eta\Phi)$, 
\begin{align}\label{5}
\begin{aligned}
\eta^{3-k}|J^\eta\Theta_\eta\Phi( \mathbf{x},\mathbf{y})|&\lesssim \eta^{3-k} |d((\mathbf{x}, \mathbf{y}),\partial D_R)|^k\|J^\eta\Theta_\eta\Phi\|_{W^{k,\infty}}\\
&\lesssim|d((\mathbf{x},\mathbf{y}),\partial D_R)|^3\|\Phi\|_{\mathcal{E}_{R,k}}
\end{aligned}
\end{align}
where we also used \eqref{1} and the fact that $\mathcal{E}_{R,k}$ is embedded in $W^{k,\infty}$. Besides
we may put \eqref{4} and \eqref{5} together to conclude that there exists $\Psi_R\in\mathcal{E}_{R,3}$ satisfying the conditions stated in this Lemma and, in addition,
$$\min\left\{1,|d((\mathbf{x},\mathbf{y}),\partial D_R)|^3\right\}\lesssim \Psi_R(\mathbf{x},\mathbf{y})$$
which completes the proof. For example we can take 
$$
\Psi_R(\mathbf{x}, \mathbf{y}) = \begin{cases}
d(( \mathbf{x}, \mathbf{y}),\partial D_R)^3, &\ \text{ if  $(\mathbf{x},\mathbf{y})\in D_R$ and $d((\mathbf{x},\mathbf{y}),\partial D_R) \le 1/2$} ,\\
						 1,&\ 
\text{ if $(\mathbf{x},\mathbf{y})\in D_R$ and  $d((\mathbf{x},\mathbf{y}),\partial D_R) \ge 3/4$},
\end{cases}
$$
and complete it with a smooth interpolation in between. 
\end{proof}

\subsection{\rmk{Preliminary estimate for the supremum of $F^{\ep}$}}\label{subsec:f-ep}

\rmk{We can now go ahead with our strategy and} state our upper bound for $F^{\ep}$ in \eqref{eq:eps1}.

\begin{proposition}\label{prop:2}
Let $F^{\ep}$ be the increment defined by \eqref{eq:def-F-Q-eps}. Then for all $0\leq s\leq t\leq T$ it holds that
\rmk{\begin{equation}\label{bou-m-ep}
\mathcal{N}[F^\varepsilon;L^\infty(s,t;\mathcal{E}_{R,0}^*)]\lesssim  \mathcal{M}(s,t,\ep,R)\ ,
\end{equation}
where 
\begin{equation}\label{defi-mathcal-m}
\mathcal{M}(s,t,\ep,R):=\ep R^N+\sup_{s\leq r\leq t}\int_{\R^N} \int_\R | \xi |\, \nu^{1,+}_{r,x} (\dd \xi)\dd x+\sup_{s\leq r\leq t} \int_{\R^N}\int_\R | \zeta |\, \nu^{2,+}_{r,y} (\dd \zeta)\dd y
\end{equation}
and the proportional in \eqref{bou-m-ep} constant does not depend on $\ep$ and $R$. We recall that, following Definition \ref{genkinsol}(i), $\nu^{i,\pm}_{t,.}$ stands for the Young measure associated with the kinetic function $f^{i,\pm}_t$.}
\end{proposition}

\rmk{\begin{remark}
Observe that our localization procedure becomes apparent here for the first time. Indeed, the bound (\ref{bou-m-ep}) still depends on the localization parameter $R$. This lack of uniformity do not poses  any problem since our procedure will later consist in sending $\varepsilon\to0$ first and then $R\to\infty$. 
\end{remark}}

\begin{proof}
Consider the function $\Upsilon^\ep:\R^2 \to [0,\infty[$ defined as
\[ \Upsilon^\ep (\xi, \zeta): =  \int_{\zeta}^{\infty}  \int_{-
   \infty}^{\xi} \ep^{-1} \mathbf{1}_{|\xi'-\zeta'|\leq 2\ep}\,\dd\xi'\dd\zeta' , 
   \]
whose main interest lies in the relation $(\partial_{\zeta} \partial_{\xi} \Upsilon^\ep)(\xi, \zeta)= \ep^{-1}\1_{|\xi-\zeta|\leq 2\ep}$. Let us derive some elementary properties of $\Upsilon^\ep$. First, we obviously have:
\begin{equation}\label{eq:deriv-ups-eps}
\partial_{\xi} \Upsilon^\ep (\xi, \zeta) =  \int_{\zeta}^{\infty} \varepsilon^{- 1} \mathbf{1}_{|\xi-\zeta'|\leq 2\ep} \,\dd\zeta',
\end{equation}
and in particular $\partial_{\xi} \Upsilon^\ep (\xi, +\infty) = 0 $. A simple change of variables also yields:
\begin{eqnarray*}
 \Upsilon^\ep (\xi, \zeta)
	&=&  \int_{{\zeta - \xi}}^{\infty}  \int_{-
   \infty}^0  \varepsilon^{- 1} \mathbf{1}_{|\xi'-\zeta'|\leq 2\ep}\,\dd\xi'\dd\zeta' = \ \Upsilon^\ep ({0,\zeta - \xi})  .
\end{eqnarray*}
Moreover, writing
 $$
 \Upsilon^\ep ( 0,\zeta) =  \varepsilon \int_{-\infty}^{{\zeta/\varepsilon}} \int_{-\infty}^{\zeta'} \mathbf{1}_{|\xi'|\leq 2}\,\dd\xi'\dd\zeta'  $$
it follows that
$$|\Upsilon^\ep ( 0,0)|=\varepsilon\int_{-\infty}^0 \int_{-\infty}^{\zeta'}\mathbf{1}_{|\xi'|\leq 2}\,\dd\xi'\dd\zeta'  \lesssim \varepsilon. $$
Finally, using the elementary bound
$$| \partial_\zeta \Upsilon^\ep(\xi,\zeta)| \leq 2,$$
which stems from \eqref{eq:deriv-ups-eps}, we obtain that
\begin{align*}
|\Upsilon^\ep (\xi, \zeta)|&= | \Upsilon^\ep ( 0, \zeta-\xi) | \lesssim \varepsilon + |\zeta-\xi|\lesssim \varepsilon + | \xi |+|\zeta|.
   \end{align*}
   
Recall that, since both $f^{1}$ and $f^{2}$ are kinetic solutions, we have $f^{1,+}_r (x,\xi)=\nu^{1,+}_{r,x} ((\xi,+\infty))$ and $\bar{f}^{2,+}_r (y,\zeta)=\nu^{2,+}_{r,y} ((-\infty, \zeta])$.
With the above properties of $\Upsilon^\ep$ in mind we thus obtain, for all $t\in [0,T]$ and $x,y\in \R^N$,
\begin{align}
\int_{\R^2}  f^{1,+}_r (x,\xi) \bar{f}^{2,+}_r (y,\zeta) \varepsilon^{-
   1}\mathrm{1}_{|\xi - \zeta|\leq 2\ep}\,\dd \xi \dd \zeta 
   & = -  \int_{\R^2} 
   f^{1,+}_r (x,\xi) \bar{f}^{2,+}_r (y,\zeta) (\partial_{\zeta} \partial_{\xi} \Upsilon^\ep)(\xi, \zeta)\,\dd \xi \dd \zeta \notag\\
&=  
   \int_{\R} \int_{\R} f^{1,+}_r (x;\xi)  (\partial_{\xi}\Upsilon^\ep) (\xi, \zeta) \,\dd \xi\, \nu^{2,+}_{r,y} (\dd  \zeta) \notag\\
& = \int_\R \int_\R  \Upsilon^\ep(\xi, \zeta)\,\nu^{1,+}_{r,x} (\dd  \xi)\,\nu^{2,+}_{r,y} (\dd \zeta)   \notag\\
&\lesssim \ep + \int_\R | \xi |\, \nu^{1,+}_{r,x} (\dd \xi)+ \int_\R | \zeta |\, \nu^{2,+}_{r,y} (\dd \zeta) . 
\label{eq:rel-product-f1-f2}
\end{align}
We are now ready to bound $F^{\ep}$ in $\mathcal{E}^{*}_{R,0}$, which is a $L^{1}$-type space. Namely, a simple change of variables yield:
\begin{align*}
\begin{aligned}
\|F^\ep_r\|_{\mathcal{E}^*_{R,0}}&\leq\int_{\R^{N+1}\times\R^{N+1}} F_r^\ep(\mathbf x,\mathbf y) \,\mathbf{1}_{|\mathbf x_-|\leq 1} \,\mathbf{1}_{| x_+|\leq R} \,\dd\mathbf x\dd\mathbf y\\
&\leq\int_{\R^{N+1}\times\R^{N+1}} F_r(\mathbf x,\mathbf y) \ep^{-N-1}\,\mathbf{1}_{|\mathbf x_-|\leq \ep} \,\mathbf{1}_{| x_+|\leq R} \,\dd\mathbf x\dd\mathbf y\\
&\leq\int_{\R^N\times\R^N}\ep^{-N}\,\mathbf{1}_{| x_-|\leq \ep}\,\mathbf{1}_{|x_+|\leq R}\int_{\R^2} F_t(\mathbf x,\mathbf y) \ep^{-1}\,\mathbf{1}_{|\xi_-|\leq \ep}\,\dd\xi\dd\zeta\dd x\dd y.
\end{aligned}
\end{align*}
Hence, thanks to relation \eqref{eq:rel-product-f1-f2}, we get
\begin{align*}
\begin{aligned}
\|F^\ep_r\|_{\mathcal{E}^*_{R,0}}
&\lesssim \int_{\R^N\times\R^N}\ep^{-N}\,\mathbf{1}_{| x_-|\leq \ep}\,\mathbf{1}_{|x_+|\leq R}\bigg(\ep + \int_\R | \xi |\, \nu^{1,+}_{r,x} (\dd \xi)+ \int_\R | \zeta |\, \nu^{2,+}_{r,y} (\dd \zeta)\bigg)\dd x\dd y\\
&\lesssim \ep R^N+\int_{\R^N\times\R^N}\ep^{-N}\,\mathbf{1}_{| x_-|\leq \ep}\bigg( \int_\R | \xi |\, \nu^{1,+}_{r,x} (\dd \xi)+ \int_\R | \zeta |\, \nu^{2,+}_{r,y} (\dd \zeta)\bigg)\dd x\dd y\\
&\lesssim \ep R^N+\int_{\R^N} \int_\R | \xi |\, \nu^{1,+}_{r,x} (\dd \xi)\dd x+ \int_{\R^N}\int_\R | \zeta |\, \nu^{2,+}_{r,y} (\dd \zeta)\dd y,
\end{aligned}
\end{align*}
and the estimate (\ref{bou-m-ep}) follows.
\end{proof}

\subsection{\rmk{Preliminary estimate for the drift term $Q^{\ep}$}}\label{subsec:q-ep}

Let us now proceed to an estimation of the drift term $Q^\ep$ in \eqref{eq:eps1}, where we recall that $Q$ is defined by \eqref{eq:def-Q-distrib}. \rmk{This estimation must fit the pattern of (\ref{cond-mu}) with respect to the smoothing $(\hat{J}^\eta)_{\eta \in (0,1)}$ introduced in Corollary \ref{cor:hat-j}.} To this aim, we set:
 $$
 q^{1}_{t} := \int_{[0,t]}  m^{1}_{\dd r}  \otimes \bar{f}^{2,-}_r 
, \qquad
 \sigma^{1}_{t} := \int_{[0,t]}  m^{1}_{\dd r} \otimes {\nu}^{2,-}_r 
 $$
and in parallel
 $$
 q^{2}_{t} := \int_{[0,t]}  f^{1,+}_{r}  \otimes m^2_{\dd r}
 , \qquad
 \sigma^{2}_{t} := \int_{[0,t]}  \nu^{1,+}_{r} \otimes m^{2}_{\dd r} \ .
 $$
With these notations, it holds true that
\begin{equation}\label{decompo-incr-q-1}
Q^1 =(\partial_\xi\otimes\mathbb{I}) q^1 =   2 \partial_\xi^+ q^1 - \sigma^1 \qquad \text{where} \qquad \partial_\xi^+:=\frac{1}{2}(\partial_\xi\otimes\mathbb{I} + \mathbb{I}\otimes \partial_\xi)  ,
\end{equation}
and in the same way $Q^2 =2 \partial_\xi^+ q^2 + \sigma^2$. We now bound the increments $q^{\ell}$ for $\ell=1,2$ uniformly.

\begin{lemma}\label{lem:3a}
For all $0\leq s\leq t\leq T$ and $\Phi \in \mathcal{E}_{R,k}$, $k=1,2,$ it holds that
\begin{equation*}
\sum_{\ell=1,2}| \delta q_{s t}^{\ell} (\partial_\xi^+T_\varepsilon  \Phi) | \lesssim  \omega_m(s, t)\|\partial_\xi^+\Phi\|_{\mathcal E_{R,0}}  ,
\end{equation*}
where $\ell=1,2,$ and the proportional constant does not depend on $\ep, R$  and the control $\omega_{m}$ is defined as follows
\begin{equation}\label{eq:def-omega-m}
\omega_m (s, t) := \| \bar{f}^2 \|_{L^{\infty}} m^1 ((s, t] \times \mathbb{R}^{N + 1}) + \| f^1\|_{L^{\infty}} m^2 ((s, t] \times \mathbb{R}^{N + 1}).
\end{equation}
\end{lemma}

\begin{proof}
We shall bound $q^{1}(\partial_\xi^+T_\varepsilon  \Phi)$ only, the bound on $q^{2}(\partial_\xi^+T_\varepsilon  \Phi)$ being obtained in a similar way.

\noindent
\textit{Step 1: Bound on $q^{1}$.} 
Consider a test function $\Psi\in\mathcal{E}_{R,0}$, and let us first point out that
\begin{equation*}
|\delta q^1_{st}(\Psi)|  \leq  \int_{\mathbf{x},\mathbf{y}}    \delta q^1_{st}(d\mathbf{x},\mathbf{y}) |\Psi(\mathbf{x},\mathbf{y})|d\mathbf{y},
\end{equation*}
so that the change of variable $\mathbf{x}^{-}=\frac12(\mathbf{x}-\mathbf{y})$ and $\mathbf{x}$ unchanged yields
\begin{eqnarray}\label{eq:delta-q1-Psi-bnd1}
|\delta q^1_{st}(\Psi)| 
& = & 2^{N+1}\int_{\mathbf{x},\mathbf{x}_-}    \delta q^1_{st}(d\mathbf{x},\mathbf{x}-2 \mathbf{x}_-) |\Psi(\mathbf{x},\mathbf{x}-2 \mathbf{x}_-)|d\mathbf{x}_- \notag\\
&\leq &  2^{N+1}   \int_{\mathbf{x}_-} \int_{\mathbf{x}} \delta q^1_{st}(d\mathbf{x},\mathbf{x}-2 \mathbf{x}_-) \sup_{\mathbf{x}}|\Psi(\mathbf{x},\mathbf{x}-2 \mathbf{x}_-)|d\mathbf{x}_- \notag\\
&\leq &  2^{N+1}   \left[ \sup_{\mathbf{x}_-} \int_{\mathbf{x}} \delta q^1_{st}(d\mathbf{x},\mathbf{x}-2 \mathbf{x}_-) \right] \left[ \int_{\mathbf{x}_-} \sup_{\mathbf{x}}|\Psi(\mathbf{x},\mathbf{x}-2 \mathbf{x}_-)|d\mathbf{x}_-\right] \notag\\
&=&  2^{N+1}   \left[ \sup_{\mathbf{x}_-} \int_{\mathbf{x}}\delta q^1_{st}(d\mathbf{x},\mathbf{x}-2 \mathbf{x}_-) \right] \left[ \int_{\mathbf{x}_-} \sup_{\mathbf{x}_+}|\Psi(\mathbf{x}_+ + \mathbf{x}_-,\mathbf{x}_+- \mathbf{x}_-)|d\mathbf{x}_-\right] .
\end{eqnarray}
Furthermore, we have:
\begin{eqnarray*}
\sup_{\mathbf{x}_-} \int_{\mathbf{x}}\delta q^1_{st}(d\mathbf{x},\mathbf{x}-2 \mathbf{x}_-)& \leq & \int_{]s,t]} \sup_{\mathbf{x}_-} \int_{\mathbf{x}} m^1(dr,d\mathbf{x}) |\bar  f^2_r(\mathbf{x}-2 \mathbf{x}_-)|\\
& \leq & \|\bar f^2\|_{L^\infty} \int_{]s,t]\times \R^{N+1}}  m^1(dr,d\mathbf{x})\\
& \leq & \|\bar f^2\|_{L^\infty} m^1((s,t]\times \R^{N+1})  .
\end{eqnarray*}
Reporting this estimate into \eqref{eq:delta-q1-Psi-bnd1} we get:
\begin{equation}\label{estim-q-1}
|\delta q^1_{st}(\Psi)| \lesssim \|\bar f^2\|_{L^\infty} m^1((s,t]\times \R^{N+1}) \|\Psi\|_{L^1_- L^\infty_+},
\end{equation}
where we have introduced the intermediate norm
\begin{equation}\label{eq:def-norm-L1-Linf}
 \| \Psi \|_{L^1_- L^{\infty}_+} := \int_{\R^{N+1}} d \mathbf{x}_-
   \sup_{\mathbf{x}_+} | \Psi (\mathbf{x}_+ +\mathbf{x}_-,
   \mathbf{x}_+ -\mathbf{x}_-) | .
\end{equation}

\noindent
\textit{Step 2: Simple properties of the $L^1_- L^{\infty}_+$-norm.}
We still consider a test function $\Psi\in\mathcal{E}_{R,0}$. Observe that by the basic change of variables $\mathbf{x}_-=\ep^{-1}\mathbf{x}_-$, one has
\begin{eqnarray}\label{eq:ident-T-eps-Psi}
\|T_{\varepsilon} \Psi \|_{L^1_- L^{\infty}_+}& = &\varepsilon^{- N - 1}
   \int_{\R^{N+1}} d \mathbf{x}_- \sup_z | \Psi (z + \varepsilon^{- 1}
   \mathbf{x}_-, z - \varepsilon^{- 1} \mathbf{x}_-) | \label{eq:ident-t-ep-Psi}\\
& =& \int_{\R^{N+1}} d \mathbf{x}_- \sup_z | \Psi (z +\mathbf{x}_-, z
   -\mathbf{x}_-) | \ = \ \| \Psi \|_{L^1_- L^{\infty}_+} \ .  \notag
\end{eqnarray}
In addition, if $\Psi\in\mathcal{E}_{R,0}$, we can use the fact that the support of $\Psi$ is bounded in the  $\mathbf{x}_-$ variable (independently of $R$) in order to get:
\begin{equation*}
\| \Psi \|_{L^1_- L^{\infty}_+}\leq \|\Psi\|_{\mathcal{E}_{R,0}}.
\end{equation*}
   
\noindent
\textit{Step 3: Conclusion.}
As a last preliminary step, notice that
\begin{equation*}
\partial_\xi^+T_\ep=T_\ep \partial_\xi^+,
\end{equation*}
\rmk{where we recall that $\partial_\xi^+$ is defined by \eqref{decompo-incr-q-1}}.
This entails:
\begin{equation*}
|\delta q^1_{st}(\partial^+_\xi T_\varepsilon \Phi)|=|\delta q^1_{st}( T_\varepsilon\partial^+_\xi \Phi)|.
\end{equation*}
Then, applying successively \eqref{estim-q-1} and \eqref{eq:ident-T-eps-Psi} it follows that
\begin{align*}
|\delta q^1_{st}(\partial^+_\xi T_\varepsilon \Phi)|&=|\delta q^1_{st}( T_\varepsilon\partial^+_\xi \Phi)|\\
&\leq \|\bar f^2\|_{L^\infty} m^1((s,t]\times \R^{N+1}) \|\partial^+_\xi \Phi\|_{L^1_- L^\infty_+}\\
&\lesssim \|\bar f^2\|_{L^\infty} m^1((s,t]\times \R^{N+1})\|\partial^+_\xi\Phi\|_{\mathcal{E}_{R,0}},
\end{align*}
which is our claim.
\end{proof}

\rmk{We are now ready to establish our main estimate on $Q^\ep$.}

\rmk{\begin{proposition}\label{prop:3}
Let $Q$ be defined by \eqref{eq:def-Q-distrib}, $Q^{\ep}:= T^{\ast}_\varepsilon Q$ and let $(\hat{J}^\eta)_{\eta\in (0,1)}$ be the set of smoothing operators introduced in Corollary \ref{cor:hat-j}. Then for all $0\leq s\leq t\leq T$ and $\Phi \in \mathcal{E}_{R,k}$, $k=1,2,$ it holds that
\begin{eqnarray}
& & | \delta Q_{s t}^{\varepsilon} (\hat{J}^\eta \Phi) | 
\lesssim  
\omega_m(s, t)\|\Phi\|_{\mathcal E_{R,1}} + \eta^{k-3}\,
\omega_{\sigma, \varepsilon,R} (s, t) \| \Phi \|_{\mathcal{E}_{R,k}},  \label{bnd-q-ep-theta}\\
& & \delta Q_{s t}^{\varepsilon} (\Psi_R) + \omega_{\sigma, \varepsilon,R} (s, t)   \lesssim
\omega_m(s,t)\| \partial_\xi^+ \Psi_R\|_{\mathcal{E}_{R,0}} .
\label{appli-1-q-ep-theta} 
\end{eqnarray}
where the proportionality constants do not depend on $\ep$, $\eta$ and $R$, and $\Psi_R$ is the function introduced in Proposition \ref{prop:1}. In \eqref{bnd-q-ep-theta} and \eqref{appli-1-q-ep-theta}, we also have  $\omega_{m}$ given by \eqref{eq:def-omega-m} and the control $\omega_{\sigma, \varepsilon,R}$ is defined as 
\begin{equation}\label{eq:def-om-sigma}
\omega_{\sigma, \varepsilon,R}(s,t):=\delta\sigma^1_{st}(T^\ep \Psi_R)+\delta\sigma^2_{st}(T^\ep \Psi_R) \ .
\end{equation}
\end{proposition}}

\rmk{\begin{remark}
Although it seems purely technical, inequality \eqref{appli-1-q-ep-theta} sets the stage for our contraction argument yielding uniqueness. Namely, the proper control we need for the measure term of our equation will stems from the fact that the control $\omega_{\sigma, \varepsilon,R}$  appears in a "good" form in the l.h.s. of \eqref{appli-1-q-ep-theta}. This damping effect is reminiscent of \eqref{estim-toy} for the heat equation model.
\end{remark}}

\rmk{\begin{remark}
Observe that  the bound (\ref{bnd-q-ep-theta}) (which will serve us in the forthcoming application of Theorem \ref{theo:apriori}) still depends on both parameters $\ep$ and $R$. At this point we are not systematically looking for uniformity but only for bounds that we will be able to control afterwards, via the Gronwall-type arguments of Section \ref{subsec:diago}. The situation here can somehow be compared with our use of the (non-uniform) estimate (\ref{eq:estim-mu-heat}) in the example treated in Section \ref{subsec:first-application}.   
\end{remark}}

\begin{proof}
Recall that $Q$ is written as $Q^{1}-Q^{2}$ in \eqref{eq:def-Q-distrib}. We focus here on the estimate for $Q^{1}$. Furthermore, owing to \eqref{decompo-incr-q-1} we have 
\begin{equation*}
\delta Q^{1,\ep}_{st}(\rmk{\hat{J}^\eta}\Phi)
=
T_{\ep}^{*} \delta Q^{1}_{st}(\rmk{\hat{J}^\eta}\Phi)
=
\delta Q^{1}_{st}(T_{\ep}\rmk{\hat{J}^\eta}\Phi)
=
2 \delta Q_{st}^{11,\ep} - \delta Q_{st}^{12,\ep},
\end{equation*}
where
\begin{equation*}
\delta Q_{st}^{11,\ep}
=
\partial_{\xi}^{+}\delta q_{st}^{1}(T_{\ep}\rmk{\hat{J}^\eta}\Phi),
\quad\text{and}\quad
\delta Q_{st}^{12,\ep}
=
\delta \sigma_{st}^{1}(T_{\ep}\rmk{\hat{J}^\eta}\Phi).
\end{equation*}
Now thanks to Lemma~\ref{lem:3a}, we have that
\begin{equation*}
|\delta Q_{st}^{11,\ep}|
=
|\delta q^1_{st}(\partial^+_\xi T_\varepsilon \rmk{\hat{J}^\eta}\Phi)|
\lesssim \|\bar f^2\|_{L^\infty} m^1((s,t]\times \R^{N+1})
\|\partial^+_\xi  \rmk{\hat{J}^\eta}\Phi\|_{\mathcal{E}_{R,0}}
\end{equation*}
Moreover, invoking the fact that $J^{\eta}$ is a bounded operator in $\mathcal{E}_{1}$ plus inequality~\eqref{1}, we get:
\begin{equation*}
\|\partial^+_\xi  \rmk{\hat{J}^\eta}\Phi\|_{\mathcal{E}_{R,0}}
\lesssim
\|  \rmk{\hat{J}^\eta}\Phi\|_{\mathcal{E}_{R,1}}
\lesssim
\|  \Phi\|_{\mathcal{E}_{R,1}},
\end{equation*}
which entails the following relation:
\begin{equation*}
|\delta Q_{st}^{11,\ep}|
\lesssim \|\bar f^2\|_{L^\infty} m^1((s,t]\times \R^{N+1})\|\Phi\|_{\mathcal{E}_{R,1}}.
\end{equation*}
As far as the term $\delta Q_{st}^{12,\ep}$ is concerned, we make use of \eqref{3} to deduce
\begin{equation*}
\eta^{3-k} |\delta Q_{st}^{12,\ep}|
=
 \eta^{3-k}| \delta \sigma^1_{s t} (T_{\varepsilon} \rmk{\hat{J}^\eta}\Phi) | \leqslant 
  \eta^{3-k}\delta \sigma^1_{s t} (T_{\varepsilon}  |\rmk{\hat{J}^\eta} \Phi |) \leqslant \| \Phi
   \|_{\mathcal{E}_{R,k}}\, \omega_{\sigma, \varepsilon,R}(s,t)  . 
	\end{equation*}
Putting together our bound on $\delta Q_{st}^{11,\ep}$ and $\delta Q_{st}^{12,\ep}$, we thus get:
$$
 | \delta Q_{s t}^1 (T_{\varepsilon}\rmk{\hat{J}^\eta}\Phi) | \lesssim  
	 \omega_m(s, t)\|\Phi\|_{\mathcal E_{R,1}} + \eta^{k-3}\,\omega_{\sigma, \varepsilon,R} (s, t) \| \Phi \|_{\mathcal{E}_{R,k}}  ,
$$
and along the same lines, we can prove that
$$| \delta Q_{s t}^2 (T_{\varepsilon} \rmk{\hat{J}^\eta}\Phi) | \lesssim \omega_m(s, t)\|\Phi\|_{\mathcal E_{R,1}} + \eta^{k-3}\,\omega_{\sigma, \varepsilon,R} (s, t) \| \Phi \|_{\mathcal{E}_{R,k}}  ,$$
which achieves the proof of our assertion \eqref{bnd-q-ep-theta}.

The second claim \eqref{appli-1-q-ep-theta} is obtained as follows: we start from relation (\ref{decompo-incr-q-1}), which yields:
\begin{eqnarray}\label{a0}
\delta Q_{s t}^{\varepsilon} (\Psi_R)  = \delta Q_{s t} (T_{\varepsilon}
   \Psi_R) &=& 
   -  \delta\sigma^1_{s t} (T_{\varepsilon} \Psi_R) - 
  \delta \sigma^2_{s t} (T_{\varepsilon}\Psi_R)+  \delta q^1_{s t} (T_{\varepsilon} \partial_\xi^+ \Psi_R) +  \delta q^2_{s t} (T_{\varepsilon} \partial_\xi^+ \Psi_R) \notag \\
  &\leq&
- \omega_{\sigma, \varepsilon,R} (s, t) 
+  \delta q^1_{s t} (T_{\varepsilon} \partial_\xi^+ \Psi_R) 
+  \delta q^2_{s t} (T_{\varepsilon} \partial_\xi^+ \Psi_R)
\end{eqnarray}
We can now proceed as for \eqref{bnd-q-ep-theta} in order to bound $\delta q^1_{s t} (T_{\varepsilon} \partial_\xi^+ \Psi_R)$ and $\delta q^2_{s t} (T_{\varepsilon} \partial_\xi^+ \Psi_R)$ above, and this immediately implies \eqref{appli-1-q-ep-theta}.
\end{proof}

\subsection{Passage to the diagonal}\label{subsec:diago}
\rmk{Thanks to the results of Sections \ref{subsec:smoothing}-\ref{subsec:q-ep}, we are now in a position to efficiently apply Theorem \ref{theo:apriori} to the transformed equation (\ref{eq:eps1}). To be more specific, we study this equation on the scale $(\mathcal{E}_{R,n})_{0\leq n\leq 3}$ defined in (\ref{recall-e-r-n}) (for fixed $R\geq 1$) and consider the smoothing $(\hat{J}^\eta)_{\eta\in(0,1)}$ given by Corollary \ref{cor:hat-j}. At this point, let us also recall that the driver $\mathbf{A}$ governing the original equation is known to be renormalizable with respect to the scale $(\mathcal{E}_{R,n})$: this was the content of Proposition \ref{prop:renorm}, which provides us with the two bounds (\ref{prop:renorm-1})-(\ref{prop:renorm-2}) (uniform in both $\ep$ and $R$) for the tensorized driver $\mathbf{\Gamma}^\ep$.}

\

\rmk{By injecting these considerations, together with the results of Proposition \ref{prop:2} and Proposition \ref{prop:3}, into the statement of Theorem \ref{theo:apriori}, we immediately obtain the following important assertion about the remainder $F^{\natural,\ep}$ in equation \eqref{eq:eps1}: there exists a constant $L>0$ such that if $\omega_Z(I)\leq L$, one has, for all $s< t\in I$,
\begin{equation}\label{estim-f-natural-ep}
\begin{split}
 \|F^{\natural,\ep}_{st}\|_{\mathcal{E}^*_{R,3}}&\lesssim \omega_{\ast,\ep,R}(s,t)^{\frac{3}{q}}\\
 &\quad\ := \mathcal{M}(s,t,\ep,R)\,\omega_Z(s,t)^{\frac{3}{p}-2\kappa}
+\omega_m(s,t)\omega_Z(s,t)^\frac{1}{p}+\omega_{\sigma,\ep,R}(s,t) \omega_Z(s,t)^{\kappa }\ ,
\end{split}
\end{equation}
where the proportional constant is independent of $\ep$ and $R$, the quantities $\mathcal{M},\omega_Z,\omega_{m},\omega_{\sigma,\ep,R}$ are respectively defined by (\ref{defi-mathcal-m}), Hypothesis \ref{hyp:z}, \eqref{eq:def-omega-m} and \eqref{eq:def-om-sigma}, and we have fixed (once and for all) the parameters $q,\ka$ such that
\begin{equation*}
q\in\left[\frac{9p}{2p+3},3\right)\quad , \quad \kappa\in \left[\frac{3}{q}-1,\frac12 \left(\frac{3}{p}-\frac{3}{q}\right)\right] \ .
\end{equation*}}

As a first consequence of \rmk{estimate \eqref{estim-f-natural-ep}}, we can derive the following bound on the limit of $\omega_{\sigma,\ep,R}(s,t)$ as $\ep \to 0$ and $R\to\infty$.

\begin{lemma}\label{lem:conv-mu-theta-ep}
Let $\omega_{\sigma,\ep,R}$ be the control defined by \eqref{eq:def-om-sigma}.
There exists a finite measure $\mu$ on $[0,T]$ such that, for all $0\leq s\leq t\leq T$,
\begin{equation}\label{limsup-omega-ep}
\limsup_{R\to\infty}\limsup_{\ep \to 0}\, \omega_{\sigma,\ep,R}(s,t) \leq \mu([s,t]) .
\end{equation}
\end{lemma}

\begin{proof}
Consider the sequence of measures $(\mu_{R}^\ep)_{\ep>0}$ on $[0,T]$ defined for every Borel set $B \subset [0,T]$ as
\begin{equation}\label{eq:def-mu-ep-R}
\mu^\ep_{R}(B):=\bigg( \int_B m^1_{dr} \otimes \nu^{2,-}_r \bigg)\big(T_\ep \Psi_R\big)+\bigg( \int_B \nu^{1,+}_r \otimes m^2_{dr} \bigg)\big( T_\ep \Psi_R\big) ,
\end{equation}
so that $\omega_{\sigma,\ep,R}(s,t)=\mu^\ep_R((s,t])$.
By applying equation \eqref{eq:eps1} to the test function $\Psi_R$ and using~(\ref{appli-1-q-ep-theta}), we get that for every $s<t\in [0,T]$,
\begin{equation*}
\delta F^{\varepsilon}_{s t} (\Psi_R) \leq F^{\varepsilon}_s\big((\Gamma_{\varepsilon,st}^{1,*} + \Gamma_{\varepsilon,st}^{2,*})(\Psi_R)\big) -  \omega_{\sigma, \varepsilon,R} (s, t) +F^{\natural, \varepsilon}_{s t} (\Psi_R) +\omega_m(s,t)\| \partial_\xi^+ \Psi_R\|.
\end{equation*}
and so
\begin{equation*}
\omega_{\sigma, \varepsilon,R} (s, t) \leq 
F^{\varepsilon}_s
(|(1+\Gamma_{\varepsilon,st}^{1,*} + \Gamma_{\varepsilon,st}^{2,*})\Psi_R|) 
+|F^{\natural, \varepsilon}_{s t} (\Psi_R) | +\omega_m(s,t)\| \partial_\xi^+ \Psi_R\|.
\end{equation*}
Therefore, due to \rmk{Proposition \ref{prop:2}, estimate \eqref{estim-f-natural-ep}} and assumption \eqref{integrov}, we can conclude that for every interval $I\subset[0,T]$ satisfying $\omega_Z(I)\leq L$, one has
\begin{align*}
\omega_{\sigma, \varepsilon,R} (I) \lesssim \ep \rmk{R^N}+1 +\omega_m(I)(1+\omega_Z(I)^\frac{1}{p})+\omega_{\sigma,\ep,R}(I) \omega_Z(I)^{\kappa },
\end{align*}
for some proportional constant independent of $\ep, \,R$. As a consequence, there exists $0<L'\leq L$ such that for every interval $I\subset [0,T]$ satisfying $\omega_Z(I)\leq L'$, it holds
$$
\omega_{\sigma, \varepsilon,R} (I) \lesssim \ep \rmk{R^N}+1 +\omega_m(I).
$$
By uniformity of both $L'$ and the proportional constant, the latter bound immediately yields
\begin{equation}\label{lim}
\omega_{\sigma, \varepsilon,R} (0,T) \lesssim \ep \rmk{R^N}+1 +\omega_m(0,T).
\end{equation}
Thus, the sequence $(\mu^\ep_{R})_{\ep>0}$ defined by \eqref{eq:def-mu-ep-R} is bounded in total variation on $[0,T]$ and accordingly, by Banach-Alaoglu theorem, there exists a subsequence, still denoted by $(\mu^\ep_{R})_{\ep>0}$, as well as a finite measure $\mu_{R}$ on $[0,T]$ such that for every $\vp \in C([0,T])$, one has
$$\mu_{R}^\ep(\vp) \to \mu_{R}(\vp)\qquad  \text{as $\ep\to0$.}$$
Moreover, as a straightforward consequence of \eqref{lim}, we get
$$\mu_{R}([0,T]) \lesssim 1+\omega_m(0,T). $$
Therefore $(\mu_R)_{R\in\N}$ is bounded in total variation and there exists a finite measure $\mu$ on $[0,T]$ satisfying
$$\mu([0,T]) \lesssim 1+\omega_m(0,T), $$
such that, along a subsequence,
$$\mu_R(\vp)\to\mu(\vp),\qquad\forall\,\vp\in C([0,T]),\qquad R\to\infty.$$

Finally, due to the properties of $BV$-functions, for every $R\in\N$, there exists an at most countable set $\mathcal{D}_R$ such that the function $t\mapsto \mu_{R}(]0,t])$ is continuous on $[0,T]\setminus\mathcal{D}_R$. Furthermore, by Portmanteau theorem,  one has
$$\mu^\ep_{R}(]0,t]) \to \mu_{R}(]0,t])\qquad\forall\,t\in [0,T]\setminus\mathcal D_R\qquad\ep\to 0.$$
Similarly, there exists a countable set $\mathcal D$ such that
$$\mu_R(]0,t])\to\mu(]0,t])\qquad\forall\,t\in [0,T]\setminus\mathcal D\qquad R\to 0.$$
Fix $s<t\in [0,T]$. Since a countable union of countable sets is countable, we may consider a sequence $(s_k)$, resp. $(t_k)$, of points outside of $\cup_R \mathcal{D}_R\cup \mathcal{D}$ that increase, resp. decrease, to $s$, resp. $t$, as $k$ tends to infinity. Then
\begin{align*}
\limsup_{\ep \to 0}\, \omega_{\sigma,\ep,R}(s,t)=\limsup_{\ep \to 0}\, \mu_{R}^\ep(]s,t])\leq \limsup_{\ep \to 0}\, \mu_{R}^\ep(]s_k,t_k])=\mu_{R}(]s_k,t_k])  
\end{align*}
and
\begin{align*}
\limsup_{R\to\infty}\limsup_{\ep \to 0}\, \omega_{\sigma,\ep,R}(s,t)\leq\limsup_{R\to\infty}\mu_{R}(]s_k,t_k]) =\mu(]s_k,t_k]) .
\end{align*}
By letting $k$ tend to infinity, we get (\ref{limsup-omega-ep}), which achieves the proof of the lemma.
\end{proof}

We are now ready to prove our main intermediate result towards uniqueness.

\begin{proposition}\label{prop:convergence}
Consider $\psi\in C^\infty_c(\R^{N+1})$ such that
$$\psi\geq0,\qquad\supp\,\psi\subset B_{2/\sqrt{2}},\qquad\int_{\R^{N+1}}\psi(\mathbf x)\,\dd\mathbf x =1.
$$
Let also $ \{\vp_R; R>0\} \subset C^\infty_c(\R^{N+1})$ be a family of smooth functions such that
$$\vp_R \geq0,\qquad\supp\,\vp_R \subset B_{R/\sqrt{2}}, \qquad \sup_R \|\vp_R \|_{W^{3,\infty}}\lesssim 1.  $$
We define 
\begin{equation}\label{eq:def-Phi-R}
\Phi_{R}(\mathbf{x},\mathbf{y})= \vp_{R}(\mathbf{x}_{+}) \psi(2\mathbf{x}_{-}).
\end{equation}
Then
for every $0\leq s\leq t\leq T$, it holds true that, \rmk{as $\varepsilon \to 0$,}
\begin{equation}\label{con}
 F^{\varepsilon}_{t} (\Phi_R) 
   \rightarrow  h_{t} (\varphi_R),
   \end{equation}
\begin{equation}\label{con-2}
(\Gamma^1_{ \varepsilon,st} + \Gamma^2_{
   \varepsilon,st})_{} F_{s}^{\varepsilon} (\Phi_R)  \rightarrow (A^1_{st} + A^2_{st})_{} h_s
   (\varphi_R), 
   \end{equation}
where $h_t :=  f^{1,+}_t \bar f^{2,+}_t$.
\end{proposition}

\begin{proof}
Consider first a function $\Psi$ supported in $D_{R}\equiv B_{R+1}\times B_{R+1} \subset \R^{N+1}\times\R^{N+1} $. Then for all functions $v^1, v^2$ we have:
\begin{align*}
|v^1 \otimes v^{2}(\Psi)|&=\bigg|\int_{D_{R}} v^{1}(\mathbf x_++\mathbf x_-) v^{2}(\mathbf x_+-\mathbf x_-)\Psi(\mathbf x_++\mathbf x_-,\mathbf x_+-\mathbf x_-)\dd\mathbf x_-\dd\mathbf x_+\bigg|\\
&\leq \int_{D_{R}} v^{1}(\mathbf x_++\mathbf x_-) v^{2}(\mathbf x_+-\mathbf x_-)\dd\mathbf x_+\sup_{\mathbf y_+}|\Psi(\mathbf y_++\mathbf x_-,\mathbf y_+-\mathbf x_-)|\dd\mathbf x_- .
\end{align*}
Recalling our definition \eqref{eq:def-norm-L1-Linf}, $|v^1 \otimes v^{2}(\Psi)|$ can be further estimated in two ways: on the one hand we have
\begin{align*}
|v^1 \otimes v^{2}(\Psi)|&\leq\|v^{1}\|_{L^1(B_{R+1})} \|v^{2}\|_{L^\infty(B_{R+1})} \|\Psi\|_{L^1_-L^\infty_+},
\end{align*}
and on the other hand we also get
\begin{align*}
|v^1 \otimes v^{2}(\Psi)|&\leq\|v^{2}\|_{L^1(B_{R+1})} \|v^{1}\|_{L^\infty(B_{R+1})}\|\Psi\|_{L^1_-L^\infty_+}.
\end{align*}

In order to apply this general estimate, define a new test function $\Phi_R (\mathbf{x},\mathbf{y}) = \varphi_R (\mathbf{x}_+) \psi
(2\mathbf{x}_-)$, and observe that $\Phi_R$ is compactly supported in the set $D_{R}$.
Since $|f^{1,+}_t|\leq 1$, $| \bar f^{2,+}_t|\leq 1$ it follows that $f^{1,+}_t, \bar  f^{2,+}_t\in L^1(B_{R+1})$ (notice that the localization procedure is crucial for this step). Therefore one may find $g^1,g^2\in C^\infty_c(\R^{N+1})$ such that $|g^1|\leq 1$, $|g^2|\leq 1$ and
\begin{equation}\label{eq:approx-g1-g2}
\|f^{1,+}_t-g^1\|_{L^1(B_{R+1})}+\|\bar f^{2,+}_t-g^2\|_{L^1(B_{R+1})}\leq \delta.
\end{equation}
We now split the difference $F_{t}^{\ep}-h_{t}$ as follows:
\begin{multline}\label{eq:split-Fep-h}
|F^{\varepsilon}_{t} (\Phi_R) -  h_{t} (\varphi_R)| \\
\leq
\left|F^{\varepsilon}_{t} (\Phi_R) -  \lp g_{1} \otimes g_{2} \rp^{\ep}(\Phi_R)\right|
+
\left| \lp g_{1} \otimes g_{2} \rp^{\ep}(\Phi_R) - (g^1g^2) (\vp_R)\right|
+
\left|  (g^1g^2) (\vp_R) - h_{t} (\varphi_R) \right|.
\end{multline}
We shall bound the three terms of the right hand side above separately. Indeed, owing to~\eqref{eq:approx-g1-g2} and the fact that $|g^1|\leq 1$, $|g^2|\leq 1$ \rmk{and $\|T_\ep \Phi_R\|_{L^1_-L^\infty_+}\lesssim 1$ uniformly in $R,\ep$}, we have
\begin{align}\label{eq:cvgce-F-ep-1}
\begin{aligned}
|F^\ep_t(\Phi_R)&-(g^1\otimes g^2)^\ep(\Phi_R)| \leq 
|(g^1\otimes (\bar f^{2,+}_t -g^2))(T_\ep \Phi_R)| +|((f^{1,+}_t -g^1)\otimes \bar f^{2,+}_t)(T_\ep \Phi_R)|
\\
& \lesssim \|g^1\|_{L^\infty(B_{R+1})}  \|\bar f^{2,+}_t -g^2\|_{L^1(B_{R+1})} +  \|\bar f^{2,+}_t\|_{L^\infty(B_{R+1})}  \|f^{1,+}_t -g^1\|_{L^1(B_{R+1})} \lesssim \delta
\end{aligned}
\end{align}
On the other hand, using the continuity of $g^1,g^2$ we have
$$
\lim_{\ep\to 0}(g^1\otimes g^2)(T_\ep \Phi_R) = (g^1g^2) (\vp_R),
$$
and thus $| \lp g_{1} \otimes g_{2} \rp^{\ep}(\Phi_R) - (g^1g^2) (\vp_R)|\le\delta$ for $\ep$ small enough.
Moreover, as in \eqref{eq:cvgce-F-ep-1}, we have
\begin{align*}
|(g^1g^2) (\vp_R) & - h_t (\vp_R)| \\ & \leq \|g^1\|_{L^\infty(B_{R+1})}  \|\bar f^{2,+}_t -g^2\|_{L^1(B_{R+1})} +  \|\bar f^{2,+}_t\|_{L^\infty(B_{R+1})}  \|f^{1,+}_t -g^1\|_{L^1(B_{R+1})}
\\ & \leq \delta. 
\end{align*}
Since $\delta$ is arbitrary we have established~\eqref{con}. 

Let us now turn to \eqref{con-2}.
Observe that by \rmk{Proposition~\ref{prop:renorm}} we have that $T_\ep \Gamma^{1,*}_\ep \Phi_R$ and $T_\ep \Gamma^{2,*}_\ep \Phi_R$ are bounded uniformly in $\ep$ in $L^1_- L^\infty_+$. Specifically, we have:
\begin{align*}
\|T_\ep \Gamma^{1,*}_\ep \Phi_R\|_{L^1_- L^\infty_+} & \leq \| \Gamma^{1,*}_\ep \Phi_R\|_{L^1_- L^\infty_+} = \| \Gamma^{1,*}_\ep \Phi_R\|_{L^1_-(B_1; L^\infty_+(B_R))}
\\ & \leq \| \Gamma^{1,*}_\ep \Phi_R\|_{\mathcal{E}_{R,0}}
\lesssim_V \|  \Phi_R\|_{\mathcal{E}_{R,1}}
\end{align*}
where we used in order the boundedness of $T_\ep$ in $L^1_- L^\infty_+$, the compact support of $\Phi_R$ to go from $L^1$ to $L^\infty$ and finally the renormalizability of $\bf{A}$ in the spaces $(\mathcal{E}_{R,n})_n$ \rmk{(as provided by Proposition~\ref{prop:renorm})}. The same reasoning applies to $T_\ep \Gamma^{2,*}_\ep \Phi_R$.
Similarly to \eqref{eq:split-Fep-h}, in order to establish the limit of $F^\ep_s(\Gamma^{j,*}_\ep \Phi_R)$ for $j=1,2$, it is enough to  consider $(g^1 \otimes g^2)(T_\ep \Gamma^{1,*}_\ep \Phi_R)$ and $(g^1 \otimes g^2)(T_\ep \Gamma^{2,*}_\ep \Phi_R)$ \rmk{for $g^{1},g^{2}$ as in~\eqref{eq:approx-g1-g2}}. 
\rmk{Now, recalling the very definition (\ref{eq:def-F-Q-eps}) of $\mathbf{\Gamma}_\ep^\ast$,}
\begin{align*}
& (g^1 \otimes g^2)(T_\ep \Gamma^{1,*}_\ep \Phi_R) = (g^1 \otimes g^2)( \Gamma^{1,*} T_\ep \Phi_R) 
=(\Gamma^{1}(g^1 \otimes g^2))(  T_\ep \Phi_R) 
\\ & \quad =(A^1 g^1 \otimes g^2 + g^1 \otimes A^1 g^2 )(  T_\ep \Phi_R)
\end{align*}
and hence we end up with:
\begin{equation*}
\lim_{\ep\to\infty} 
(g^1 \otimes g^2)(T_\ep \Gamma^{1,*}_\ep \Phi_R)
=
((A^1 g^1)  g^2 + g^1 (A^1 g^2) )( \vp_R) = ( g^1 g^2)(A^{1,*} \vp_R). 
\end{equation*}
\rmk{Similarly we have}:
\begin{align*}
& (g^1 \otimes g^2)(T_\ep \Gamma^{2,*}_\ep \Phi_R) = (g^1 \otimes g^2)( \Gamma^{2,*} T_\ep \Phi_R) 
=(\Gamma^{2}(g^1 \otimes g^2))(  T_\ep \Phi_R) 
\\ & \quad =(A^2 g^1 \otimes g^2 + g^1 \otimes A^2 g^2 + A^1 g^1 \otimes A^1 g^2)(  T_\ep \Phi_R)
\end{align*}
Therefore we obtain:
\begin{equation*}
\lim_{\ep\to\infty} 
(g^1 \otimes g^2)(T_\ep \Gamma^{2,*}_\ep \Phi_R) 
=
((A^2 g^1)  g^2 + g^1 (A^2 g^2) + (A^1 g^1) (A^1 g^2) )( \vp_R) = ( g^1 g^2)(A^{2,*} \vp_R). 
\end{equation*}
This finishes the proof of \eqref{con-2}. 
\end{proof}

The following contraction principle is the main result of this section. By considering two equal initial conditions, it yields in particular our desired uniqueness result for generalized solutions of equation~\eqref{eq1}.

\begin{proposition}\label{prop:compare}
Let $f^1$ and $f^2$ be two generalized kinetic solutions of~\eqref{eq1} with initial conditions $f^1_0$ and $f^2_0$. Assume that $f^{1}_0 \bar f^{2}_0 \in L^1(\R^{N+1})$ then 
$$
\sup_{t\in[0,T]} \|f^{1,+}_t\bar f^{2,+}_t\|_{L^1(\R^{N+1})} \le  \|f^{1}_0 \bar f^{2}_0\|_{L^1(\R^{N+1})}. 
$$ 
\end{proposition}

\begin{proof}  
Our global strategy is to take limits in \eqref{eq:eps1} in order to show the \rmk{contraction} principle. We now divide the proof in several steps.

\noindent
\textit{Step 1: Limit in $\ep$.}
Recall that $\Phi_R$ has been defined by \eqref{eq:def-Phi-R}. Applying \eqref{eq:eps1} to the test function $\Phi_{R}$ yields:
\begin{equation}\label{eq;dynamics-test-Phi-R}
\delta F^\varepsilon_{st}(\Phi_{R})= \delta Q^\varepsilon_{st}(\Phi_{R})
+ F_s^\varepsilon((\Gamma^{1,*}_{\varepsilon,st}+\Gamma^{2,*}_{\varepsilon,st})\Phi_{R})
+F_{st}^{\natural,\varepsilon} (\Phi_{R}).
\end{equation}
Furthermore, similarly to \eqref{appli-1-q-ep-theta}, we have that
\[ \delta Q_{st}^{\varepsilon} (\Phi_R)  \leq - \delta \sigma^1_{st} (T_{\varepsilon} \Phi_R) 
- \delta \sigma_{st}^2 (T_{\varepsilon} \Phi_R) 
+ \omega_m(s,t) \|\partial_\xi^
+ \Phi_R \|_{\mathcal{E}_{R,0}} \leq 
\|\partial_\xi^{+} \Phi_R \|_{\mathcal{E}_{R,0}}  \, \omega_m(s,t) . 
\]
\rmk{By using Proposition~\ref{prop:convergence}, we can take limits in  relation \eqref{eq;dynamics-test-Phi-R}, which, together with (\ref{estim-f-natural-ep}), gives the following bound for the increments of the path $h_t =  f^{1,+}_t \bar f^{2,+}_t$: for every interval $I$ such that $\omega_Z(I)\leq L$ and all $s<t\in I$,}
\begin{equation}\label{eq:bnd-delta-h-1}
\delta h_{st} (\vp_R) \leq (A^1_{s t} + A^2_{s t})_{} h_s
   (\vp_R) +\limsup_{\ep\to0} \, \omega_{\ast,\ep,R}(s,t) ^{\frac{3}{q}} +  \|\partial_\xi^+ \Phi_R \|_{\mathcal{E}_{R,0}} \omega_m(s,t).
\end{equation}
where $\omega_{\ast,\ep,R}$ is the control defined in \eqref{estim-f-natural-ep}. Application of Proposition~\ref{prop:2} and Lemma~\ref{lem:conv-mu-theta-ep} gives a uniform bound in $R$ on $\limsup_{\ep\to0}\omega_{\ast,\ep,R}(s,t)$ in terms of the control $\omega_\natural$  given by
\begin{equation}\label{eq:def-omega-natural}
\omega_{\natural}(s,t) ^{\frac{3}{q}}
= \omega_Z(s,t)^{\frac{3}{p}-2\kappa}
+\omega_m(s,t)\omega_Z(s,t)^\frac{1}{p}+\mu([s,t]) \omega_Z(s,t)^{\kappa }.
\end{equation}
Namely, we have
$
\limsup_{\ep\to0}\omega_{\ast,\ep,R}(s,t) \leq \omega_{\natural}(s,t),
$
hence we can recast inequality \eqref{eq:bnd-delta-h-1} as:
\begin{equation}\label{eq:h-test-phi-R}
\delta h_{st} (\vp_R) \leq (A^1_{s t} + A^2_{s t})_{} h_s(\vp_R) 
+\omega_{\natural}(s,t)^{\frac{3}{q}} 
+  \|\partial_\xi^+ \Phi_R \|_{\mathcal{E}_{R,0}} \, \omega_m(s,t) \ ,
\end{equation}
\rmk{for all $s<t\in I$ with $\omega_Z(I)\leq L$.}

\smallskip

\noindent
\textit{Step 2: Uniform $L^{1}$ bounds.}
We now wish to test the increment $\delta h_{st}$ against the function $\1(x,\xi)= 1$ in order to get uniform (in $t$) $L^{1}$ bounds on $h$. This should be obtained by taking the limit $R\to\infty$ in~\eqref{eq:h-test-phi-R}. However, the difficulty here is the estimation of the term $(A^1_{s t} + A^2_{s t}) h_s
   (\vp_R)$, uniformly in $R$. To circumvent this problem, we want to choose another test function which is easier to estimate but with unbounded support. Namely, instead of the function $\vp_{R}$ of~Proposition \ref{prop:convergence}, let us consider a function $\vp_{R,L}(\mathbf{x}_+)=\vp(\mathbf{x}_+/R) \psi_L(\mathbf{x}_+)$. In this definition $\psi_L(\mathbf{x}_+) =\psi(\mathbf{x}_+/L)$ with $\psi(\mathbf{x}_+) = (1+|\mathbf{x}_+|^2)^{-M}$ for $M>(N+1)/2$, and $\vp$ is a smooth compactly supported function with $\vp |_{B_{1/4}}=1$.  

With these notations in hand, relation \eqref{eq:h-test-phi-R} is still satisfied for the function $\vp_{R,L}$:
\begin{equation}\label{eq:h-test-phi-R-L}
\delta h_{st} (\vp_{R,L}) \leq (A^1_{s t} + A^2_{s t}) h_s(\vp_{R,L}) 
+\omega_{\natural}(s,t)^{\frac{3}{q}} 
+  \|\partial_\xi^+ \Phi_{R,L} \|_{\mathcal{E}_{R,0}} \, \omega_m(s,t),
\end{equation}
where $\Phi_{R,L}$ is defined similarly to \eqref{eq:def-Phi-R}. We can now take limits as $R$ goes to infinity in \eqref{eq:h-test-phi-R-L}. That is, since $A^{j,*}_{s t}\vp_{R,L}$ is an element of $L^{1}$, uniformly in $R$ and for for $j=1,2$, we have
\begin{equation*}
\lim_{R\to\infty} (A^1_{s t} + A^2_{s t}) h_s(\vp_{R,L})
= (A^1_{s t} + A^2_{s t}) h_s(\psi_{L})
\end{equation*}
In addition, \rmk{according to the definition \eqref{decompo-incr-q-1} of $\partial_\xi^+$,} we have that if $g=g(\mathbf{x}_{+})$ then $\partial_{\xi}^{+}g=g'(\mathbf{x}_{+})$, while $\partial_{\xi}^{+}g=0$ whenever $g=g(\mathbf{x}_{-})$. Therefore, it is readily checked that:
$$\partial_\xi^+ \Phi_{R,L} (\mathbf{x},\mathbf{y}) 
=R^{-1} (\partial_\xi \varphi) (\mathbf{x}_+/R) \psi(\mathbf{x}_+/L) \psi
(2\mathbf{x}_-) + L^{-1} \varphi (\mathbf{x}_+/R) (\partial_\xi\psi)(\mathbf{x}_+/L) \psi
(2\mathbf{x}_-),$$ 
and thus $\|\partial_\xi^+ \Phi_{R,L} \|_{\mathcal{E}_{R,0}} \lesssim( R^{-1}+L^{-1})$.
 Hence, invoking our last two considerations, we can take limits as $R\to\infty$ in relation~\eqref{eq:h-test-phi-R-L} to deduce, \rmk{for all $s< t\in I$ with $\omega_Z(I)\leq L$,}
\begin{equation}\label{eq:bnd-delta-h-2}
\delta h_{st} (\psi_L) \lesssim  (A^1_{s t} + A^2_{s t}) h_s(\psi_L)
+\omega_{\natural}(s,t)^{\frac{3}{q}}
+ C L^{-1}\omega_m(s,t).
\end{equation}

We are now in a position to apply our Gronwall type Lemma \ref{lemma:basic-rough-gronwall} to relation~\eqref{eq:bnd-delta-h-2}. To this aim, we can highlight the reason to choose $\psi_L$ as a new test function. Indeed, \rmk{invoking Proposition \ref{prop:renorm}} it is easy to show that for \rmk{this particular test function we have} 
\[ 
|(A^1_{s t} + A^2_{s t})_{} h_s (\psi_L)| \lesssim_V h_s(\psi_L) \omega_Z(s,t)^{1/p} 
\]
for some constant depending only on the vector fields $V$ but uniform in $L$. The other terms on the right hand side of~\eqref{eq:bnd-delta-h-2} are controls.
\rmk{Note that even though $\mu([s,t])$ is not superadditive due to the possible presence of jumps, $\mu$ is a positive measure anyway. Therefore one has the following property which can be used as a replacement for superadditivity}
$$\mu([s,u])+\mu([u,t])\leq 2\mu([s,t]).$$
Hence a simple modification of  Lemma \ref{lemma:basic-rough-gronwall} (to take into account this small deviation from superadditivity) gives readily
$$
\sup_{t\in[0,T]}h_t(\psi_L) \lesssim 1,
$$
where the constant is uniform in $L$ and depends only on $\omega_m,\omega_Z,V,\mu$ and $h_0(\psi_L)$. This implies that if $h_0 \in L^1$ we can send $L\to \infty$ and get that
\begin{equation}\label{eq:h-unf-L1}
\sup_{t\in[0,T]}h_t(\1) \lesssim 1,
\end{equation}
by monotone convergence. Summarizing, we have obtained that $h_t$ is in $L^1$ uniformly in $t\in[0,T]$. 

\noindent
\textit{Step3: Conclusion.}
Having relation \eqref{eq:h-unf-L1} in hand, we can go back to equation \eqref{eq:bnd-delta-h-2} for $h_t(\psi_L)$, and send $L\to \infty$ therein. We first resort to the fact that $\sup_{t\in[0,T]}h_t(\1)$ is bounded in order to get that:
\[ 
\lim_{L\to \infty}(A^1_{s t} + A^2_{s t}) h_s (\psi_L) 
\lesssim
(A^1_{s t} + A^2_{s t}) h_s (\1)
= 0 ,
\]
where the second identity is due to the fact that $\diver V=0$ (as noted in \eqref{divfree}). Thus, the limiting relation for $\delta h_{st}(\psi_{L})$ is:
\begin{equation}\label{eq:h1}
\delta h_{st}(\1) 
\lesssim 
\omega_{\natural}(s,t) ^{\frac{3}{q}},
\end{equation}
for all $s,t\in [0,T]$ at a sufficiently small distance from each other.  
Thus, one may telescope \eqref{eq:h1} on a partition $\{0=t_0<t_1<\cdots<t_n=t\}$ whose mesh vanishes with $n$. Invoking our expression~\eqref{eq:def-omega-natural} for $\omega_{\natural}$, this entails:
\begin{align*}
 h_t (\1)- h_0 (\1)& =\sum_{i=0}^{n-1} \delta h_{t_{i}t_{i+1}} (\1)\\
 &\lesssim \Big[\sup_{i=0,\dots,n-1}\omega_Z(t_i,t_{i+1})\Big]^{(\frac{3}{p}-2\kappa-1)\wedge\frac{1}{p}\wedge \kappa}\big(\omega_Z(0,t)+ \omega_m(0,t)+2\mu([0,t])\big)\ ,
\end{align*}
\rmk{due to $\frac{3}{p}-2\kappa-1 >0$.} Eventually we send $n\to\infty$ and use the fact that $\omega_{Z}$ is a regular control. This yields:
\begin{align*}
\|f^{1,+}_t\bar f^{2,+}_t\|_{L^1_{x,\xi}}=h_t(\1)\leq h_0(\1)=\|f^1_0\bar f^2_0\|_{L^1_{x,\xi}},
\end{align*}
which ends the proof.
\end{proof}

\rmk{Once endowed with the result of Proposition \ref{prop:compare}, we can use standard arguments on kinetic equations (such as those in \cite{debus}) to derive uniqueness of the solution to equation \eqref{eq1},} as well as the reduction of a generalized kinetic solution to a kinetic solution and the $L^1$-contraction property. This is the contents of the following corollary.

\begin{corollary}
Under the assumptions of Theorem \ref{thm:main}, uniqueness holds true for equation~\eqref{eq1}. Furthermore, Theorem \ref{thm:main} (ii) and (iii) are satisfied.
\end{corollary}

\begin{proof}
Let us start by the reduction part, that is Theorem \ref{thm:main} (ii).
Let then $f$ be a generalized kinetic solution to \eqref{eq1} with an initial condition at equilibrium: $f_0=\ind_{u_0>\xi}$ . Applying Proposition \ref{prop:compare} to $f^1=f^2=f$ leads to
\begin{align*}
\sup_{0\leq t\leq T}\|f^{+}_t\bar f^{+}_t\|_{L^1_{x,\xi}}\leq \|f_0\bar f_0\|_{L^1_{x,\xi}}=\|\ind_{u_0>\xi}\ind_{u_0\leq \xi}\|_{L^1_{x,\xi}}=0.
\end{align*}
Hence $f^+_t(1-f^+_t)=0$ for a.e. $(x,\xi)$. Now, the fact that $f^+_t$ is a kinetic function for all $t\in[0,T)$ gives the
conclusion: indeed, by Fubini's Theorem, for any $t \in [0, T )$, there is a set $B_t$ of
full measure in $\R^N$ such that, for all $x\in B_t$, $f^+_t(x,\xi)\in\{0,1\} $ for
a.e. $\xi\in\R$. Recall that $-\partial_\xi f^+_t(x, \cdot)$ is a probability measure on $\R$ hence,
necessarily, there exists $u^+: [0,T)\times\R^N\rightarrow\R$ measurable such that $f^+_t(x,\xi)=\ind_{u^+(t,x)>\xi}$ for a.e. $(x,\xi)$ and all $t\in[0,T)$. Moreover, according to \eqref{integrov}, it holds
\begin{equation}\label{a1}
\sup_{0\leq t\leq T}\int_{\R^N}|u^+(t,x)|\,\dif x=\sup_{0\leq t\leq T}\int_{\R^N}\int_\R|\xi|\,\dif\nu^+_{t,x}(\xi)\,\dif x<\infty.
\end{equation}
Thus $u^+$ is a kinetic solution and Theorem \ref{thm:main}(ii) follows.

In order to prove the $L^1$-contraction property (that is Theorem \ref{thm:main}(iii)), consider two kinetic solutions $u_{1},u_{2}$ of equation \eqref{eq1} with respective initial conditions $u_{1,0},u_{2,0}$. Then we have:
$$(u_1-u_2)^+=\int_\R\ind_{u_1>\xi}\overline{\ind_{u_2>\xi}}\,\dif \xi \ .$$
Let ${u}^+_1$ and ${u}^+_2$ denote the representatives of $u_{1}, u_{2}$ \rmk{as constructed above}. Then we apply Proposition \ref{prop:compare} and obtain
\begin{align*}
\big\|\big(u^+_1(t) -u^+_2(t)\big)^+\big\|_{L^1_x}&=\| f^{1,+}_t\overline{f}^{2,+}_t\|_{L^1_{x,\xi}}\leq\| f^{1,+}_{0}\overline{f}^{2,+}_{0}\|_{L^1_{x,\xi}}=\|(u^+_{1,0}-u^+_{2,0})^+\|_{L^1_x}
\end{align*}
which completes the proof of Theorem \ref{thm:main}(iii). Uniqueness is obtained in the same way, by considering two identical initial conditions.
\end{proof}

\section{Rough conservation laws III: A priori estimates}\label{sec:apriori}

In this section we will establish a priori $L^{q}$-estimates for kinetic solutions to \eqref{eq1}. We thus consider a kinetic solution $u$ to \eqref{eq1} and let $f_{t}(x,\xi)=\ind_{u_{t}(x)>\xi}$ be the corresponding kinetic function, \rmk{to which we can associate a Borel measure $m$ (along Definition \ref{genkinsol}(iii))}. Let us also introduce some useful notation for the remainder of the section.

\begin{notation}\label{not:apriori}
We denote by $\chi_{t}$ the function $\chi_{t}(x,\xi) = f _{t}(x,\xi)- \ind_{\xi<0}$. We also define the functions $\beta_{q},\gamma_q:\R^{N+1}\to\R$, where $q\geq 0$, as follows:
\begin{equation*}
\beta_{q+1}(x,\xi)=\xi|\xi|^{q}, 
\quad\text{and}\quad
\gamma_{q}(x,\xi)=\begin{cases} |\xi|^q & \text{if $q>0$,}\\ 1 & \text{if $q=0$.} \end{cases}
\end{equation*}
\end{notation}

The interest of the functions $\beta_{q},\gamma_q$ lies in the following elementary relations, which are labeled here for further use:
$$
\partial_\xi\beta_{q+1}=(q+1)\gamma_{q},\qquad \partial_\xi\gamma_{q+2}=(q+2)\beta_{q+1},
$$
and consequently for $q \geq 2$ we have
\begin{equation}\label{eq:moments-u}
\chi_t(\beta_{q-1}) =\frac{1}{q} \int_{\R^N} |u_{t}(x)|^{q}\, \dd x  ,
\quad\text{and}\quad
|\chi_t|(\1) =\chi_t(\mathrm{sgn}(\xi)) = \int_{\R^N} |u_{t}(x)|\,\dd x.
\end{equation}
With these preliminary notations in mind, our a priori estimate takes the following form.

\begin{theorem}
\label{th:apriori-est-existence}
Assume Hypothesis \ref{hyp:V} holds true. Then $u$ satisfies the following relation:
\begin{equation}\label{eq:L1}
\sup_{t\in[0,T]}\|u(t,\cdot)\|_{L^{1}} \le \|u(0,\cdot)\|_{L^{1}}
\end{equation}
and, for all $q\geq 2$,
\begin{equation}\label{eq:Lq}
\sup_{t\in[0,T]}\|u(t,\cdot)\|^{q}_{L^{q}} + (q-1)\delta m_{0T}(\gamma_{q-2}) \lesssim_{\bfA,q} \|u(0,\cdot)\|^{q}_{L^{q}}+ \|u(0,\cdot)\|^2_{L^{2}} + \|u(0,\cdot)\|_{L^{1}} . 
\end{equation}
\end{theorem}

\begin{remark}\label{rem:rig}
The above result gives a priori estimates for kinetic solutions that depend only on the rough regularity of the driver $\bfA$, and are therefore well-suited for the proof of existence in the next section. Note that in order to make all the arguments below entirely rigorous, it is necessary to either work at the level of a smooth approximation \rmk{(just as in the example treated in Section \ref{subsec:first-application})} or to introduce an additional cut-off of the employed test functions. Since we will only apply Theorem \ref{th:apriori-est-existence} to smooth approximations, we omit the technical details here. For classical solutions it is easy to prove $L^q$ bounds. These bounds will depend on the $C^1$ norm of the driver and so will not pass to the limit. But using the fact, proved in Lemma~\ref{lemma:classical-to-rough}, that  classical kinetic solutions are, in particular, rough kinetic solutions, we can justify the steps below and get the uniform estimates claimed in Theorem \ref{th:apriori-est-existence}. 
\end{remark}

\begin{proof}[Proof of Theorem \ref{th:apriori-est-existence}]
Due to relation \eqref{eq:moments-u}, our global strategy will be to test $u_{t}$ against the functions $\beta_{q}$ defined in Notation \ref{not:apriori}. We will split this procedure in several steps.

\noindent
\textit{Step 1: Equation governing $\chi$.}
Let $\chi$ be the function  introduced in Notation \ref{not:apriori}, and observe that $\delta \chi (\varphi)=\delta f (\varphi)$.  Furthermore we have (in the distributional sense) $\nabla\1_{\xi<0}=0$ whenever $\xi\ne 0$,  and we have assumed $V(x,0)=0$ in~\eqref{H}. Having in mind relation \eqref{d:urd} defining $A^{1}$ and $A^{2}$, this easily yields:
$$
\chi (A^{1, \ast} \varphi + A^{2, \ast}
   \varphi) = f (A^{1, \ast} \varphi + A^{2, \ast}
   \varphi).
$$
Then the function $\chi$ solves the rough equation
\begin{equation}
\label{eq:rough-chi}
 \delta \chi (\varphi) =\delta \partial_{\xi}  m (\varphi)+ \chi (A^{1, \ast} \varphi) + \chi (A^{2, \ast}
   \varphi)  + \chi^{\natural} (\varphi)
\end{equation}
where $\chi^{\natural} = f^\natural$.

\noindent
\textit{Step 2: Considerations on weights.}
Our aim is to apply equation \eqref{eq:rough-chi} to a test function of the 
from $\beta_{q-1}$ for some $q\geq 2$. The growth of the test function does not pose particular problems since we can use a scale of spaces  of test function with a polynomial weight $w_{q-1}=1+\gamma_{q-1}$. However, in order to obtain useful estimates, we cannot apply directly the Rough Gronwall strategy. Indeed, estimates for  $ \chi^{\natural} (\beta_{q-1})$ will in general depend on $m(w_{q-1})$ and on $|\chi| (w_{q-1})$, and we cannot easily control $m(w_{q-1})$. 

To avoid this problem we have to inspect more carefully the equation satisfied by $\chi^{\natural}(\beta_{q-1})$. 
Applying $\delta$ to \eqref{eq:rough-chi} with $\varphi(\xi)=\beta_{q-1}$
we obtain
\begin{equation}\label{eq:delta}
\delta\chi^\natural_{s u t}(\beta_{q-1})=(\delta\chi)_{su}(A^{2,*}_{ut}\beta_{q-1})+\chi^\sharp_{su}(A^{1,*}_{ut}\beta_{q-1})
\end{equation}
with the usual notation
$$\chi^\sharp_{su}=(\delta\chi)_{su}-A^1_{su}\chi_s.$$
Note that this point can be made rigorous as explained in Remark \ref{rem:rig}.
Moreover, using the \rmk{specific definition of $A^{1, \ast},A^{2, \ast}$, namely}
$$A^{1, \ast} \varphi= -
Z^1 V \cdot\nabla\varphi,\qquad A^{2, \ast} \varphi= Z^2 V \cdot\nabla (V \cdot\nabla\varphi),$$
we have that the
test functions on the right hand side of \eqref{eq:delta}, i.e. $A^{1,*}_{ut}\beta_{q-1}$ and $A^{2,*}_{ut}\beta_{q-1}$, as well as their derivatives are bounded by the weight $w_{q-2}$ and not just $w_{q-1}$ as we would naively expect. So we can use the scale of spaces with weight $w_{q-2}$ in order to estimate the remainder.

To this end, consider the family \rmk{$(E_n^{q})_{0\leq n\leq 3}$} of weighted spaces given by 
$$E_n^q:=\left\{\varphi:\R^{N+1}\to\R;\,\|\varphi\|_{E_n^q}:=\sum_{0\leq k\leq n}\left\|\frac{\nabla^k\varphi}{w_{q-2}}\right\|_{L^\infty_{x,\xi}}<\infty\right\}$$
\rmk{Since $w_{q-2}$ stands for a fixed weight (independent of $n$), it is easy to check that the basic convolution procedure (\ref{smoothing-convolution}) gives birth to a smoothing $(J^\eta)_{\eta\in(0,1)}$ with respect to this scale.}

\noindent
\textit{Step 3: Estimation of $\chi^{\natural}$ as a distribution.}
We are now in a position to see relation \eqref{eq:delta} as an equation of the form \eqref{eq:gen} on the scale \rmk{$(E_n^{q})_{0\leq n\leq 3}$}, and apply the general a priori estimate of Theorem \ref{theo:apriori} \rmk{(or more simply Corollary \ref{cor:apriori})} in this context. Indeed, if $\varphi\in E_1^q$ then it holds true that
\begin{equation*}
|\delta m_{st}(\partial_\xi\varphi)|\leq \delta m_{st}(w_{q-2})\|\varphi\|_{E_1^q}=\left(\delta m_{st}(\1)+\delta m_{st}(\gamma_{q-2})\right)\|\varphi\|_{E_1^q}.
\end{equation*}
Besides, if $\varphi\in E_{n+1}^q$, $n=0,1,2$, then
\begin{align*}
\|A^{k,*}_{st}\varphi\|_{E^q_n}&\lesssim_{\|V\|_{W^{n,\infty}}}\omega_Z(s,t)^\frac{k}{p}\|\varphi \|_{E_{n+k}^q}, \qquad k=1,2,
\end{align*}
which implies that $\bfA=(A^1,A^2)$ is a continuous unbounded $p$-rough driver on the scale \rmk{$(E_n^{q})_{0\leq n\leq 3}$}. \rmk{Hence, thanks to Corollary \ref{cor:apriori} (keep in mind that we implicitly consider smooth approximations of the noise $Z$ here), we get that for all $s<t$ sufficiently close to each other,}
\begin{equation}\label{apriori-bound1-lp} 
\begin{split}
\|\chi^\natural_{st}\|_{E^q_{-3}}&\lesssim \omega_\natural(s,t)^\frac{3}{p}\\
&\quad := \, \sup_{r\in[s,t]}|\chi_r |(\1)\,\omega_A(s,t)^{\frac{3}{p}}
+\left(\delta m_{st}(\1)+\delta m_{st}(\gamma_{q-2})\right) \omega_A(s,t)^{\frac{3-p}{p} }.
\end{split}
\end{equation}
In the latter relation, we still have to find an accurate bound for $|\chi |(\1)$ and $m(\gamma_{q-2})$.

\noindent
\textit{Step 4: Reduction to $L^{1}$ estimates.}
\rmk{Inserting the above smoothing $(J^\eta)_{\eta\in(0,1)}$} into \eqref{eq:delta}, we obtain 
\[ \delta \chi^{\natural} (\beta_{q-1}) = \delta \chi ((1 - J^{\eta}) A^{1, \ast} \beta_{q-1})
   + \delta \chi ((1 - J^{\eta}) A^{2, \ast} \beta_{q-1}) - \chi (A^{1, \ast} (1 -
   J^{\eta}) A^{1, \ast} \beta_{q-1}) \]
\[ + \chi (A^{2, \ast} J^{\eta} A^{1, \ast} \beta_{q-1}) + \chi (A^{1, \ast} J^{\eta}
   A^{2, \ast} \beta_{q-1}) + \chi (A^{2, \ast} J^{\eta} A^{2, \ast} \beta_{q-1}) \]
\[ + \chi^{\natural} (J^{\eta} A^{1, \ast} \beta_{q-1}) + \chi^{\natural} (J^{\eta}
   A^{2, \ast} \beta_{q-1}) \]
\[ -\delta m ( \partial_{\xi} J^{\eta} A^{1, \ast} \beta_{q-1}) -\delta m ( \partial_{\xi}J^{\eta}
   A^{2, \ast} \beta_{q-1}) \]  
As already explained above, the test functions on the right hand side always contain derivatives of $\beta_{q-1}$, so that the scale $(E^q_n)$ is sufficient to control the right hand side. Indeed, we may use~\eqref{apriori-bound1-lp} for the remainder as well as the elementary bound (observe that $|\chi|(\gamma_{q-1})=\chi(\beta_{q-1})$)
$$
|\chi(\varphi)|
\leq |\chi|(w_{q-2}) \|\varphi\|_{E^q_0}
\leq |\chi|(\1+ w_{q-1}) \|\varphi\|_{E^q_0}
= \left(2|\chi|(\1)+\chi(\beta_{q-1})\right)\|\varphi\|_{E^q_0}
$$
to deduce, along the same lines as in the proof of Theorem \ref{theo:apriori}, that
\begin{multline*}
|\delta\chi_{sut}^\natural(\beta_{q-1})|\lesssim  \Big(\sup_{r\in[s,t]}|\chi_r|(\1)+\sup_{r\in[s,t]}\chi_r(\beta_{q-1})\Big)\omega_Z(s,t)^{\frac{3}{p}}\\
+\omega_Z(I)^{\frac{1}{p}}\omega_\natural(s,t)^\frac{3}{p}
+\left(\delta m_{st}(\1)+\delta m_{st}(\gamma_{q-2})\right)\omega_A(s,t)^{\frac{3-p}{p}},
\end{multline*}
provided $s,u,t \in I$ with $\omega_Z(I)$ sufficiently small. We can now resort to the (original) sewing Lemma \ref{lemma-lambda}, which gives
\begin{align*}
|\chi_{st}^\natural(\beta_{q-1})|&
\lesssim  \Big(\sup_{r\in[s,t]}|\chi_r|(\1)+\sup_{r\in[s,t]}\chi_r(\beta_{q-1})\Big)\omega_Z(s,t)^{\frac{3}{p}}\\
&\quad+\left(\delta m_{st}(\1)+\delta m_{st}(\gamma_{q-2})\right)\omega_A(s,t)^{\frac{3-p}{p}}.\end{align*}
Finally, \eqref{eq:rough-chi} applied to $\beta_{q-1}$ reads as
$$
 \delta \chi (\beta_{q-1}) = \chi (A^{1, \ast} \beta_{q-1}) + \chi (A^{2, \ast}
   \beta_{q-1}) - (q-1) \delta m (\gamma_{q-2}) + \chi^{\natural} (\beta_{q-1})
$$
so that recalling relation \eqref{eq:moments-u} and applying the Rough Gronwall lemma yields, for any $q\geq 2$,
\begin{equation}\label{eq:apriori-Lq}
\sup_{t\in[0,T]} \chi_t (\beta_{q-1}) + (q-1) \delta m_{0 T} (\gamma_{q-2}) 
\lesssim  
\chi_0 (\beta_{q-1})  + \delta m_{0 T}(\1) + \sup_{t\in[0,T]}|\chi_t|(\1).
\end{equation}
In particular, for $q=2$ we obtain an estimate for $\delta m_{0 T}(\1)$ in terms of  $\sup_{t\in[0,T]}|\chi_t|(\1)$ and the initial condition only:
$$
 \sup_{t\in[0,T]} \chi_t (\beta_{1}) + \delta m_{0 T} (\1) \lesssim \chi_0 (\xi)+ \sup_{t\in[0,T]}|\chi_t|(\1). 
$$
Plugging this relation into \eqref{eq:apriori-Lq} and recalling relation \eqref{eq:moments-u}, we thus end up with:
\begin{equation}\label{eq:apriori-Lq-2}
\sup_{t\in[0,T]}\|u_{t}\|_{L^{q}}^{q} + (q-1) \delta m_{0 T} (\gamma_{q-2})
\lesssim
\chi_0 (\beta_{q-1}) + \sup_{t\in[0,T]}|\chi_t|(\1).
\end{equation}
This way, we have reduced the problem of obtaining a priori estimates in $L^{q}$ to estimates in $L^1$, and more specifically to an upper bound on $ \sup_{t\in[0,T]}|\chi_t|(\1)$.

\noindent
\textit{Step 5: $L^{1}$ estimates.}
The first obvious idea in order to estimate $|\chi|(\1)$ is to follow the computations of the previous step. However, this strategy requires to test the equation against the singular test function $(x,\xi) \mapsto \mathrm{sgn}(\xi)$. It might be possible to approximate this test function and then pass to the limit. In order to do so one would have to prove that the rough driver behaves well under this limit and that we have uniform estimates. 

Without embarking in this strategy, we shall first upper bound $u_{t}$ in $L^{1}$. Namely, observe that the $L^1$-contraction property established in Section \ref{sec:uniq} immediately implies the $L^1$-estimate we need.
Indeed we note that under hypothesis \eqref{H}, equation \eqref{eq1} with null initial condition possesses a kinetic solution which is constantly zero. Hence the $L^1$-contraction property applied to $u_1=u$ and $u_2\equiv 0$ yields \eqref{eq:L1}. Going back to relation \eqref{eq:moments-u}, this also implies:
\begin{equation*}
\sup_{t\in[0,T]}|\chi_t|(\1) \le \|u_{0}\|_{L^{1}},
\end{equation*}
which is the required bound for $|\chi_t|(\1)$ needed to close the $L^{q}$-estimate \eqref{eq:apriori-Lq-2}. Our claim \eqref{eq:Lq} thus follows.
\end{proof}

\section{Rough conservation laws IV: Existence}\label{sec:existence}

\rmk{Let us finally establish the existence part of Theorem \ref{thm:main}. To this end, we consider $(\mathbf{Z}^n)_{n\in\N},$ a family of canonical rough paths lifts associated with smooth paths $(z^n)$, which converge to $\mathbf Z$ in the uniform sense (over the time interval $[0,T]$), and such that
\begin{equation}\label{uniform-assumption-z-exi}
\sup_{n \geq 0} \Big\{ \big| Z^{1,n}_{st}\big|^p +\big| Z^{2,n}_{st}\big|^{\frac{p}{2}} \Big\} \lesssim \omega_{\mathbf{Z}}(s,t) \ ,
\end{equation}
for some proportional constant independent of $s,t\in [0,T]$, and where $\omega_Z$ is the regular control introduced in Hypothesis \ref{hyp:z}. Note that the existence of such an approximation $(\mathbf{Z}^n)_{n}$ is for instance guaranteed by the result of \cite[Proposition 8.12]{FV}. Then we define the approximate drivers $\bfA^n=(A^{n,1},A^{n,2})$ as follows}
\begin{align*}
\begin{aligned}
A^{n,1}_{st}\varphi:&\!=Z^{n,1,i}_{st} \, V^i\cdot\nabla_{\xi,x}\varphi\ ,\\
A^{n,2}_{st}\varphi:&\!=Z^{n,2,ij}_{st}\, V^j\cdot\nabla_{\xi,x}( V^i\cdot\nabla_{\xi,x}\varphi)\ .
\end{aligned}
\end{align*}
\rmk{It is readily checked that both $\bfA$ and $\bfA^n$ define continuous unbounded $p$-rough drivers (in the sense of Definition \ref{def:urd}) on the scale $(E_k)_{0\leq k\leq 3}$ given by $E_k=W^{k,1}(\R^{N+1})\cap W^{k,\infty}(\R^{N+1})$, and, according to (\ref{uniform-assumption-z-exi}), we can clearly pick the related controls $\omega_{A^n}$ of $\bfA^n$ in such a way that
\begin{equation}\label{uniform-assumption-a}
\sup_{n\geq 0}\, \omega_{A^n}(s,t)\lesssim \omega_Z(s,t) \ ,
\end{equation}
for some proportional constant independent of $s,t\in [0,T]$. We fix this scale $(E_k)_{0\leq k\leq 3}$, as well as these uniformly-bounded controls $\omega_{A^n}$, for the rest of proof.}

\

Using the standard  theory for conservation laws, one obtains existence of a unique kinetic solution $u^n$ to the approximate problem
\begin{equation*}
\dd u^n+\diver(A(x,u^n))\,\dd z^n=0,\qquad
u^n(0)=u_0.
\end{equation*}
moreover we denote by  $f^n=\ind_{u^n>\xi}$ the kinetic function associated to $u^n$ and by  $m^n$ the kinetic measure appearing in the kinetic formulation \eqref{eq:weak}.
We are now ready to prove the existence of a solution to equation \eqref{eq1}.

\begin{proof}[Proof of Theorem \ref{thm:main} (i)]

\noindent
\textit{Step 1: A priori bound for the regularized solutions.}
Due to Lemma~\ref{lemma:classical-to-rough}, the classical solutions $f^n$ corresponds to rough kinetic solutions $f^{n,\pm}$ satisfying
\begin{align}\label{eq:234}
\begin{aligned}
\delta f^{n,+}_{st}(\varphi)&=f^{n,+}_s(A^{n,1,*}_{st}\varphi)+f^{n,+}_s(A^{n,2,*}_{st}\varphi)+f^{n,+,\natural}_{st}(\varphi)-m^n(\ind_{(s,t]}\partial_\xi\varphi),\\
\delta f^{n,-}_{st}(\varphi)&=f^{n,-}_s(A^{n,1,*}_{st}\varphi)+f^{n,-}_s(A^{n,2,*}_{st}\varphi)+f^{n,-,\natural}_{st}(\varphi)-m^n(\ind_{[s,t)}\partial_\xi\varphi),
\end{aligned}
\end{align}
which holds true in the \rmk{above scale $(E_k)_{0\leq k\leq 3}$,} for some remainders $f^{n,\pm,\natural}$. 

Under our standing assumption on the initial condition, it follows from Theorem \ref{th:apriori-est-existence} that the approximate solutions $u^n$ are bounded uniformly in $L^\infty(0,T;L^1\cap L^2(\R^N))$ and the corresponding kinetic measures $m^n$ are uniformly bounded in total variation. 
Therefore the Young measures $\nu^n=\delta_{u^n=\xi}$ satisfy
\begin{equation}
\label{est:conv-unif-est}
\sup_{t\in[0,T]}\int_{\R^N}\int_\R |\xi|\,\nu^n_{t,x}(\dd\xi)+\sup_{t\in[0,T]}\int_{\R^N}\int_\R |\xi|^2\,\nu^n_{t,x}(\dd\xi)\lesssim \|u_0\|_{L^1}+\|u_0\|^2_{L^2}.
\end{equation}
Now we simply invoke \rmk{Corollary \ref{cor:apriori} and (\ref{uniform-assumption-a}) to obtain}, since $|f^{n,\pm}|\leq 1$,
\begin{align}\label{est:rem1}
\begin{aligned}
&\|f^{n,\pm,\natural}_{st}\|_{E_{-3}}\lesssim \omega_Z(s,t)^{\frac{3}{p}}
+m^n(\ind_{(s,t]}) \,\omega_Z(s,t)^{\frac{3-p}{p} },
\end{aligned}
\end{align}
provided $\omega_Z(s,t)\leq L$. Notice that this restriction on the distance of $s,t$ induces a covering $\{I_k; \, k\le M\}$ of the interval $[0,T]$, for a finite $M\in\N$. To be more specific, $I_k$ is just chosen so that:
$$ 
\sup_{s,t\in I_k}\omega_{Z}(s,t)\leq L\qquad \forall k.
$$
Thus relation \eqref{est:rem1} is satisfied on each interval $I_{k}$.

\noindent
\textit{Step 2: Limit in equation \eqref{eq:234}.} 
\rmk{By \eqref{est:conv-unif-est}} the assumptions of Lemma \ref{kinetcomp} are fulfilled and there exists a kinetic function $f$ on $[0,T]\times\R^N$ such that, along a subsequence,
\begin{equation}\label{eq:conv-fn}
f^n\overset{*}{\rightharpoonup} f \qquad\text{in}\qquad L^\infty([0,T]\times\R^{N+1}),
\end{equation}
and the associated Young measure $\nu$ satisfies
$$
\esssup_{t\in[0,T]}\int_{\R^N}\int_\R \big(|\xi|+|\xi|^2\big)\,\nu_{t,x}(\dd\xi)\lesssim \|u_0\|_{L^1}+\|u_0\|^2_{L^2}.
$$
Moreover by the Banach-Alaoglu theorem there exists a nonnegative bounded Borel measure $m$ on $[0,T]\times\R^{N+1}$ such that, along a subsequence,
\begin{equation} \label{eq:conv-mn}
m^n\overset{*}{\rightharpoonup} m\qquad\text{in}\qquad \mathcal{M}_b([0,T]\times\R^{N+1}).
\end{equation}
 Moreover using Lemma \ref{lemma:left-right} we have also the existence of the good representatives $f^\pm$ of $f$. In order to pass to the limit in the equation \eqref{eq:234} 
the main difficulty originates in the fact that the only available convergence of $f^{n,+}$ (as well as $f^{n,-}$ and $f^n$) is weak* in time. Consequently, we cannot pass to the limit pointwise for a fixed time $t$.
In order to overcome this issue, we observe that the first three terms on the right hand sides in \eqref{eq:234}, i.e. the approximation of the Riemann-Stieltjes integral, are continuous in $t$. The kinetic measure poses problems as it contains jumps, which are directly related to the possible noncontinuity of $f^{n}$. Therefore, let us define an auxiliary distribution $f^{n,\flat}$ by
$$
f^{n,\flat}_t(\varphi):=f^{n,+}_t(\varphi)+m^n(\ind_{[0,t]}\partial_\xi\varphi),
$$
and observe that due to \eqref{eq:234} it can also be written as
$$
 f^{n,\flat}_t(\varphi)=f^{n,-}_t(\varphi)+m^n(\ind_{[0,t)}\partial_\xi\varphi).
$$
Then we have
\begin{align}\label{eq:23b}
\delta f^{n,\flat}_{st}(\varphi)&=f^{n,\pm}_s(A^{n,1,*}_{st}\varphi)+f^{n,\pm}_s(A^{n,2,*}_{st}\varphi)+f^{n,\pm,\natural}_{st}(\varphi)
\end{align}
and due to \eqref{est:rem1}, satisfied on each $I_{k}$, this yields:
\begin{align}\label{equic}
|\delta f^{n,\flat}_{st}(\varphi)|\lesssim \left(\omega_Z(s,t)^{\frac{1}{p}}
+m(\1_{[0,T]}) \,\omega_Z(s,t)^{\frac{3-p}{p} }\right)\|\varphi\|_{E_3}\lesssim \omega_Z(s,t)^\frac{3-p}{p}\|\varphi\|_{E_3},
\end{align}
where the second inequality stems from \eqref{eq:conv-mn}.

Owing to the fact that $f^{n,\flat}$ is a path,
the local bound \eqref{equic} on each interval $I_{k}$ can be extended globally on $[0,T]$ by a simple  telescopic sum argument. In other words, $(f^{n,\flat}(\varphi))_{n\in\N}$ is equicontinuous and bounded in $ V^q_1([0,T];\R)$ for $q=\frac{p}{3-p}$. So as a corollary of the Arzel\`a-Ascoli theorem (cf. \cite[Proposition 5.28]{FV}), there exists a subsequence, possibly depending on $\varphi$, and an element $f^{\flat, \vp} \in  {V}^q_1([0,T];\R)$ such that
\begin{equation}\label{conv:flat}
f^{n,\flat}(\varphi)\to f^{\flat,\vp} \qquad\text{in}\qquad V^{q'}_1([0,T];\R)\qquad\forall q'>q.
\end{equation}
As the next step, we prove that the limit $f^{\flat,\vp}$ can be identified to be given by a true distribution $f^\flat_t$ as 
\begin{equation}\label{eq:flat}
f^{\flat,\vp}_t=f^{+}_t(\varphi)+m(\ind_{[0,t]}\partial_\xi\varphi)=f^{-}_t(\varphi)+m(\ind_{[0,t)}\partial_\xi\varphi) =: f^\flat_t(\vp).
\end{equation}
To this end, let us recall that given $r_1,r_2\geq1$ such that $\frac{1}{r_1}+\frac{1}{r_2}>1$, we can define the Young integral as a bilinear continuous mapping
$$ V^{r_1}_1([0,T];\R)\times  V^{r_2}_1([0,T];\R)\to  V^{r_2}_1([0,T];\R),\qquad (g,h)\mapsto\int_{0}^\cdot g \,\dd h,$$
see \cite[Theorem 6.8]{FV}.
Let $\psi\in C^\infty_c([0,T))$. Then it follows from the definition of $f^{n,\flat}$, the integration by parts formula for Young integrals and $f^{n,\flat}_{0}=f^{n,+}_{0}=f_0$ that
\begin{align*}
\int_0^T f^{n}_t(\varphi) \psi'_t\,\dd t- m^n(\psi \partial_\xi\varphi)+f_0(\varphi)\psi_0=-\int_0^T\psi_t\,\dd f^{n,\flat}_t(\vp).
\end{align*}
The convergences \eqref{eq:conv-fn} and \eqref{eq:conv-mn} allow now to pass to the limit on the left hand side, whereas by \eqref{conv:flat} we obtain the convergence of the Young integrals on the right hand side.

We obtain
\begin{align*}
\int_0^T f_t(\varphi) \psi'_t\,\dd t- m(\psi \partial_\xi\varphi)+f_0(\varphi)\psi_0=-\int_0^T\psi_t\,\dd f^{\flat,\vp}_t
\end{align*}
Now, in order to derive \eqref{eq:flat} we consider again the two sequences of test functions \eqref{eq:23a} and pass to the limit as $\varepsilon\to0$. Indeed, due to Lemma \ref{lemma:left-right} we get the convergence of the first term on the left hand side, the kinetic measure term converges due to dominated convergence theorem and for the right hand side we use the continuity of the Young integral. We deduce
\begin{align*}
-f^{+}_t(\varphi)-m(\ind_{[0,t]}\partial_\xi\varphi)+f_0(\varphi)&=-f^{\flat,\vp}_t+f^{\flat,\vp}_0,\\
-f^{-}_t(\varphi)-m(\ind_{[0,t)}\partial_\xi\varphi)+f_0(\varphi)&=-f^{\flat,\vp}_t+f^{\flat,\vp}_0,
\end{align*}
and \eqref{eq:flat} follows since for $t=0$ we have
$$f^{\flat,\vp}_0=\lim_{n\to\infty} f^{n,\flat}_0(\varphi)=\lim_{n\to\infty}  f^{n,\pm}_0(\varphi)=f_0(\varphi).$$
Now, it only remains to prove that $f^+_0=f_0$. The above formula at time $t=0$ rewrites as
$$
f^{+}_0(\varphi)-f_0(\varphi)=-m(\ind_{\{0\}}\partial_\xi\varphi).
$$
Hence the claim can be proved following the lines of \cite[Lemma 4.3]{H16} and we omit the details. For the sake of completeness, let us set $f^-_0:=f_0$ and $f^+_T:=f^-_T$.

Finally we have all in hand to complete the proof of convergence in \eqref{eq:234}. Fix $\varphi \in E_3$ and integrate~\eqref{eq:23b} over $s$ as follows
\begin{equation*}
\begin{split}
\frac{1}{\varepsilon} \int_s^{s+\varepsilon}  (\delta f^{n,\flat})_{rt}(\varphi)\,\dd r-\frac{1}{\varepsilon} \int_s^{s+\varepsilon}  f^{n,+}_r(A^{n,1,*}_{rt} \varphi +A^{n,2,*}_{rt} \varphi)\,\dd r&=\frac{1}{\varepsilon} \int_s^{s+\varepsilon}  f^{n,+,\natural}_{rt}(\varphi) \,\dd r ,\\
\frac{1}{\varepsilon} \int^s_{s-\varepsilon}  (\delta f^{n,\flat})_{rt}(\varphi)\,\dd r-\frac{1}{\varepsilon} \int^s_{s-\varepsilon}  f^{n,-}_r(A^{n,1,*}_{rt} \varphi +A^{n,2,*}_{rt} \varphi)\,\dd r&=\frac{1}{\varepsilon} \int^s_{s-\varepsilon}  f^{n,-,\natural}_{rt}(\varphi) \,\dd r .
\end{split}
\end{equation*}
On the left hand side we can successively take the limit as $n \to \infty$ and $\varepsilon \to 0$ (or rather for a suitable subsequence of $n$ and $\varepsilon$ depending possibly on $\varphi$ and $s$, to be more precise). This leads to the following assertion: for every $s<t \in [0,T]$, the quantities
\begin{equation}\label{defi-f-flat}
\begin{split}
f^{+,\natural}_{s t}(\varphi):&\!=\lim_{\varepsilon \to 0}\, \lim_{n \to \infty} \frac{1}{\varepsilon} \int_s^{s+\varepsilon}  f^{n,+,\natural}_{rt}(\varphi)\,\dd r\\
f^{-,\natural}_{s t}(\varphi):&\!=\lim_{\varepsilon \to 0}\, \lim_{n \to \infty} \frac{1}{\varepsilon} \int^s_{s-\varepsilon}  f^{n,-,\natural}_{rt}(\varphi)\,\dd r
\end{split}
\end{equation}
are well-defined, finite and satisfy
\begin{equation}\label{second-part}
(\delta f^{\flat})_{st}(\varphi)=f^\pm_s(A^{1,*}_{st} \varphi +A^{2,*}_{st} \varphi)+f^{\pm,\natural}_{st}(\varphi) . 
\end{equation}
Injecting (\ref{eq:flat}) into (\ref{second-part}) yields that for every $s<t\in [0,T]$,
\begin{equation*}
\begin{split}
\delta f^+_{st}(\varphi)&=f^+_s(A^{1,*}_{st} \varphi +A^{2,*}_{st} \varphi)-m(\1_{(s,t]}  \partial_\xi \varphi)+f^{+,\natural}_{st}(\varphi) ,\\
\delta f^-_{st}(\varphi)&=f^-_s(A^{1,*}_{st} \varphi +A^{2,*}_{st} \varphi)-m(\1_{[s,t)}  \partial_\xi \varphi)+f^{-,\natural}_{st}(\varphi) ,
\end{split}
\end{equation*}
and so it only remains to prove that the remainders $f^{\pm,\natural}$ defined by (\ref{defi-f-flat}) are sufficiently regular.
To this end, we first observe that
\begin{align}\label{control-mu-n}
\limsup_{n\to\infty}m^n(\1_{(s,t]})\leq m(\ind_{[s,t]}),\qquad \limsup_{n\to\infty}m^n(\1_{[s,t)})\leq m(\ind_{[s,t]})
\end{align}
holds true for every $s<t\in[0,T]$.
Indeed, the weak* convergence of $m^n$ to $m$, as described by \eqref{eq:conv-mn}, allows us to assert that for every $t$ in the (dense) subset $\mathfrak{C}_m$ of continuity points of the function $t\mapsto m(\1_{(0,t]})$, one has $m^n(\1_{(0,t]}) \to m(\1_{(0,t]})$ as $n\to \infty$. Consider now two sequences $s_k,t_k$ in $\mathfrak{C}_m$ such that $s_k$ strictly increases to $s$ and $t_k$ decreases to $t$, as $k\to \infty$.
Hence
it holds that
$$ \limsup_{n\to \infty} \, m^n(\ind_{(s,t]}) \leq \limsup_{n\to \infty} \, m^n(\1_{(s_k,t_k]})=m(\1_{(s_k,t_k]}) .$$
Taking the limsup over $k$ yields the first part of (\ref{control-mu-n}), the second part being similar. 
Next, we make use of (\ref{est:rem1}) and (\ref{control-mu-n}) to deduce for every $\varphi \in E_3$ and every $s<t \in I_k$,
$$|f^{\pm,\natural}_{st}(\varphi)| \lesssim  \lVert \varphi \rVert_{E_3} \left(\omega_Z(s,t)+m([s,t])^{\frac{p}{3}} \omega_Z(s,t)^{1-\frac{p}{3}} \right)^{\frac{3}{p}} .$$
We can conclude that $f^{\pm,\natural} \in {V}^{\frac{p}{3}}_{2,\text{loc}}([0,T];E_3^\ast)$, and finally the pair $(f,m)$ is indeed a generalized kinetic solution on the interval $[0,T]$.

\noindent
\textit{Step 3: Conclusion.}
The reduction Theorem \ref{thm:main}(ii) now applies and yields the existence of $u^+:[0,T)\times\R^N\to\R$ such that $\1_{u^+(t,x)>\xi}=f^+_{t}(x,\xi)$ for a.e. $(x,\xi)$ and every $t$. Besides, we deduce from \eqref{est:23e} that $u^+\in L^\infty(0,T;L^1\cap L^2(\R^N))$. Hence, the function $u^+$ is a representative of a class of equivalence $u$ which is a kinetic solution to \eqref{eq1}. In view of Remark \ref{u+}, this is the representative which then satisfies the $L^1$-contraction property for every $t\in[0,T]$ and not only almost everywhere. The proof of Theorem \ref{thm:main}(i) is now complete.
\end{proof}

\section*{Conflict of Interest}
 The authors declare that they have no conflict of interest.

\def\ocirc#1{\ifmmode\setbox0=\hbox{$#1$}\dimen0=\ht0 \advance\dimen0
  by1pt\rlap{\hbox to\wd0{\hss\raise\dimen0
  \hbox{\hskip.2em$\scriptscriptstyle\circ$}\hss}}#1\else {\accent"17 #1}\fi}


\begin{thebibliography}{10}

\bibitem{ambrosio} L. Ambrosio, N. Fusco, D. Pallara, Functions of bounded variation and free discontinuity problems, Oxford Mathematical Monographs, The Clarendon Press, Oxford University Press, New York, 2000.

\bibitem{BG}
I.~Bailleul and M.~Gubinelli.
\newblock Unbounded rough drivers.
\newblock Annales de la faculté des sciences de Toulouse Sér. 6, 26 no. 4 (2017), p. 795-830.

\bibitem{bauzet}
Caroline Bauzet, Guy Vallet, and Petra Wittbold.
\newblock The {C}auchy problem for conservation laws with a multiplicative
  stochastic perturbation.
\newblock {\em J. Hyperbolic Differ. Equ.}, 9(4):661--709, 2012.

\bibitem{BVW}
Caroline Bauzet, Guy Vallet, and Petra Wittbold.
\newblock A degenerate parabolic-hyperbolic {C}auchy problem with a stochastic
  force.
\newblock {\em J. Hyperbolic Differ. Equ.}, 12(3):501--533, 2015.

\bibitem{vov1}
F.~Berthelin and J.~Vovelle.
\newblock A {B}hatnagar-{G}ross-{K}rook approximation to scalar conservation
  laws with discontinuous flux.
\newblock {\em Proc. Roy. Soc. Edinburgh Sect. A}, 140(5):953--972, 2010.

\bibitem{car}
Jos{\'e} Carrillo.
\newblock Entropy solutions for nonlinear degenerate problems.
\newblock {\em Arch. Ration. Mech. Anal.}, 147(4):269--361, 1999.

\bibitem{MR2510132}
Michael Caruana and Peter Friz.
\newblock Partial differential equations driven by rough paths.
\newblock {\em J. Differential Equations}, 247(1):140--173, 2009.

\bibitem{MR2765508}
Michael Caruana, Peter~K. Friz, and Harald Oberhauser.
\newblock A (rough) pathwise approach to a class of non-linear stochastic
  partial differential equations.
\newblock {\em Ann. Inst. H. Poincar\'e Anal. Non Lin\'eaire}, 28(1):27--46,
  2011.

\bibitem{catellier_rough_2015}
R{\'e}mi Catellier.
\newblock Rough linear transport equation with an irregular drift.
\newblock {\em arXiv:1501.03000 [math]}, January 2015.
\newblock arXiv: 1501.03000.

\bibitem{karlsen}
Gui-Qiang Chen, Qian Ding, and Kenneth~H. Karlsen.
\newblock On nonlinear stochastic balance laws.
\newblock {\em Arch. Ration. Mech. Anal.}, 204(3):707--743, 2012.

\bibitem{chen}
Gui-Qiang Chen and Beno{\^{\i}}t Perthame.
\newblock Well-posedness for non-isotropic degenerate parabolic-hyperbolic
  equations.
\newblock {\em Ann. Inst. H. Poincar\'e Anal. Non Lin\'eaire}, 20(4):645--668,
  2003.
	
\bibitem{davie} A. M. Davie.
\newblock Differential Equations Driven by Rough Paths: An Approach via Discrete Approximation.
\newblock {\it Applied Mathematics Research Express. AMRX}, {\bf 2}, Art. ID abm009, 40, 2007.

\bibitem{degen2}
A.~{Debussche}, M.~{Hofmanov{\'a}}, and J.~{Vovelle}.
\newblock {Degenerate Parabolic Stochastic Partial Differential Equations:
  Quasilinear case}.
\newblock {\em Ann. Prob.}, 2016.
\newblock to appear.

\bibitem{debus}
A.~Debussche and J.~Vovelle.
\newblock Scalar conservation laws with stochastic forcing.
\newblock {\em Journal of Functional Analysis}, 259(4):1014--1042, 2010.

\bibitem{DV}
A.~Debussche and J.~Vovelle.
\newblock Scalar conservation laws with stochastic forcing, revised version.
\newblock unpublished,
  \url{http://math.univ-lyon1.fr/~vovelle/DebusscheVovelleRevised.pdf}, 2014.

\bibitem{debus2}
A.~Debussche and J.~Vovelle.
\newblock Invariant measure of scalar first-order conservation laws with
  stochastic forcing.
\newblock {\em Probab. Theory Related Fields}, 163(3-4):575--611, 2015.

\bibitem{MR2925571}
A.~Deya, M.~Gubinelli, and S.~Tindel.
\newblock Non-linear rough heat equations.
\newblock {\em Probab. Theory Related Fields}, 153(1-2):97--147, 2012.

\bibitem{14126557}
J.~{Diehl}, P.~K. {Friz}, and W.~{Stannat}.
\newblock {Stochastic partial differential equations: a rough path view}.
\newblock {\em ArXiv e-prints}, December 2014.

\bibitem{MR2978136}
Joscha Diehl and Peter Friz.
\newblock Backward stochastic differential equations with rough drivers.
\newblock {\em Ann. Probab.}, 40(4):1715--1758, 2012.

\bibitem{MR3332714}
Joscha Diehl, Peter~K. Friz, and Harald Oberhauser.
\newblock Regularity theory for rough partial differential equations and
  parabolic comparison revisited.
\newblock In {\em Stochastic analysis and applications 2014}, volume 100 of
  {\em Springer Proc. Math. Stat.}, pages 203--238. Springer, Cham, 2014.

\bibitem{MR1022305}
R.~J. DiPerna and P.-L. Lions.
\newblock Ordinary differential equations, transport theory and {S}obolev
  spaces.
\newblock {\em Invent. Math.}, 98(3):511--547, 1989.

\bibitem{Ev90} L. C. Evans, Weak convergence methods for nonlinear partial differential equations, No. 74. American Mathematical Soc., 1990.

\bibitem{feng}
Jin Feng and David Nualart.
\newblock Stochastic scalar conservation laws.
\newblock {\em J. Funct. Anal.}, 255(2):313--373, 2008.

\bibitem{160204746}
P.~K. {Friz}, P.~{Gassiat}, P.-L. {Lions}, and P.~E. {Souganidis}.
\newblock {Eikonal equations and pathwise solutions to fully non-linear SPDEs}.
\newblock {\em ArXiv e-prints}, February 2016.

\bibitem{MR3152786}
Peter Friz and Harald Oberhauser.
\newblock Rough path stability of (semi-)linear {SPDE}s.
\newblock {\em Probab. Theory Related Fields}, 158(1-2):401--434, 2014.

\bibitem{friz_stochastic_2014}
Peter~K. Friz and Benjamin Gess.
\newblock Stochastic scalar conservation laws driven by rough paths.
\newblock {\em arXiv:1403.6785 [math]}, March 2014.
\newblock arXiv: 1403.6785.

\bibitem{friz_course_2014}
Peter~K. Friz and Martin Hairer.
\newblock {\em A {Course} on {Rough} {Paths}: {With} an {Introduction} to
  {Regularity} {Structures}}.
\newblock Springer, August 2014.

\bibitem{FV}
Peter~K. Friz and Nicolas~B. Victoir.
\newblock {\em Multidimensional {Stochastic} {Processes} {As} {Rough} {Paths}:
  {Theory} and {Applications}}.
\newblock Cambridge University Press, February 2010.

\bibitem{MR3423246}
Mar{\'{\i}}a~J. Garrido-Atienza, Kening Lu, and Bj{\"o}rn Schmalfuss.
\newblock Local pathwise solutions to stochastic evolution equations driven by
  fractional {B}rownian motions with {H}urst parameters {$H\in(1/3,1/2]$}.
\newblock {\em Discrete Contin. Dyn. Syst. Ser. B}, 20(8):2553--2581, 2015.

\bibitem{gess_semi_discretization_2015}
Benjamin Gess, Beno{\^\i}t Perthame, and Panagiotis~E. Souganidis.
\newblock Semi-discretization for stochastic scalar conservation laws with
  multiple rough fluxes.
\newblock {\em arXiv:1512.06056 [math]}, December 2015.
\newblock arXiv: 1512.06056.

\bibitem{gess_long_time_2014}
Benjamin Gess and Panagiotis~E. Souganidis.
\newblock Long-time behavior, invariant measures and regularizing effects for
  stochastic scalar conservation laws.
\newblock {\em arXiv:1411.3939 [math]}, November 2014.
\newblock arXiv: 1411.3939.

\bibitem{MR3351442}
Benjamin Gess and Panagiotis~E. Souganidis.
\newblock Scalar conservation laws with multiple rough fluxes.
\newblock {\em Commun. Math. Sci.}, 13(6):1569--1597, 2015.

\bibitem{gubinelli_lectures_2015}
M.~Gubinelli and N.~Perkowski.
\newblock Lectures on singular stochastic {PDEs}.
\newblock {\em Ensaios Mathematicos}, 29:1--89, January 2015.
\newblock arXiv: 1502.00157.

\bibitem{gubinelli_controlled_2014}
M.~Gubinelli, S.~Tindel, and I.~Torrecilla.
\newblock Controlled viscosity solutions of fully nonlinear rough {PDEs}.
\newblock arXiv: 1403.2832, 2014.

\bibitem{MR3406823}
Massimiliano Gubinelli, Peter Imkeller, and Nicolas Perkowski.
\newblock Paracontrolled distributions and singular {PDE}s.
\newblock {\em Forum Math. Pi}, 3:e6, 75, 2015.

\bibitem{MR2599193}
Massimiliano Gubinelli and Samy Tindel.
\newblock Rough evolution equations.
\newblock {\em Ann. Probab.}, 38(1):1--75, 2010.

\bibitem{MR3274562}
M.~Hairer.
\newblock A theory of regularity structures.
\newblock {\em Invent. Math.}, 198(2):269--504, 2014.

\bibitem{MR3071506}
Martin Hairer.
\newblock Solving the {KPZ} equation.
\newblock {\em Ann. of Math. (2)}, 178(2):559--664, 2013.

\bibitem{MR3179667}
Martin Hairer, Jan Maas, and Hendrik Weber.
\newblock Approximating rough stochastic {PDE}s.
\newblock {\em Comm. Pure Appl. Math.}, 67(5):776--870, 2014.

\bibitem{MR3129808}
Martin Hairer and Hendrik Weber.
\newblock Erratum to: {R}ough {B}urgers-like equations with multiplicative
  noise [mr3010394].
\newblock {\em Probab. Theory Related Fields}, 157(3-4):1011--1013, 2013.

\bibitem{MR3010394}
Martin Hairer and Hendrik Weber.
\newblock Rough {B}urgers-like equations with multiplicative noise.
\newblock {\em Probab. Theory Related Fields}, 155(1-2):71--126, 2013.

\bibitem{hof}
Martina Hofmanov{\'a}.
\newblock Degenerate parabolic stochastic partial differential equations.
\newblock {\em Stochastic Process. Appl.}, 123(12):4294--4336, 2013.

\bibitem{bgk}
Martina Hofmanov{\'a}.
\newblock A {B}hatnagar-{G}ross-{K}rook approximation to stochastic scalar
  conservation laws.
\newblock {\em Ann. Inst. Henri Poincar\'e Probab. Stat.}, 51(4):1500--1528,
  2015.

\bibitem{H16}
Martina Hofmanova.
\newblock Scalar conservation laws with rough flux and stochastic forcing.
\newblock {\em Stoch. PDE: Anal. Comp.}, 2016.
\newblock to appear.

\bibitem{holden}
H.~Holden and N.~H. Risebro.
\newblock Conservation laws with a random source.
\newblock {\em Appl. Math. Optim.}, 36(2):229--241, 1997.

\bibitem{hu-nualart-2007}
Y. Hu and D. Nualart.
\newblock  Differential equations driven by Holder continuous functions of order greater than 1/2.
\newblock  In: Stochastic Analysis and Applications, 399-413,  Abel Symp., 2, Springer, Berlin, 2007.

\bibitem{MR2471936}
Yaozhong Hu and David Nualart.
\newblock Rough path analysis via fractional calculus.
\newblock {\em Trans. Amer. Math. Soc.}, 361(5):2689--2718, 2009.

\bibitem{vov}
C.~Imbert and J.~Vovelle.
\newblock A kinetic formulation for multidimensional scalar conservation laws
  with boundary conditions and applications.
\newblock {\em SIAM J. Math. Anal.}, 36(1):214--232 (electronic), 2004.

\bibitem{kim}
Jong~Uhn Kim.
\newblock On a stochastic scalar conservation law.
\newblock {\em Indiana Univ. Math. J.}, 52(1):227--256, 2003.

\bibitem{kruzk}
S.~N. Kru{\v{z}}kov.
\newblock First order quasilinear equations with several independent variables.
\newblock {\em Mat. Sb. (N.S.)}, 81 (123):228--255, 1970.

\bibitem{lions}
P.-L. Lions, B.~Perthame, and E.~Tadmor.
\newblock A kinetic formulation of multidimensional scalar conservation laws
  and related equations.
\newblock {\em J. Amer. Math. Soc.}, 7(1):169--191, 1994.

\bibitem{MR3327520}
Pierre-Louis Lions, Beno{\^{\i}}t Perthame, and Panagiotis~E. Souganidis.
\newblock Scalar conservation laws with rough (stochastic) fluxes.
\newblock {\em Stoch. Partial Differ. Equ. Anal. Comput.}, 1(4):664--686, 2013.

\bibitem{MR3380984}
Pierre-Louis Lions, Beno{\^{\i}}t Perthame, and Panagiotis~E. Souganidis.
\newblock Stochastic averaging lemmas for kinetic equations.
\newblock In {\em S\'eminaire {L}aurent {S}chwartz---\'{E}quations aux
  d\'eriv\'ees partielles et applications. {A}nn\'ee 2011--2012}, S\'emin.
  \'Equ. D\'eriv. Partielles, pages Exp. No. XXVI, 17. \'Ecole Polytech.,
  Palaiseau, 2013.

\bibitem{MR3274890}
Pierre-Louis Lions, Beno{\^{\i}}t Perthame, and Panagiotis~E. Souganidis.
\newblock Scalar conservation laws with rough (stochastic) fluxes: the
  spatially dependent case.
\newblock {\em Stoch. Partial Differ. Equ. Anal. Comput.}, 2(4):517--538, 2014.

\bibitem{lpt1}
Pierre-Louis Lions, Beno{\^{\i}}t Perthame, and Eitan Tadmor.
\newblock Formulation cin\'etique des lois de conservation scalaires
  multidimensionnelles.
\newblock {\em C. R. Acad. Sci. Paris S\'er. I Math.}, 312(1):97--102, 1991.

\bibitem{MR1654527}
Terry~J. Lyons.
\newblock Differential equations driven by rough signals.
\newblock {\em Rev. Mat. Iberoamericana}, 14(2):215--310, 1998.

\bibitem{MR2314753}
Terry~J. Lyons, Michael Caruana, and Thierry L{\'e}vy.
\newblock {\em Differential equations driven by rough paths}, volume 1908 of
  {\em Lecture Notes in Mathematics}.
\newblock Springer, Berlin, 2007.
\newblock Lectures from the 34th Summer School on Probability Theory held in
  Saint-Flour, July 6--24, 2004, With an introduction concerning the Summer
  School by Jean Picard.

\bibitem{malek}
J.~M{\'a}lek, J.~Ne{\v{c}}as, M.~Rokyta, and M.~R{\ocirc{u}}{\v{z}}i{\v{c}}ka.
\newblock {\em Weak and measure-valued solutions to evolutionary {PDE}s},
  volume~13 of {\em Applied Mathematics and Mathematical Computation}.
\newblock Chapman \& Hall, London, 1996.

\bibitem{nunu-vuillermot}
David Nualart and Pierre-A. Vuillermot.
\newblock Variational solutions for partial differential equations driven by a
  fractional noise.
\newblock {\em Journal of Functional Analysis}, 232(2):390--454, March 2006.

\bibitem{perth}
Beno{\^{\i}}t Perthame.
\newblock {\em Kinetic formulation of conservation laws}, volume~21 of {\em
  Oxford Lecture Series in Mathematics and its Applications}.
\newblock Oxford University Press, Oxford, 2002.

\bibitem{tadmor}
Beno{\^{\i}}t Perthame and Eitan Tadmor.
\newblock A kinetic equation with kinetic entropy functions for scalar
  conservation laws.
\newblock {\em Comm. Math. Phys.}, 136(3):501--517, 1991.

\bibitem{stoica}
Bruno Saussereau and Ion~Lucretiu Stoica.
\newblock Scalar conservation laws with fractional stochastic forcing:
  existence, uniqueness and invariant measure.
\newblock {\em Stochastic Process. Appl.}, 122(4):1456--1486, 2012.

\bibitem{MR2836540}
Josef Teichmann.
\newblock Another approach to some rough and stochastic partial differential
  equations.
\newblock {\em Stoch. Dyn.}, 11(2-3):535--550, 2011.

\bibitem{wittbold}
Guy Vallet and Petra Wittbold.
\newblock On a stochastic first-order hyperbolic equation in a bounded domain.
\newblock {\em Infin. Dimens. Anal. Quantum Probab. Relat. Top.},
  12(4):613--651, 2009.

\end{thebibliography}
\end{document}